\documentclass{amsart}

\usepackage{thesis} 

\begin{document}

\title{\MakeUppercase{\sspacestitle}}
\date{\today}
\author{\textsc{Aaron Mazel-Gee}}

\begin{abstract}
Both simplicial sets and simplicial spaces are used pervasively in homotopy theory as presentations of spaces, where in both cases we extract the ``underlying space'' by taking geometric realization.  We have a good handle on the category of simplicial sets in this capacity; this is due to the existence of a suitable model structure thereon, which is particularly convenient to work with since it enjoys the technical properties of being \textit{proper} and of being \textit{cofibrantly generated}.  This paper is devoted to showing that, if one is willing to work $\infty$-categorically, then one can manipulate simplicial spaces exactly as one manipulates simplicial sets.  Precisely, this takes the form of a proper, cofibrantly generated model structure on the \textit{$\infty$-category} of simplicial spaces, the definition of which we also introduce here.
\end{abstract}

\maketitle

\papernum{1}

\setcounter{tocdepth}{1}
\tableofcontents

\setcounter{section}{-1}

\section{Introduction}

\subsection{Simplicial spaces and model $\infty$-categories}\label{subsection sspaces and model infty-cats}

A simplicial space can be thought of as a \textit{resolution} of a space, namely its homotopy colimit; in nice cases this is computed by its geometric realization, and we will henceforth blur the distinction.  The operation of taking a homotopy colimit, being homotopy invariant, descends to a functor $|{-}| : s\S \ra \S$ of \textit{$\infty$-categories}, from that of simplicial spaces to that of spaces.  The primary purpose of the present paper is to introduce a new framework for studying this functor.  In particular, we give \textit{$\infty$-categorical} criteria for determining
\begin{enumeratesmallsub}
\item\label{criterion for KQ weak equivalence} when a map of simplicial spaces becomes an equivalence upon geometric realization, and
\item\label{criterion for ho-p.b.} when a homotopy pullback of simplicial spaces remains a homotopy pullback upon geometric realization,
\end{enumeratesmallsub}
which we have found to be much easier to verify than their existing 1-categorical counterparts.  We hope that this encourages homotopy theorists grappling with simplicial spaces to work $\infty$-categorically: even if a map of simplicial spaces or a homotopy pullback of simplicial spaces began its life 1-categorically, these questions are homotopy invariant and hence inherently $\infty$-categorical, and thus should be approachable using the framework given here.

To set the stage, let us recall Quillen's theory of \bit{model categories}: given a category $\C$ equipped with a subcategory $\bW \subset \C$, a \textit{model structure} on the category $\C$ consists of additional data which provide an efficient and computable method of accessing its localization $\C[\bW^{-1}]$.  As a prime example, the \bit{Kan--Quillen model structure} on the category $s\Set$ of simplicial sets provides a combinatorial framework for studying the homotopy category $s\Set[\bW^{-1}] \simeq \Top[\bW^{-1}]$ of topological spaces.

Now, suppose that $\C$ is not merely a category but an \textit{$\infty$-category}, again equipped with a subcategory $\bW \subset \C$.  We can analogously localize the $\infty$-category $\C$ at the subcategory $\bW$, forming a new $\infty$-category $\loc{\C}{\bW}$ equipped with a functor $\C \ra \loc{\C}{\bW}$, which is initial among those functors from $\C$ which invert all the maps in $\bW$.\footnote{Note that even if $\C$ is just a 1-category, this will generally differ from the more crude and lossy 1-categorical localization $\C[\bW^{-1}]$: the canonical map $\loc{\C}{\bW} \ra \C[\bW^{-1}]$ is precisely the projection to the homotopy category $\ho(\loc{\C}{\bW}) \simeq \C[\bW^{-1}]$.}  A key example for us is $\loc{s\S}{\bWgr} \simeq \S$, where we denote by $\bWgr \subset s\S$ the subcategory spanned by those maps which become invertible upon geometric realization.\footnote{For an explanation of this equivalence, see item \ref{free vs left vs right localizations} of \cref{section conventions}.}

With this background in place, we can now explain the two goals of this paper.

\begin{enumerate}

\item Inspired by Quillen, we introduce the notion of a \bit{model $\infty$-category}: given an $\infty$-category $\C$ equipped with a subcategory $\bW \subset \C$, a \textit{model structure} on the $\infty$-category $\C$ consists of additional data which provide an efficient and computable method of accessing its localization $\loc{\C}{\bW}$.  This provides a theory of \textit{resolutions} which is native to the $\infty$-categorical context.  As indicated by the title of this paper, the theory of model $\infty$-categories is developed in subsequent work.

\item In precise analogy with the Kan--Quillen model structure on the category $s\Set$ of simplicial sets, we endow the $\infty$-category $s\S$ of simplicial spaces with a \bit{Kan--Quillen model structure}, whose subcategory of weak equivalences is exactly $\bWgr \subset s\S$, which allows us to access its $\infty$-categorical localization $\loc{s\S}{\bWgr} \simeq \S$.  This model structure is likewise \textit{proper}, and is \textit{cofibrantly generated} by the sets $I_\KQ , J_\KQ \subset s\Set$ of boundary inclusions $\partial \Delta^n \subset \Delta^n$ and of horn inclusions $\Lambda^n_i \subset \Delta^n$, considered as maps of simplicial spaces via the inclusion $s\Set \subset s\S$ of simplicial sets as the discrete simplicial spaces.

\end{enumerate}

By way of illustrating some typical applications of the Kan--Quillen model structure, we now return to the criteria promised above.

\begin{enumeratesub}

\item As the Kan--Quillen model structure on $s\S$ is cofibrantly generated, a map is an \textit{acyclic fibration} if it has the right lifting property against the set $I_\KQ$ of generating cofibrations.  In particular, $\rlp(I_\KQ) \subset \bWgr$.

\item As the Kan--Quillen model structure on $s\S$ is right proper, then a pullback in which at least one of the two maps is a fibration (i.e.\! has $\rlp(J_\KQ)$) is a \textit{homotopy pullback}: that is, it remains a pullback under geometric realization.

\end{enumeratesub}

However, far from providing just a few assorted tricks, the Kan--Quillen model structure on $s\S$ gives a robust and extensive framework for understanding simplicial spaces vis-\`{a}-vis their geometric realizations, which is further augmented by the general theory of model $\infty$-categories.  For instance, the theory of \textit{homotopy co/limits} in model $\infty$-categories furnishes criteria for comparing the co/limit of an arbitrary diagram $\I \ra s\S$ of simplicial spaces with the co/limit of the resulting diagram $\I \ra s\S \ra \S$ obtained by componentwise geometric realization.

\subsection{Relation to existing literature}

Simplicial spaces and their geometric realizations play a key role in a number of different areas of topology.  Their study appears to have been initiated in loopspace theory (see \cite{SegalCSSS}, \cite{SegalConfLoop}, \cite{SegalCatCohThy}, and \cite{MayGILS}).  Relatedly, they find much use in algebraic K-theory (see e.g.\! \cite{Wald1126} or \cite{WeibelK}).\footnote{Many deep results in algebraic K-theory rely crucially on commuting certain pullbacks of simplicial spaces with geometric realization.  This commutation is generally ``not purely formal'', but is instead based in delicate simplicial manipulations.  We are hopeful that our Kan--Quillen model structure will prove useful in enabling (or at least streamlining) such arguments.}  More recently, some of the results in the present paper provide key input to applications in algebraic L-theory (see \cite{Luriesspaces}) and in derived differential geometry (see \cite{BEBdBP}).

In addition to the references already mentioned, hints of the Kan--Quillen model structure abound in the literature.  For instance, there is already explicit mention of the ``two homotopy theories'' on the category of simplicial topological spaces -- the ``topological'' one corresponding to the interval object $\const([0,1])$, and the ``algebraic'' one corresponding to the interval object $\Delta^1$ -- as early as \cite{MayGILS}.  Moreover, various scattered results bear anywhere between passing and striking resemblances to aspects of the Kan--Quillen model structure.  To illustrate this, we compare
\begin{itemizesmall}
\item our criterion \ref{criterion for KQ weak equivalence} with analogous criteria arising
\begin{itemizesmall}
\item from the ``Moerdijk model structure'' (of \cite{Moer}) in \cref{compare with moerdijk}, and
\item from the Cegarra--Remedios ``$\ol{W}$ model structure'' (of \cite{CegRem}) in \cref{rem ceg-rem};
\end{itemizesmall}
\item our criterion \ref{criterion for ho-p.b.} with analogous criteria arising
\begin{itemizesmall}
\item from Bousfield--Friedlander's ``$\pi_*$-Kan condition'' (of \cite{BF}) and
\item from Anderson's notion of a simplicial groupoid being ``fully fibrant'' (of \cite{AndFibGeomReal})
\end{itemizesmall}
in \cref{compare with anderson and bousfield--friedlander}; and
\item our notion of a fibration
\begin{itemizesmall}
\item with Rezk's ``realization fibrations'' -- which sorts of maps have also at times been called ``sharp maps'', ``fibrillations'', ``right proper maps'', ``h-fibrations'', and ``Grothendieck W-fibrations'', and which are closely related to the classical notion of ``quasifibrations'' -- (of \cite{Rezk-pi-star-kan}) in \cref{compare with rezk's realization fibrations}, and
\item with both of Seymour's and Brown--Szczarba's related but distinct ``continuous'' notions of Kan fibrations (of \cite{Seymour} and \cite{BrownSzczarba} resp.)\! in \cref{compare with seymour and BS}.
\end{itemizesmall}
\end{itemizesmall}

\begin{rem}\label{model infty-cats are better than double model cats}
These sundry results make abundantly clear the vast superiority of the $\infty$-categorical perspective for the study of simplicial spaces and their geometric realizations.  The definition of the Kan--Quillen model structure on the $\infty$-category of simplicial spaces is both straightforward and economical.  By contrast, to describe it in purely model-categorical terms would require some sort of notion of a ``double model structure'', in which an ``external'' model structure would present the underlying $\infty$-category, and then the actual model structure of interest thereon would be presented by an ``internal'' model structure (whose distinguished classes of maps would then have to be invariant under the homotopy relation generated by ``external'' model structure; whose lifting axioms would need to be formulated in a homotopy-coherent way relative to the ``external'' model structure; etc.).  It seems unlikely that the full extent of the Kan--Quillen model structure -- or for that matter, anything comparable to the general theory of model $\infty$-categories -- would ever have been discovered in a purely model-categorical context.
\end{rem}

\subsection{Goerss--Hopkins obstruction theory for $\infty$-categories}\label{subsection GHOsT motivation}

As the present paper is the first in its series, we spend a moment describing the original motivation for the theory of model $\infty$-categories.

The overarching goal of this project is to generalize \bit{Goerss--Hopkins obstruction theory} (\cite{GH-moduli-spaces, GH-moduli-problems}), a powerful tool for obtaining existence and uniqueness results for $\bbE_\infty$ ring spectra via purely algebraic computations, to the equivariant and motivic settings.  However, the original obstruction theory is based in a model category of spectra satisfying a number of technical conditions, making it relatively difficult to generalize directly.  Relatedly, the foundations for its construction rely on various point-set considerations, which appear for the sake of simplification but play no real mathematical role.\footnote{In fact, to those unfamiliar with the more nuanced techniques of model categories, these considerations might even appear to amount to something like black magic.}  Thus, as the obstruction theory ultimately lives on the \textit{underlying $\infty$-category} of spectra anyways, we instead aim to generalize Goerss--Hopkins obstruction theory to an arbitrary (presentably symmetric monoidal) $\infty$-category $\C$; this will yield equivariant and motivic obstruction theories simply by specializing to the $\infty$-categories of equivariant and motivic spectra.

Now, Goerss--Hopkins obstruction theory is constructed in the \textit{resolution model structure} (a/k/a the ``$\Etwo$ model structure'') on simplicial spectra, originally introduced in \cite{DKS-E2}, which provides a theory of \textit{nonabelian projective resolutions}.\footnote{More precisely, it is actually constructed in a certain model category which is monadic over simplicial spectra, whose model structure is lifted along the defining adjunction.}  Correspondingly, suppose we are given a presentable $\infty$-category $\C$, along with a set $\G$ of generators which we assume (without real loss of generality) to be closed under finite coproducts.  Then, Goerss--Hopkins obstruction theory for $\C$ will take place in the \bit{nonabelian derived $\infty$-category} of $\C$, i.e.\! the $\infty$-category $\PS(\G) = \Fun_\Sigma(\G^{op},\S)$ of those presheaves on $\G$ that take finite coproducts in $\G$ to finite products in $\S$, originally introduced in \cite[\sec 5.5.8]{LurieHTT}.  (If $\C$ is the underlying $\infty$-category of an appropriately chosen model category, then $\PS(\G)$ will be the underlying $\infty$-category of the resolution model structure on the category of simplicial objects therein.)  This admits a natural functor
\[ s\C \ra \PS(\G) , \]
given by taking a simplicial object $Y_\bullet \in s\C$ to the functor
\[ S^\beta \mapsto | \hom^\lw_\C(S^\beta,Y_\bullet)| \]
(for any generator $S^\beta \in \G$, and writing ``$\lw$'' to denote ``levelwise'').  In fact, this functor is a localization: denoting by $\bW_\res \subset s\C$ the subcategory spanned by those maps which it inverts, it induces an equivalence
\[ \loc{s\C}{\bW_\res} \xra{\sim} \PS(\G) . \]

From here, we can explain our motivation for the theory of model $\infty$-categories: the definition of the $\infty$-category $\PS(\G)$ is extremely efficient, but its abstract universal characterizations alone are insufficient for making the actual computations \textit{within} this $\infty$-category that are necessary to set up the obstruction theory.  Rather, computations therein generally rely on choosing \textit{simplicial resolutions} of objects, i.e.\! preimages under the functor $s\C \ra \PS(\G)$, and then working in $s\C$ to deduce results back down in $\PS(\G)$.\footnote{An important example is the ``spiral exact sequence'', which is a key ingredient to setting up the obstruction theory (see e.g.\! \cite[Lemma 3.1.2(2)]{GH-moduli-problems}).}  Thus, in order to organize these computations, we will provide a \bit{resolution model structure} on the $\infty$-category $s\C$, giving an efficient and computable method of accessing the localization $\loc{s\C}{\bW_\res} \simeq \PS(\G)$.

\begin{rem}
In a sense, the $\infty$-categorical Goerss--Hopkins obstruction theory is actually \textit{easier} to set up than its classical counterpart.  As described above, in the latter, one must carefully choose an appropriate ``ground floor'' model category of spectra on which to build the relevant resolution model structure (which depends nontrivially on the ground floor model structure).  By contrast, in the former, the relevant resolution model structure on $s\C$ is built on the \textit{trivial} model structure on the $\infty$-category $\C$, in which every map is both a cofibration and a fibration and the weak equivalences are precisely the equivalences (see \cref{trivial model structure}).

We view this state of affairs as falling squarely in line with the core philosophy of $\infty$-categories, namely that they are meant to isolate and dispense with exactly those manipulations which ought to be purely formal.  The usage of $\infty$-categories (and model structures thereon) allows us to sidestep the point-set technicalities which were ultimately of no homotopical interest in the first place, and hence to address only the truly interesting aspect of the story: nonabelian projective resolutions.
\end{rem}

\begin{rem}\label{rem motivic ghost for motivic morava E-theories}
As a sample application of the $\infty$-categorical Goerss--Hopkins obstruction theory, we prove the following result in forthcoming joint work with David Gepner.  As background, recall that the first application of Goerss--Hopkins obstruction theory was to prove that the Morava $E$-theory spectra admit essentially unique $\bbE_\infty$ structures, and moreover that their spaces of $\bbE_\infty$ automorphisms are essentially discrete and are given by the corresponding Morava stabilizer groups (see \cite[\sec 7]{GH-moduli-spaces}).  Bootstrapping up their arguments, we use Goerss--Hopkins obstruction theory in the $\infty$-category of motivic spectra to prove
\begin{itemizesmall}
\item that the \textit{motivic} Morava $E$-theory spectra again admit essentially unique $\bbE_\infty$ structures,
\item that again their spaces of $\bbE_\infty$ automorphisms are essentially discrete, but
\item that they can admit ``exotic'' $\bbE_\infty$ automorphisms not seen in ordinary topology.
\end{itemizesmall}
(More precisely, their groups of $\bbE_\infty$ automorphisms will generally contain the corresponding Morava stabilizer groups as \textit{proper} subgroups.)
\end{rem}

\subsection{Conventions}

\refMIC

\tableofcodenames

\examplecodename

\citelurie \ \luriecodenames

\butinvt \ \seeappendix

\subsection{Outline}

We now provide a more detailed outline of the contents of this paper.

\begin{itemize}

\item In \cref{section model infty-cats defns}, we define \textit{model $\infty$-categories}, and we define the notions of \textit{Quillen adjunctions} and \textit{Quillen equivalences} between them.

\item In \cref{section model infty-cats exs}, we provide a host of examples of the objects introduced in \cref{section model infty-cats defns}.  We also speculate on the existence of other examples -- some of a foundational nature, some which would provide yet more models for the $\infty$-category of $\infty$-categories, and one related to $\bbE_n$ deformation theory -- whose verifications lie beyond the scope of the current project.

\item In \cref{section cofgen model infty-cats}, we define \textit{cofibrantly generated model $\infty$-categories} and provide recognition and lifting theorems analogous to the classical ones.

\item In \cref{section define kan--quillen model structure}, we assert the existence of the \textit{Kan--Quillen model structure} on the $\infty$-category $s\S$ of simplicial spaces.  The proof (which we only outline, leaving the real substance for \cref{section proof of kan--quillen model structure}) relies on the recognition theorem of \cref{section cofgen model infty-cats}.  We also define what it means for a model $\infty$-category to be \textit{proper}.

\item In \cref{section auxiliary results}, we collect some auxiliary results regarding spaces and simplicial spaces.  In particular, we state a particular result -- \cref{crazy lemma} -- which ultimately represents the key piece of not-totally-formal input that makes the entire story tick (but we defer its proof to \cref{section proof of crazy lemma}).

\item In \cref{section fibrations}, we prove some convenient properties enjoyed by the fibrant objects and the fibrations in the Kan--Quillen model structure on $s\S$, and we define a ``fibrant replacement'' endofunctor $\Ex^\infty$ analogous to the classical one.

\item In \cref{section proof of kan--quillen model structure}, we prove the main result: that the data described in \cref{section define kan--quillen model structure} do indeed define a proper, cofibrantly generated model structure on the $\infty$-category $s\S$ of simplicial spaces.

\item In \cref{section proof of crazy lemma}, we prove \cref{crazy lemma}, using the classical theory of model categories and ultimately some rather delicate arguments regarding bisimplicial sets.

\item In \cref{section conventions}, we carefully lay out the notation, terminology, and conventions that we will adopt in this sequence of papers.

\item In \cref{section notation index}, we provide an index of the notation used in this sequence of papers.

\end{itemize}

\subsection{Acknowledgments}

This paper has benefitted substantially from conversations with Clark Barwick, David Gepner, Saul Glasman, Zhen Lin Low, and Justin Noel; we warmly thank them for their generosity, both with their ideas and their time.  We are additionally grateful to Daniel Br\"{u}gmann, Rune Haugseng, Eric Peterson, and Adeel Khan Yusufzai for giving helpful feedback on a draft of this paper, to David Ayala for suggesting some of the proposed model structures of \cref{subsection speculative examples}, to Dmitri Pavlov for discussions surrounding \cref{other KQ model structures}, and to both the NSF graduate research fellowship program (grant DGE-1106400) and to UC Berkeley's geometry and topology RTG (grant DMS-0838703) for financial support during the time that this work was carried out.  Finally, it is our distinct pleasure to thank Saul Glasman for his deep and abiding faith in the veracity of \cref{crazy lemma}
, without which this entire project would be utterly defunct.

\section{Model $\infty$-categories: definitions}\label{section model infty-cats defns}

In this section, we define model $\infty$-categories, Quillen adjunctions, and Quillen equivalences.  We will provide numerous examples of all of these concepts in \cref{section model infty-cats exs}.

\subsection{The definition of a model $\infty$-category}

\begin{defn}\label{define model infty-category}
We say that three wide subcategories $\bW,\bC,\bF \subset \M$ -- called the subcategories of \bit{weak equivalences}, of \bit{cofibrations}, and of \bit{fibrations}, respectively, and with their morphisms denoted by the symbols $\we$, $\cofibn$, and $\fibn$, respectively -- make an $\infty$-category $\M$ into a \bit{model $\infty$-category} if they satisfy the following evident $\infty$-categorical analogs of the usual axioms for a model category.
\begin{enumeratesmall}
\item[\limitaxiom] \textit{(limit)} $\M$ is finitely bicomplete.
\item[\twooutofthreeaxiom] \textit{(two-out-of-three)} $\bW$ satisfies the two-out-of-three property.
\item[\retractaxiom] \textit{(retract)} $\bW$, $\bC$, and $\bF$ are all closed under retracts.
\item[\liftingaxiom] \textit{(lifting)} There exists a lift in any commutative square
\[ \begin{tikzcd}
x \arrow[tail]{d}[swap]{i} \arrow{r} & z \arrow[two heads]{d}{p} \\
y \arrow{r} \arrow[dashed]{ru} & w
\end{tikzcd} \]
in which ($i$ is a cofibration, $p$ is a fibration, and) either $i$ or $p$ is a weak equivalence.
\item[\factorizationaxiom] \textit{(factorization)} Every map in $\M$ factors via both $\bF \circ (\bW \cap \bC)$ and $(\bW \cap \bF) \circ \bC$.

\end{enumeratesmall}
To indicate that a morphism lies in one of these subcategories, we will use the symbols $\we$, $\cofibn$, and $\fibn$, respectively.  We call $\bW \cap \bC \subset \M$ the subcategory of \bit{acyclic cofibrations}, and we call $\bW \cap \bF \subset \M$ the subcategory of \bit{acyclic fibrations}.  Morphisms lying in these subcategories are denoted by the symbols $\wcofibn$ and $\wfibn$, respectively.
\end{defn}

\begin{notn}\label{notn for decorating model structure data}
In order to disambiguate our notation associated to various model $\infty$-categories, we introduce the following conventions.
\begin{itemize}

\item We will always subscript the data associated to a ``named'' model $\infty$-category with (an abbreviation of) its name.  For instance, we write $s\S_\KQ$ to denote the model $\infty$-category given by the Kan--Quillen model structure on the $\infty$-category $s\S$ of simplicial spaces (see \cref{define kan--quillen model structure on sspaces}), and we write $\bF_\KQ \subset s\S$ to denote its subcategory of fibrations.

\item On the other hand, for an ``unnamed'' model $\infty$-category, we may subscript its associated data with the name of the underlying $\infty$-category.  For instance, if $\M$ is a model $\infty$-category, we may write $\bC_\M \subset \M$ to denote its subcategory of cofibrations.

\item When two different $\infty$-categories have model structures with the same name, we may additionally superscript their associated data with the name of the ambient $\infty$-category.  For instance, we may write $\bW^{s\Set}_\KQ \subset s\Set$ to denote the subcategory of weak equivalences in the classical Kan--Quillen model structure on the category $s\Set$ of simplicial sets (see \cref{define kan--quillen model structure on ssets}).  However, when no ambiguity should arise, we will generally omit this superscript.
\end{itemize}
\end{notn}

\begin{defn}\label{define co/fibrant objects}
An object of a model $\infty$-category $\M$ is called
\begin{itemizesmall}
\item \bit{cofibrant} if its unique map from the initial object $\es_\M$ is a cofibration,
\item \bit{fibrant} if its unique map to the terminal object $\pt_\M$ is a fibration, and
\item \bit{bifibrant} if it is both cofibrant and fibrant.
\end{itemizesmall}
We denote the full subcategories of these objects by $\M^c \subset \M$, $\M^f \subset \M$, and $\M^{cf} \subset \M$, respectively.  More generally, we will use these same superscripts to denote the intersection of some other subcategory of $\M$ with the indicated subcategory just defined, so that e.g.\! $\bW^{cf} = \bW \cap \M^{cf} \subset \M$ denotes the subcategory of weak equivalences between bifibrant objects.
\end{defn}

\begin{rem}\label{remark fundamental theorem}
As in the classical case, one should think of cofibrant objects as being ``good for mapping out of'', and of fibrant objects as being ``good for mapping into''.  Indeed, the \textit{fundamental theorem of model $\infty$-categories} (\Cref{fundthm:fundamental theorem}) asserts that if $x \in \M^c$ and $y \in \M^f$, then the natural map
\[ \hom_\M(x,y) \ra \hom_{\loc{\M}{\bW}}(x,y) \]
is a surjection, and becomes an equivalence after applying either ``$\infty$-categorical equivalence relation'' of \textit{left homotopy} or of \textit{right homotopy} to the source (in a sense made precise in \cref{fundthm:section fundamental theorem}).

Moreover, as in the classical case, factorization axiom {\factorizationaxiom} guarantees that any object of $\M$ admits both
\begin{itemizesmall}
\item a cofibrant replacement by an acyclic fibration and
\item a fibrant replacement by an acyclic cofibration.
\end{itemizesmall}
Taken together, these imply that every object of $\loc{\M}{\bW}$ can even be represented by a \textit{bifibrant} object of the model $\infty$-category $\M$.
\end{rem}

\begin{rem}\label{absolute computations only need to know co/fibrant objects}
In view of \cref{remark fundamental theorem}, we see that if we are only interested in computing hom-spaces in $\loc{\M}{\bW}$, then it suffices to know only which objects of $\M$ are co/fibrant (as long as we also have some control over the left and right homotopy relations).  However, there are many other constructions that one might perform in an $\infty$-category besides extracting hom-spaces, and from this point of view we can think of the subcategories $\bC,\bF \subset \M$ as telling us which objects are \textit{relatively} co/fibrant, i.e.\! co/fibrant in some undercategory or overcategory (see \cref{slice of a model infty-category}).  Nevertheless, we may view the data of the subcategories $\bW, \M^c,\M^f \subset \M$ as a ``first approximation'' to the data of a model structure.  As it can at times be quite difficult to prove the existence of a model structure, in some examples below we will just content ourselves by producing choices of co/fibrant objects corresponding to a given class of weak equivalences.
\end{rem}

\begin{rem}
One of the most important constructions that one might perform in a (1- or $\infty$-)category is the extraction of co/limits.  Classically, the theory of \textit{homotopy co/limits} in a model 1-category gives a way of computing co/limits in its $\infty$-categorical localization (see e.g.\! Theorem T.4.2.4.1).  (In contrast with the ``first approximation'' of \cref{absolute computations only need to know co/fibrant objects}, these certainly require the full model structure!)  We provide a theory of homotopy co/limits in model $\infty$-categories in \cref{qadjns:subsection ho-co-lims}.
\end{rem}

\begin{rem}\label{lifting axiom is stronger than that on htpy cat}
The $\infty$-categorical lifting axiom {\liftingaxiom} encodes a homotopically coherent version of the usual lifting axiom.  It is strictly stronger than the usual lifting axiom for hom-sets in the homotopy category (see \cref{model structure on homotopy category}).
\end{rem}

\begin{rem}
Factorization systems in $\infty$-categories (in which fillers are unique) are studied in \sec T.5.2.8.  These are distinct from what we study here, which might be called \textit{weak} factorization systems in $\infty$-categories.
\end{rem}

\begin{rem}
When defining model categories, there are always the choices to be made
\begin{itemizesmall}
\item of whether to require that the factorizations be functorial, and
\item of whether to require bicompleteness or only finite bicompleteness.
\end{itemizesmall}
Of course, in defining model $\infty$-categories, these choices persist.  With regards to both of these, we have chosen the less restrictive option.
\end{rem}

\begin{rem}
Since the opposite of a model $\infty$-category is canonically a model $\infty$-category, many of the statements that we make throughout this series of papers have obvious duals.  For conciseness, we will often just make whichever of the pair of dual statements is more convenient, and then we will simply refer to its dual if and when we require it.
\end{rem}

\begin{rem}\label{obvious lifting criteria}
There are many basic facts about model categories which follow easily and directly from the definitions; these generally remain true for model $\infty$-categories.  For instance, we will repeatedly use the facts
\begin{itemizesmall}
\item that $\bC = \llp(\bW \cap \bF)$,
\item that $\bW \cap \bC = \llp(\bF)$,
\item that $\bF = \rlp(\bW \cap \bC)$, and
\item that $\bW \cap \bF = \rlp(\bC)$.
\end{itemizesmall}
(The proofs remains the same, see \cite[Proposition 7.2.3]{Hirsch}.)  We also note here that it follows from these characterizations that $\M^\simeq \subset \bW \cap \bC \cap \bF \subset \M$, i.e.\! that the subcategory $\M^\simeq \subset \M$ of equivalences is contained in all three defining subcategories of a model $\infty$-category.
\end{rem}

\subsection{The definitions of Quillen adjunctions and Quillen equivalences}

Model categories are useful not just in isolation, but in how they relate to one another.  Following the classical situation, we make the following definition.

\begin{defn}\label{define quillen adjunction}
Suppose that $\M$ and $\N$ are two model $\infty$-categories, and suppose that
\[ F : \M \adjarr \N : G \]
is an adjunction between their underlying $\infty$-categories.  We say that this adjunction is a \bit{Quillen adjunction} if any of the following equivalent conditions is satisfied:
\begin{itemizesmall}
\item $F$ preserves cofibrations and acyclic cofibrations;
\item $G$ preserves fibrations and acyclic fibrations;
\item $F$ preserves cofibrations and $G$ preserves fibrations;
\item $F$ preserves acyclic cofibrations and $G$ preserves acyclic fibrations.
\end{itemizesmall}
(That these conditions are indeed equivalent follows immediately from \cref{obvious lifting criteria}.)  In this situation, we call $F$ a \bit{left Quillen functor} and we call $G$ a \bit{right Quillen functor}.
\end{defn}

\begin{rem}\label{quillen adjunction induces derived adjunction}
We prove as \cref{qadjns:adjunction thm} that a Quillen adjunction
\[ F : \M \adjarr \N : G \]
induces a canonical adjunction
\[ \bbL F : \loc{\M}{\bW_\M} \adjarr \loc{\N}{\bW_\N} : \bbR G \]
on localizations, called its \textit{derived adjunction} (see \cref{qadjns:defn der functors of q adjn}).
\end{rem}

Of course, we also have the following special case of \cref{define quillen adjunction}.

\begin{defn}\label{define quillen equivalence}
We say that a Quillen adjunction $F : \M \adjarr \N : G$ is a \bit{Quillen equivalence} if for all $x \in \M^c$ and all $y \in \N^f$, the equivalence
\[ \hom_\M(x,G(y)) \simeq \hom_\N(F(x),y) \]
induces an equivalence of subspaces
\[ \hom_{\bW_\M}(x,G(y)) \simeq \hom_{\bW_\N}(F(x),y) . \]
(Note that this condition can be checked on path components.)
\end{defn}

\begin{rem}\label{remark derived adjunction of a quillen equivalence is an equivalence}
Extending \cref{quillen adjunction induces derived adjunction}, we prove as \cref{qadjns:cor quillen equivce} that the derived adjunction of a Quillen equivalence induces an \textit{equivalence} of localizations.
\end{rem}

\section{Model $\infty$-categories: examples}\label{section model infty-cats exs}

Having discussed the generalities of model $\infty$-categories, we now proceed to give a number of examples.  In \cref{subsection examples of MICs} we give some examples of model $\infty$-categories, in \cref{subsection examples of adjunctions} we give some examples of Quillen adjunctions and Quillen equivalences between model $\infty$-categories, and in \cref{subsection speculative examples} we give some speculative examples whose verifications lie beyond the scope of the current project.  This section may be safely omitted from a first reading.

\subsection{Examples of model $\infty$-categories}\label{subsection examples of MICs}

In this subsection, we give a plethora of examples of model $\infty$-categories.  They are organized into subsubsections based on their nature:
\begin{itemizesmall}
\item in \cref{subsubsection ex foundational} we begin with some general examples of model $\infty$-categories,
\item in \cref{subsubsection ex named} we list some more specific examples model $\infty$-categories,
\item in \cref{subsubsection ex diagrams} we give examples of model structures on $\infty$-categories of diagrams in a model $\infty$-category, and
\item in \cref{subsubsection ex exotic} we explore some more exotic examples of model $\infty$-categories (which we have included primarily for their intrinsic interest).
\end{itemizesmall}

\subsubsection{General examples of model $\infty$-categories}\label{subsubsection ex foundational}

\begin{ex}\label{model 1-cat is model infty-cat}
In the case that $\M$ is actually a 1-category considered as an $\infty$-category with discrete hom-spaces, then \cref{define model infty-category} recovers the classical definition of a model category.  Thus, any model 1-category gives an example of a model $\infty$-category.
\end{ex}

\begin{ex}\label{trivial model structure}
Any finitely bicomplete $\infty$-category $\M$ has a \textit{trivial} model structure, in which we set $\bW = \M^\simeq$ and $\bC = \bF = \M$.  We denote this model $\infty$-category by $\M_\triv$.
\end{ex}

\begin{ex}\label{slice of a model infty-category}
If $\M$ is a model $\infty$-category and $x \in \M$ is any object, then both the undercategory $\M_{x/}$ and the overcategory $\M_{/x}$ inherit the structure of a model $\infty$-category, where in both cases the three defining subcategories are created by the forgetful functor to $\M$.  Iterating this observation, for any morphism $x \ra y$ in $\M$ we obtain the structure of a model $\infty$-category on $\M_{x/ \! /y}$.
\end{ex}

\begin{ex}
The \textit{recognition theorem} for cofibrantly generated model $\infty$-categories (\ref{recognize cofgen}) gives general criteria for the existence of a cofibrantly generated model structure on an $\infty$-category with respect to given choices of weak equivalences, generating cofibrations, and generating acyclic cofibrations (see \cref{section cofgen model infty-cats}).
\end{ex}

\subsubsection{Specific examples of model $\infty$-categories}\label{subsubsection ex named}

\begin{ex}
The main purpose of this paper is to define the \textit{Kan--Quillen model structure} on $s\S$, which will be given as \cref{define kan--quillen model structure on sspaces}, and which is denoted by $s\S_\KQ$.
\end{ex}

\begin{ex}
The $\infty$-category $\Cati$ of $\infty$-categories admits a \textit{Thomason model structure}, described in \cref{quillen equivalence giving thomason model str}, which is denoted by $(\Cati)_\Thomason$.
\end{ex}

\begin{ex}\label{ex resolution model str}
In future work, given a model $\infty$-category $\M$ (with a chosen set of generators), following \cite{DKS-E2} and \cite{BousCosimp} we will define a \textit{resolution model structure} (a/k/a an ``$\Etwo$ model structure'') on the $\infty$-category $s\M$ which presents the corresponding \textit{nonabelian derived $\infty$-category} of $\M$, which is denoted $s\M_\res$.
\end{ex}

\subsubsection{Examples of model structures on functor $\infty$-categories}\label{subsubsection ex diagrams}

\begin{ex}\label{ex proj model str}
Given a model $\infty$-category $\M$ and an $\infty$-category $\C$, there is sometimes a \textit{projective model structure} on the $\infty$-category $\Fun(\C,\M)$ (see \cref{qadjns:subsection ho-co-lims}), which is denoted by $\Fun(\C,\M)_\projective$.
\end{ex}

\begin{ex}\label{ex inj model str}
Given a model $\infty$-category $\M$ and an $\infty$-category $\C$, there is sometimes a \textit{injective model structure} on the $\infty$-category $\Fun(\C,\M)$ (see \cref{qadjns:subsection ho-co-lims}), which is denoted by $\Fun(\C,\M)_\injective$.
\end{ex}

\begin{ex}
Given a model $\infty$-category $\M$ and a Reedy category $\C$ (see \cref{qadjns:define reedy cat}), there is always a \textit{Reedy model structure} on the $\infty$-category $\Fun(\C,\M)$ (see \cref{qadjns:subsection reedy model str}), which is denoted by $\Fun(\C,\M)_\Reedy$.
\end{ex}

\subsubsection{Exotic examples of model $\infty$-categories}\label{subsubsection ex exotic}

\begin{ex}\label{model structure on homotopy category}
If $\M$ is a model $\infty$-category and $\ho(\M)$ is finitely bicomplete, then the model structure on $\M$ descends to a model structure on $\ho(\M)$.  In fact, the two-out-of-three, retract, and factorization axioms in $\M$ are even verifiable in $\ho(\M)$.  Moreover, the map
\[
\lim \left(
\begin{tikzcd}
& \hom_\M(y,w) \arrow{d}{i^*} \\
\hom_\M(x,z) \arrow{r}[swap]{p_*} & \hom_\M(x,w)
\end{tikzcd} \right)
\ra
\lim \left(
\begin{tikzcd}
& \hom_{\ho(\M)}(y,w) \arrow{d}{i^*} \\
\hom_{\ho(\M)}(x,z) \arrow{r}[swap]{p_*} & \hom_{\ho(\M)}(x,w)
\end{tikzcd} \right)
\]
is a surjection, so if the lifting axiom holds in $\M$ then it must hold in $\ho(\M)$ as well.  Of course, this generalizes \cref{model 1-cat is model infty-cat}: if $\M$ is a 1-category, then the projection $\M \xra{\sim} \ho(\M)$ is an equivalence.  On the other hand, since for an arbitrary $\infty$-category $\M$ there is usually a complicated interplay between co/limits in $\M$ and co/limits in $\ho(\M)$, it does not appear possible to draw any general conclusions about the relationship between these two model $\infty$-categories, or about the induced map
\[ \loc{\M}{\bW} \ra \loc{\ho(\M)}{\ho(\bW)} \]
on localizations.  (Note that the functor $\M \ra \ho(\M)$ is the unit of the left localization $\ho : \Cati \adjarr \Cat : \forget_\Cat$, but it is not generally itself an adjoint.)
\end{ex}

As we will elaborate upon in \cref{model structure as simultaneous generalization of left and right localizations}, the following example illustrates a particularly ``one-sided'' sort of model $\infty$-category.

\begin{ex}\label{left localizations give model structures}
Suppose that $\M$ is a finitely bicomplete $\infty$-category and that $\leftloc : \M \adjarr \leftloc \M : \forget$ is a left localization.  This means that we may consider the right adjoint $\forget : \leftloc \M \hookra \M$ as the inclusion of the reflective subcategory of ``local'' objects, and that for any $x \in \M$ and any $y \in \leftloc\M \subset \M$, the localization map $x \ra \leftloc x$ induces an equivalence
\[ \hom_\M(\leftloc x , y) \xra{\sim} \hom_\M(x,y) . \]
This fits squarely into the ``first approximation'' rubric of \cref{absolute computations only need to know co/fibrant objects}: we should consider $\M^c = \M$ and $\M^f = \forget (\leftloc\M) \subset \M$ (with trivial left/right homotopy relations).  Hence, it is natural to guess that there might be a \textit{left localization} model structure $\M_\leftloc$ on $\M$, in which
\begin{itemizesmall}
\item $\bW_\leftloc \subset \M$ is the subcategory of morphisms that become equivalences in $\leftloc\M$,
\item $\bC_\leftloc = \M$, and
\item $\bF_\leftloc \subset \M$ is determined by the lifting condition $\bF_\leftloc = \rlp((\bW \cap \bC)_\leftloc) = \rlp(\bW_\leftloc)$;
\end{itemizesmall}
in this case, the left adjoint will coincide with the localization $\M \ra \loc{\M}{\bW_\leftloc} \simeq \leftloc\M$ (and moreover, the left homotopy relation will be visibly trivial (and hence the right homotopy relation will be trivial as well)).  What remains, then, is to check
\begin{itemizesmall}
\item the half of the lifting axiom asserting that $(\bW \cap \bF)_\leftloc \subset \rlp(\bC_\leftloc) = \rlp(\M)$, and
\item the half of the factorization axiom regarding $\bF_\leftloc \circ (\bW \cap \bC)_\leftloc = \bF_\leftloc \circ \bW_\leftloc$.
\end{itemizesmall}

Now, for a map to have $\rlp(\M)$ in particular means that it has the right lifting property against the identity map from itself, and this implies that the map is an equivalence (with inverse given by the guaranteed lift).  Thus, $\rlp(\M) \subset \M^\simeq$.  On the other hand, clearly $\M^\simeq \subset \rlp(\M)$, so $\rlp(\M) = \M^\simeq$, and hence to satisfy the lifting axiom it is both necessary and sufficient to have that $(\bW \cap \bF)_\leftloc \subset \M^\simeq$.  (Of course, the reverse containment follows from the definitions, so it is also necessary and sufficient to have that $(\bW \cap \bF)_\leftloc = \M^\simeq$.)

The factorization axiom here is more subtle, and there does not appear to be a general criterion for when this might hold.  One possibility is that we might try to use the factorization of an arbitrary map $x \ra y$ via the construction
\[ \begin{tikzcd}
x \arrow[bend left=15]{rrd} \arrow[bend right=15]{ddr} \arrow{rd} \\
& \leftloc x \underset{\leftloc y}{\times} y \arrow{r} \arrow{d} & y \arrow{d} \\
& \leftloc x \arrow{r} & \leftloc y ,
\end{tikzcd} \]
i.e.\! via the composite
\[ x \ra \leftloc x \underset{\leftloc y}{\times} y \ra y . \]
In the case that the left localization $\leftloc$ is additionally \textit{left exact}, i.e.\! in the case that it commutes with finite limits (see Remark T.5.3.2.3), then this does indeed produce a factorization as $\bF_\leftloc \circ \bW_\leftloc$:
\begin{itemizesmall}
\item the first map is in $\bW_\leftloc$ by the left exactness of $\leftloc$, and
\item the second map is in $\bF_\leftloc$ since it is a pullback of the map $\leftloc x \ra \leftloc y$, which is directly seen to have $\rlp(\bW_\leftloc)$.
\end{itemizesmall}
In fact, this does not require the full strength of left exactness, but is only using the weaker notion of $\leftloc$ being a \textit{locally cartesian} localization (see \cite[1.2]{GepnerKock}).\footnote{There are a number of results in the literature surrounding factorizations in 1-categories which, if suitably generalized to $\infty$-categories, would give a much wider class of left localizations in which the factorization axiom holds; a notable example is \cite[Corollary 3.4]{CHK}.  (Somewhat relatedly, factorization systems in stable $\infty$-categories are explored in \cite{Fosco}.)}

In a different direction, the proof of \cite[Proposition 2.9]{SalchBall} (which actually only uses \textit{finite} bicompleteness) appears to generalize directly to show that whenever the subcategory $\leftloc\M \subset \M$ is additionally ``strictly saturated'' (in the sense of \cite[Definition 2.7]{SalchBall}), then the model $\infty$-category $\M_\leftloc$ does indeed exist, and moreover
\begin{itemizesmall}
\item the subcategory $\bF_\leftloc \subset \M$ consists of precisely those maps which are pullbacks of maps in the subcategory $\leftloc\M \subset \M$, while
\item $(\bW \cap \bF)_\leftloc \simeq \M^\simeq$.
\end{itemizesmall}
More broadly, \cite{SalchBall} makes a detailed study of these left localization model category structures and their evident duals (see \cref{right localizations give model structures}), which should presumably generalize naturally to model $\infty$-categories.
\end{ex}

\begin{ex}\label{pi-0 model structure on spaces}
A particular case where \cref{left localizations give model structures} does indeed give a model structure is the left localization $\pi_0 : \S \adjarr \Set : \disc$.  Here, we have that $\bF_{\pi_0} = \rlp(\bW_{\pi_0})$ consists of precisely the \textit{\'{e}tale} maps of spaces, i.e.\! those maps which induce $\pi_{\geq 1}$-isomorphisms for every basepoint of the source.  (Note that this condition is independent from that of being $0$-connected; in particular, such a map may be neither injective nor surjective on $\pi_0$.)  It follows that $(\bW \cap \bF)_{\pi_0} = \S^\simeq$, and so we have the required lifting condition.  Moreover, the proposed factorization of \cref{left localizations give model structures} is in this case precisely the standard epi-mono factorization, and this gives the required factorization $\bF_{\pi_0} \circ \bW_{\pi_0}$.  As $\S$ is certainly finitely bicomplete, we obtain a model $\infty$-category $\S_{\pi_0}$ with localization $\S \ra \loc{\S}{\bW_{\pi_0}} \simeq \Set$.
\end{ex}

\begin{ex}\label{n-truncation model structure on spaces}
In fact, \cref{pi-0 model structure on spaces} generalizes to the left localizations $\tau_{\leq n} : \S \adjarr \S^{\leq n} : \forget_{\leq n}$.  Using the notation of \cref{pi-0 model structure on spaces is almost cofgen} below, we have that
\[ \bF_{\tau_{\leq n}} = \rlp((I^\S_\triv)_{\geq n+2}) \cap \rlp'(\{S^n \ra \pt_\S \}) \]
consists of precisely those maps which induce $\pi_{\geq n+1}$-isomorphisms for every basepoint in the source.  So again $(\bW \cap \bF)_{\tau_{\leq n}} = \S^\simeq$, and it is easy to check using the long exact sequence in homotopy for a pullback square that again the suggestion in \cref{left localizations give model structures} yields the factorization $\bF_{\tau_{\leq n}} \circ \bW_{\tau_{\leq n}}$.  (Alternatively, an easy spheres-and-disks construction yields this factorization as well.)  Hence, we obtain a model $\infty$-category $\S_{\tau_{\leq n}}$ with localization $\S \ra \loc{\S}{\bW_{\tau_{\leq n}}} \simeq \S^{\leq n}$.
\end{ex}

\begin{rem}
It is clear that \cref{n-truncation model structure on spaces} relies on the fact that there is a good theory of cellular approximation in $\S$ (e.g., the fact that attaching an $(n+2)$-cell to a space doesn't change its $n$-truncation).  Thus, it does not appear to immediately generalize to an arbitrary $\infty$-topos: it is important that the generators be suitably compatible with the truncation functors.
\end{rem}

\begin{rem}
The Kan--Quillen model structure on $s\S$ of \cref{define kan--quillen model structure on sspaces} has its weak equivalences created by the left adjoint $|{-}| : s\S \ra \S$, but it is \textit{not} a left localization model structure (in the sense of \cref{left localizations give model structures}).  Indeed, this left adjoint is certainly not left exact: the question of when limits commute with geometric realizations is generally very difficult to answer.  (However, in the case of pullbacks, this question is addressed to some extent by \cref{consequence of right properness of kan--quillen model structure} below; see also the surrounding Remarks \ref{observe that the fiber of a fibration of sspaces is homotopically correct}, \ref{compare with rezk's realization fibrations}, \ref{compare with anderson and bousfield--friedlander}, \and \ref{compare with result for levelwise-connected sspaces}.  The case of more general limits will also be addressed in \cref{qadjns:subsection ho-co-lims}.)
\end{rem}

\begin{ex}\label{right localizations give model structures}
Given a right localization $\forget : \rightloc \M \adjarr \M : \rightloc$ with $\M$ finitely bicomplete, we obtain a dual story to that of \cref{left localizations give model structures}.  Now we should think of every object as \textit{fibrant}, and of the left adjoint as the inclusion of the coreflective subcategory of \textit{cofibrant} objects.  In this case, the guess at a \textit{right localization} model structure $\M_\rightloc$ on $\M$ has that
\begin{itemizesmall}
\item $\bW_\rightloc \subset \M$ is the subcategory of morphisms that become equivalences in $\rightloc\M$,
\item $\bF_\rightloc = \M$, and
\item $\bC_\rightloc \subset \M$ is determined by the lifting condition $\bC_\rightloc = \llp(\bW_\rightloc)$.
\end{itemizesmall}
From here, it remains to check
\begin{itemizesmall}
\item the lifting condition $(\bW \cap \bC)_\rightloc \subset \llp(\bF_\rightloc) = \M^\simeq$, and
\item the factorization condition for $(\bW \cap \bF)_\rightloc \circ \bC_\rightloc = \bW_\rightloc \circ \bC_\rightloc$.
\end{itemizesmall}
Now, the factorization condition will be implied by the \textit{right exactness} of the right localization $\rightloc$ (or more generally, if the right localization is ``locally cocartesian'').
\end{ex}

\begin{rem}\label{model structure as simultaneous generalization of left and right localizations}
\cref{left localizations give model structures} reinterprets a left localization as describing a model $\infty$-category in which all objects are \textit{cofibrant}; dually, \cref{right localizations give model structures} reinterprets a right localization as describing a model $\infty$-category in which all objects are \textit{fibrant}.  Since in an arbitrary model $\infty$-category, not all the objects will be cofibrant and not all the objects will be fibrant, we should think of the notion of a model $\infty$-category as giving a \textit{simultaneous generalization} of the notions of left and right localizations.
\end{rem}

\begin{ex}\label{try to put model structure on bounded spectra}
As a simple case of \cref{model structure as simultaneous generalization of left and right localizations}, we can obtain a ``first approximation'' (as in \cref{absolute computations only need to know co/fibrant objects}) to a model structure on the $\infty$-category $\Sp$ of spectra which would present the $\infty$-category $\Sp^{[m,n]}$ of spectra that only have nontrivial homotopy groups in some interval $[m,n] \subset \bbZ$: the cofibrant objects would be $\Sp^{\geq m} \subset \Sp$, while the fibrant objects would be $\Sp^{\leq n} \subset \Sp$.  This example, though illustrative, is somewhat degenerate, since any weak equivalence between bifibrant objects is already an equivalence.  Indeed, the ``homotopy relations'' would all be trivial, corresponding to the fact that we have an inclusion $\Sp^{[m,n]} \subset \Sp$ which is a section to the projection $\tau_{\geq m} \circ \tau_{\leq n} \simeq \tau_{\leq n} \circ \tau_{\geq m} : \Sp \ra \Sp^{[m,n]}$.
\end{ex}

\begin{rem}\label{try to use n-cotruncation of pointed simply connected spaces to get a model structure}
Analogously to \cref{n-truncation model structure on spaces}, we might try to obtain a model structure as in \cref{right localizations give model structures} from the right localization $\forget_{\geq n} : \S^{\geq n}_* \adjarr \S^{\geq 1}_* : \tau_{\geq n}$.  However, the existence of such a model structure is much less clear.

We do still have the lifting condition.  Indeed, suppose that $x \ra y$ is in $(\bW \cap \bC)_{\tau_{\geq n}}$.  Since $x \ra y$ is in $\bW_{\tau_{\geq n}}$, it induces an isomorphism on $\pi_{\geq n}$.  On the other hand, obtaining a lift in the commutative square
\[ \begin{tikzcd}
x \arrow{r} \arrow{d} & \tau_{\leq n-1} x \arrow{d} \\
y \arrow{r} & \tau_{\leq n-1} y
\end{tikzcd} \]
(in which $\tau_{\geq n}$ applied to the right map yields $\id_{\pt_\S}$), we see that the map also induces isomorphisms on $\pi_{\leq n-1}$.

But the factorization condition is trickier.  Given a map $x \ra y$ in $\S^{\geq 1}_*$, the map
\[ \tau_{\geq n} y \ra \tau_{\geq n}y \amalg_{\tau_{\geq n} x} x \]
will not necessarily be a $\pi_{\geq n}$-isomorphism.  For instance, when $n \geq 2$ and $x$ is a 1-type, then this map will just be the inclusion $\tau_{\geq n} y \ra \tau_{\geq n}y \vee x$ of the first wedge summand, which will not generally induce a $\pi_{\geq n}$-isomorphism.  (Of course, this observation does not preclude the existence of suitable factorizations.)
\end{rem}

\begin{rem}
While not exactly an example of a model $\infty$-category, it seems worth observing that given a model $\infty$-category $\M$, the pair $(\M^c,\bW^c)$ gives an example of the evident $\infty$-categorical analog of Waldhausen's notion of a ``category with cofibrations and weak equivalences'' (see \cite[\sec 1.2]{Wald1126}); moreover, a left Quillen functor $\M \ra \N$ of model $\infty$-categories then gives rise to an ``exact functor'' of such.  Consequently, there is an evident definition of the \textit{algebraic K-theory} of these objects, visibly functorial for left Quillen functors (see \cite[\sec 1.3]{Wald1126}).

Note that this does \textit{not} coincide with Barwick's notion of a ``Waldhausen $\infty$-category'' given as \cite[Definition 2.7]{BarwickAKT}.  Rather, one recovers this latter notion as a special case of an ``$\infty$-category with cofibrations and weak equivalences'' in which the weak equivalences are just the equivalences.  On the other hand, it seems likely that the algebraic K-theory of $(\M^c,\bW^c)$ in the above sense would simply compute the algebraic K-theory of the localization $\loc{\M^c}{(\bW^c)} \simeq \loc{\M}{\bW}$ with respect to its \textit{maximal} pair structure in Barwick's sense (as this is true for model 1-categories, see \cite[Proposition 9.15 and Corollary 10.10.3]{BarwickAKT}).\footnote{This equivalence is given by (the dual of) \cref{qadjns:inclusion of fibts induces equivce on gpd-compln}.}
\end{rem}

\subsection{Examples of Quillen adjunctions and Quillen equivalences}\label{subsection examples of adjunctions}

\begin{ex}\label{quillen equivalence between kan--quillen model structures}
As will be immediate from Definitions \ref{define kan--quillen model structure on ssets} \and \ref{define kan--quillen model structure on sspaces}, the adjunction $\pi_0 : s\S_\KQ \adjarr s\Set_\KQ : \disc$ (see \cref{adjunction between sspaces and ssets}) between the Kan--Quillen model $\infty$-category structures on $s\S$ and $s\Set$ is a Quillen equivalence.
\end{ex}

\begin{ex}
In \cref{bonus properties of sd and Ex}, we will see that the ``subdivision'' and ``extension'' endofunctors of \cite{KanEx}, suitably extended from $s\Set$ to $s\S$ (see \cref{subsection Ex-infty}), define a Quillen equivalence $\sd : s\S_\KQ \adjarr s\S_\KQ : \Ex$.
\end{ex}

\begin{ex}\label{example lift cofgen to get quillen adjunction}
The \textit{lifting theorem} for cofibrantly generated model $\infty$-categories (\ref{lift cofgen}) gives general general criteria for constructing Quillen adjunctions by lifting a cofibrantly generated model structure along a left adjoint (see \cref{section cofgen model infty-cats}).
\end{ex}

\begin{ex}\label{bousfield localization of trivial model structure}
If a left localization $\leftloc : \M \adjarr \leftloc \M : \forget$ induces a left localization model structure $\M_\leftloc$ as in \cref{left localizations give model structures}, then we obtain a Quillen adjunction $\id_\M : \M_\triv \adjarr \M_\leftloc : \id_\M$, whose derived adjunction is precisely the original left localization $\leftloc : \M \adjarr \leftloc\M : \forget$.  This is of course closely related to the theory of \textit{left Bousfield localizations} of model categories (see e.g.\! \cite[\sec 3.3]{Hirsch}).  Dual statements apply to the right localization model structures of \cref{right localizations give model structures}.
\end{ex}

\begin{ex}
As a particular case of \cref{bousfield localization of trivial model structure}, the model $\infty$-category $\S_{\tau_{\leq n}}$ of \cref{n-truncation model structure on spaces} participates in a Quillen adjunction $\id_\S : \S_\triv \adjarr \S_{\tau_{\leq n}} : \id_\S$, whose derived adjunction is $\tau_{\leq n} : \S \adjarr \S^{\leq n} : \forget_{\leq n}$.
\end{ex}

\begin{ex}\label{quillen equivalence giving thomason model str}
Recall that the $\infty$-category $\CSS$ of \textit{complete Segal spaces} sits as a left localization $s\S \adjarr \CSS$, and moreover admits a natural equivalence $\CSS \simeq \Cati$ (see \cref{rnerves:section CSSs}).  We show as \cref{gr:Thomason model str on CSS} that, as a particular case of \cref{example lift cofgen to get quillen adjunction}, we can transfer the Kan--Quillen model structure of \cref{define kan--quillen model structure on sspaces} along the adjunction $s\S \adjarr \CSS \simeq \Cati$ to obtain the \textit{Thomason model structure} on the $\infty$-category $\Cati$ of $\infty$-categories.  This Quillen adjunction is in fact a Quillen equivalence, and hence \cref{remark derived adjunction of a quillen equivalence is an equivalence} implies that the Thomason model structure on $\Cati$ once again presents the $\infty$-category $\S$.  As explained in \cref{gr:different Thomasons}, this model structure resolves some of the less satisfying aspects of the classical Thomason model structure on $\strcat$ (which also presents $\S$).
\end{ex}

\begin{ex}
Given a model $\infty$-category $\M$ and an $\infty$-category $\C$, if both the projective and injective model structures on $\Fun(\C,\M)$ exist (see Examples \ref{ex proj model str} \and \ref{ex inj model str}), then the identity adjunction defines a Quillen equivalence
\[ \id_{\Fun(\C,\M)} : \Fun(\C,\M)_\projective \adjarr \Fun(\C,\M)_\injective : \id_{\Fun(\C,\M)} \]
between them (see \cref{qadjns:rem qeqs between inj proj and reedy}).
\end{ex}

\begin{ex}
Given a model $\infty$-category $\M$ and a Reedy category $\C$, if the projective model structure on $\Fun(\C,\M)$ exists (see \cref{ex proj model str}), then the identity adjunction defines a Quillen equivalence
\[ \id_{\Fun(\C,\M)} : \Fun(\C,\M)_\projective \adjarr \Fun(\C,\M)_\Reedy : \id_{\Fun(\C,\M)} \]
(see \cref{qadjns:rem qeqs between inj proj and reedy}).
\end{ex}

\begin{ex}
Given a model $\infty$-category $\M$ and a Reedy category $\C$, if the injective model structure on $\Fun(\C,\M)$ exists (see \cref{ex inj model str}), then the identity adjunction defines a Quillen equivalence
\[ \id_{\Fun(\C,\M)} : \Fun(\C,\M)_\Reedy \adjarr \Fun(\C,\M)_\injective : \id_{\Fun(\C,\M)} \]
(see \cref{qadjns:rem qeqs between inj proj and reedy}).
\end{ex}

\begin{ex}
Given a model $\infty$-category $\M$ and an $\infty$-category $\C$ such that $\M$ admits $\C$-shaped colimits, the adjunction
\[ \colim : \Fun(\C,\M) \adjarr \M : \const \]
is
\begin{itemizesmall}
\item always a Quillen adjunction with respect to the projective model structure on $\Fun(\C,\M)$ if it exists (see \cref{qadjns:rem q adjns for proj and inj model strs})
\item sometimes (but not always) a Quillen adjunction if $\C$ is additionally a Reedy category and we equip $\Fun(\C,\M)$ with the Reedy model structure, e.g.\! in the case that $\C$ has \textit{fibrant constants} (see \cref{qadjns:defn co/fibt constants}).
\end{itemizesmall}
\end{ex}

\begin{ex}
Given a model $\infty$-category $\M$ and an $\infty$-category $\C$ such that $\M$ admits $\C$-shaped limits, the adjunction
\[ \const : \M \adjarr \Fun(\C,\M) : \lim \]
is
\begin{itemizesmall}
\item always a Quillen adjunction with respect to the injective model structure on $\Fun(\C,\M)$ if it exists (see \cref{qadjns:rem q adjns for proj and inj model strs})
\item sometimes (but not always) a Quillen adjunction if $\C$ is additionally a Reedy category and we equip $\Fun(\C,\M)$ with the Reedy model structure, e.g.\! in the case that $\C$ has \textit{cofibrant constants} (see \cref{qadjns:defn co/fibt constants}).
\end{itemizesmall}
\end{ex}

\subsection{Speculative examples}\label{subsection speculative examples}

\begin{specul}\label{other KQ model structures}
Let us temporarily refer to the model structure of \cref{define kan--quillen model structure on sspaces} as the ``strong'' Kan--Quillen model structure on $s\S$: its subcategory $\bW_\KQ \subset s\S$ of weak equivalences is created by the geometric realization functor $|{-}| : s\S \ra \S$, and it is cofibrantly generated by the sets
\[ I_{\KQ_{\textup{strong}}} = \{ \partial \Delta^n \ra \Delta^n \}_{n \geq 0} \]
and
\[ J_{\KQ_{\textup{strong}}} = \{ \Lambda^n_i \ra \Delta^n \}_{0 \leq i \leq n \geq 1} . \]
Then, there should exist other Kan--Quillen model structures on $s\S$: these would have the same subcategory of weak equivalences, but would have more cofibrations.  For instance, we might define ``medium'' and ``weak'' Kan--Quillen model structures by extending the set of generating cofibrations to be given by
\[ I_{\KQ_{\textup{medium}}} = I_{\KQ_{\textup{strong}}} \cup \{ S^i \tensoring \partial \Delta^n \ra S^i \tensoring \Delta^n \}_{i \geq 1, n \geq 0} \]
and
\[ I_{\KQ_{\textup{weak}}} = I_{\KQ_{\textup{medium}}} \cup \{ S^i \tensoring \Delta^n \ra \pt_\S \tensoring \Delta^n \}_{i \geq 1, n \geq 0} , \]
with sets of generating acyclic cofibrations extended to match.  There would then exist Quillen equivalences
\[ s\S_{\KQ_{\textup{strong}}} \adjarr s\S_{\KQ_{\textup{medium}}} \adjarr s\S_{\KQ_{\textup{weak}}} \]
(in which all underlying functors are $\id_{s\S}$): moving to the right, more and more maps become cofibrations, while moving to the left, more and more maps become fibrations.  This explains the terminology: the geometric realization functor (being a colimit) already plays well with colimits, and hence it does not seem to be particularly useful to identify more maps as cofibrations.  On the other hand, these variants would enjoy certain features not shared by $s\S_{\KQ_{\textup{strong}}}$.
\begin{itemize}

\item The model $\infty$-category $s\S_{\KQ_{\textup{medium}}}$ would be obtained by closing up the generating sets under the tensoring, i.e.\! by performing an \textit{enriched} small object argument (see \cref{why unenriched soa}), which would easily provide \textit{functorial} factorizations.

\item The model $\infty$-category $s\S_{\KQ_{\textup{weak}}}$ would have all objects cofibrant, just as $s\Set_\KQ$.  However, it seems that the primary importance of this fact is that it implies left properness, so this may not be much of an advantage, since $s\S_{\KQ_{\textup{strong}}}$ is already left proper (for essentially trivial reasons).

\end{itemize}

The existence of these alternate Kan--Quillen model structures almost follows easily from the recognition theorem for cofibrantly generated model $\infty$-categories (\ref{recognize cofgen}) and our proof of \cref{kan--quillen model structure on sspaces} (the main theorem of this paper, which asserts the existence of $s\S_{\KQ_{\textup{strong}}}$); more precisely, using the results presented here, it is straightforward to verify all the conditions given in \cref{recognize cofgen} except for condition \ref{recognition theorem 3}.  On the other hand, it seems eminently plausible that Smith's recognition theorem for combinatorial model categories (see Proposition T.A.2.6.8), especially its simpler special case given by Lurie (see Proposition T.A.2.6.13), would admit a straightforward generalization to the model $\infty$-categorical setting.  From here, a version of Proposition T.A.2.6.13 would guarantee that \textit{any} set of maps $I \subset s\S$ containing $I_{\KQ_{\textup{strong}}}$ would constitute a set of generating cofibrations for a model structure on $s\S$ with subcategory of weak equivalences given by $\bW_\KQ \subset s\S$.  (Condition (1) would be satisfied by a combination of variants of Example T.A.2.6.11 and Corollary T.A.2.6.12, condition (2) would be true (even without the assumption that $I \supset I_{\KQ_{\textup{strong}}}$) because geometric realization (being a colimit) commutes with pushouts, and condition (3) would follow from the fact that $\rlp(I) \subset \rlp(I_{\KQ_{\textup{strong}}}) \subset \bW_\KQ$.)

\end{specul}

\begin{rem}
It seems that the model $\infty$-category $s\S_{\KQ_{\textup{weak}}}$ of \cref{other KQ model structures} would be closely related to the Moerdijk model structure on $ss\Set$ described in \cref{compare with moerdijk}.  Moreover, a putative set of generating acyclic cofibrations
\[ J_{\KQ_{\textup{strong}}} \cup \{ S^j \tensoring \Lambda^n_i \ra S^j \tensoring \Delta^n \}_{j \geq 1, 0 \leq i \leq n \geq 1} \]
seems closely related to our comparison with the $\pi_*$-Kan condition in \cref{compare with anderson and bousfield--friedlander}.
\end{rem}

\begin{specul}\label{KQ model structure on sspectra}
There should exist a Kan--Quillen model structure on the $\infty$-category $s\Sp$ of simplicial spectra.  However, the ``levelwise infinite loopspace'' functor $s\Omega^\infty : s\Sp \ra s\S_*$ isn't conservative, and so it wouldn't make much sense to lift this from a Kan--Quillen model structure on $s\S_*$.\footnote{On the other hand, this issue would vanish if we were to restrict to \textit{connective} spectra.}  (By contrast, the usual model structure on simplicial abelian groups is lifted from $s\Set_\KQ$.)  This should give rise to model structures e.g.\! on simplicial module spectra over a (simplicial) ring spectrum, and in other stable contexts.  Alternatively, the putative applications of such model structures might all be handled sufficiently by resolution model structures.
\end{specul}

\begin{specul}\label{model infty-cat from simp model cat}
Given a simplicial model category $\M_\bullet$, we can consider its underlying $s\Set$-enriched category as an $\infty$-category $\M_\bullet \in (\strcat_{s\Set})_\Bergner$ in its own right.  Closing up the defining simplicial subcategories $\bW_\bullet,\bC_\bullet,\bF_\bullet \subset \M_\bullet$ to subcategories in the $\infty$-categorical sense should then determine a model $\infty$-category structure (though there is some subtlety in ensuring that the model $\infty$-category axioms continue to hold in the $\infty$-categorical sense).

As an illustrative example of this apparent phenomenon, let use consider the case of the simplicial model category $(s\Set_\bullet)_\KQ$.  Let us write
\[ \bW_\he \subset \bW_\whe \subset s\Set \]
for the subcategories of \textit{homotopy equivalences} and \textit{weak homotopy equivalences}, respectively (with respect to the Kan--Quillen model structure (so that by definition, $\bW_\whe = \bW_\KQ$)).  Then, it is not hard to see that the canonical functor
\[ s\Set[\bW^{-1}_\he] \ra s\Set[\bW^{-1}_\whe] \simeq \ho(\S) \]
on homotopy categories is a left localization.  Moreover, every acyclic fibration in $s\Set_\KQ$ is actually a homotopy equivalence, so that \textit{every} map is simplicially homotopic to a cofibration.  Taking these two facts together, it seems reasonable to guess that this procedure yields a left localization model structure (in the sense of \cref{left localizations give model structures}) corresponding to a left localization adjunction
\[ \loc{s\Set}{\bW_\he} \adjarr \loc{s\Set}{\bW_\whe} \simeq \S . \]
\end{specul}

\begin{specul}\label{free resolutions model structure}
There should exist a model structure on simplicial objects in algebras over an operad which accounts for \textit{free resolutions}.  For example, this would recover as a special case the model structure on simplicial commutative rings alluded to in \cite[\sec 2]{Quillenhomcomm} (and laid out explicitly in \cite[\sec 3.1]{SchwedeSpectraAQ}), and would provide a framework organizing the ``prove it for a free simplicial resolution, then prove that it commutes with (sifted) colimits'' arguments involving ``$B$-structured $n$-disk algebras'' (e.g.\! $\bbE_n$ algebras) that appear throughout \cite{AFPKD}.
\end{specul}

\begin{specul}\label{joyal model str on sspaces}
There should exist a \textit{Joyal model structure} on $s\S$, whose fibrant objects are the ``homotopical quasicategories'', namely those $Y \in s\S$ such that for all $n \geq 0$ and all $0 < i < n$, the inner horn inclusion $\Lambda^n_i \ra \Delta^n$ induces a surjection
\[ Y_n \simeq \Match_{\Delta^n} (Y) \ra \Match_{\Lambda^n_i} (Y) \]
in $\S$.  This should moreover participate in a left Bousfield localization $s\S_\Joyal \adjarr s\S_{\KQ_{\textup{weak}}}$ with the ``weak'' Kan--Quillen model structure of \cref{other KQ model structures}.

The proof of the Joyal model structure on $s\S$ would presumably follow that of the one on $s\Set$ fairly closely; a short and streamlined exposition of the latter is given in \cite[Appendix C]{DuggerSpivak} (in contrast with the one given in \sec T.2.2.5, which proceeds by using the model category $(\strcat_{s\Set})_\Bergner$.)

There should similarly exist other model $\infty$-categories which present $\Cati$ based on model 1-categories that do, for instance
\begin{itemizesmall}
\item a Barwick--Kan model structure on the $\infty$-category of relative $\infty$-categories (see Definitions \Cref{rnerves:define rel infty-cat} \and \Cref{rnerves:define BarKan rel str on RelCati}),
\item a Bergner model structure on the $\infty$-category of $s\S$-enriched $\infty$-categories (see \cref{hammocks:define sspatial infty-cats}), and
\item a Bergner model structure on ``Segal pre-categories'' in $s\S$, i.e.\! those simplicial spaces whose $0\th$ space is discrete.
\end{itemizesmall}
\end{specul}

\begin{specul}\label{speculn derived koszul duality}
The central theorem regarding \textit{formal moduli problems} for $\bbE_n$ algebras, described in \cite{LurieICM} (and made precise in \cite{LurieDAGX}), posits an equivalence between the $\infty$-categories of formal $\bbE_n$ moduli problems and of augmented $\bbE_n$ algebras (see \cite[Theorem 6.20]{LurieICM}).  This equivalence takes an augmented $\bbE_n$ algebra to its associated \textit{Maurer--Cartan functor}.

The inverse equivalence is somewhat trickier to describe.  When the formal $\bbE_n$ moduli problem is \textit{affine} or \textit{pro-affine} (i.e.\! corepresented by a small $\bbE_n$ algebra or by a pro-object in such), then this inverse equivalence is implemented by \textit{Koszul duality} (see \cite[Example 8.5 and Remark 8.9]{LurieICM}).  However, more generally one must take a resolution of the given formal $\bbE_n$ moduli problem by a \textit{smooth hypercovering} consisting of pro-affine ones, apply Koszul duality levelwise to this simplicial object, and then take the colimit.  Thus, in general this inverse equivalence is implemented by \textit{the derived functor of Koszul duality} (in analogy with e.g.\! the statement that the cotangent complex is the derived functor of derivations).

Hence, there should then exist a model structure on the $\infty$-category of simplicial objects in formal $\bbE_n$ moduli problems (presumably related to the model structure of \cref{free resolutions model structure}): appropriate levelwise pro-affine objects would be cofibrant, smooth hypercoverings would be acyclic fibrations, and then e.g.\! for an arbitrary formal $\bbE_n$ moduli problem, \cite[Proposition 8.19]{LurieICM} would provide a cofibrant replacement by an acyclic fibration.
\end{specul}

\section{Cofibrantly generated model $\infty$-categories}\label{section cofgen model infty-cats}

In this section we discuss \textit{cofibrantly generated} model $\infty$-categories.  As in the classical situation, these are model structures which are determined by a relatively small amount of data, namely by a set of \textit{generating cofibrations} and a set of \textit{generating acyclic cofibrations}, which simultaneously
\begin{itemizesmall}
\item generate the subcategories $\bC$ and $\bW \cap \bC$, respectively, in a suitable sense (as their names suggest),
\item detect the subcategories $\bW \cap \bF$ and $\bF$, respectively, in accordance with \cref{obvious lifting criteria}, and
\item are suited for obtaining the factorizations required by \cref{define model infty-category}.
\end{itemizesmall}
In \cref{section define kan--quillen model structure}, we will use this setup to define the Kan--Quillen model structure on the $\infty$-category $s\S$ of simplicial spaces.

We begin with a sequence of definitions.  They are all direct generalizations of their 1-categorical counterparts (after replacing a set of maps with a set of homotopy classes of maps), and it is routine to verify that they enjoy completely analogous properties (see e.g.\! \cite[\sec 10.4-5]{Hirsch}).  We also point out once and for all that these definitions do not depend on choices of representatives for the elements of the given set $I$ of homotopy classes of maps.

\begin{defn}\label{injectives have rlp(I)}
Given a set $I$ of homotopy classes of maps in $\C$, the subcategory $I \dashinj \subset \C$ of \bit{$I$-injectives} is the subcategory of maps with $\rlp(I)$.
\end{defn}

\begin{defn}\label{cofibrations have llp(rlp(I))}
Given a set $I$ of homotopy classes of maps in $\C$,  the subcategory $I \dashcof \subset \C$ of \bit{$I$-cofibrations} is the subcategory of maps with $\llp(\rlp(I))$.
\end{defn}

\begin{defn}\label{rel cell cxes}
Assume that $\C$ admits pushouts and sequential colimits.  Given a set $I$ of homotopy classes of maps in $\C$, the subcategory $I \dashcell \subset \C$ of \bit{relative $I$-cell complexes} is the subcategory of maps that can be constructed as transfinite compositions of pushouts of elements of $I$.  An object is called an \bit{$I$-cell complex} if its unique map from $\es_\C$ is a relative $I$-cell complex.  Note that $(I \dashcell) \dashinj = I \dashinj$ and that $I \dashcell \subset I \dashcof$.
\end{defn}

\begin{defn}
Given a cardinal $\kappa$, an object $x \in \C$ is called \bit{$\kappa$-small relative to $I$} if for every regular cardinal $\lambda \geq \kappa$ and every $\lambda$-sequence $\{y_\beta\}_{\beta < \lambda}$ of relative $I$-cell complexes, $\colim_{\beta < \lambda} \hom_\C(x,y_\beta) \xra{\sim} \hom_\C(x,\colim_{\beta<\lambda}y_\beta)$.  An object of $\C$ is called \bit{small relative to $I$} if it is $\kappa$-small relative to $I$ for some cardinal $\kappa$.  An object of $\C$ is called \bit{$\kappa$-small} if it is $\kappa$-small relative to $\C$, and is called \bit{small} if it is $\kappa$-small for some cardinal $\kappa$.
\end{defn}

\begin{defn}
We say that a set $I$ of homotopy classes of maps in $\C$ \bit{permits the small object argument} if the sources of its elements are small relative to $I$.
\end{defn}

We now come to the key result on which the theory of cofibrantly generated model $\infty$-categories rests, which we refer to as the \bit{small object argument}.

\begin{prop}\label{small object argument}
Suppose that $\C$ is an $\infty$-category that admits pushouts and sequential colimits, and suppose that $I$ is a set of homotopy classes of maps in $\C$ which permits the small object argument.  Then every map in $\C$ admits a factorization into a relative $I$-cell complex followed by an $I$-injective.
\end{prop}

\begin{proof}
The proof runs identically to that of \cite[Proposition 10.5.16]{Hirsch}, except that we take a coproduct over \textit{homotopy classes} of commutative squares and we choose arbitrary representatives for these classes when forming the pushout.  (See also \cite[Proposition 1.4.7]{LurieDAGX}.)
\end{proof}

\begin{rem}\label{why unenriched soa}
Although the above definitions are analogous to the classical ones, there is one wrinkle that appears in the $\infty$-categorical case.  Namely, the classical small object argument is visibly functorial: one simply takes a coproduct over the \textit{set} of commutative squares (and no choices of representatives of homotopy classes is necessary).  In an $\infty$-category, however, we instead have a \textit{space} of commutative squares.  Thus, it would be more natural in some respects for us to instead carry out an ``$\S$-enriched small object argument'', which would then be similarly functorial.  However, this has two drawbacks for us.

First of all, to do so would shrink the right class of the associated weak factorization system.  (Indeed, in enriched category theory, passing from an unenriched lifting condition to an enriched lifting condition is equivalent to closing up the given set of maps under tensors with the enriching category.)  We will be making much use of fibrations and acyclic fibrations (for instance in the model $\infty$-category $s\S_\KQ$ of \cref{define kan--quillen model structure on sspaces}), and so it is in our best interest to keep these classes as large as possible; given that the stronger statement holds (i.e.\! that in our cases of interest, it suffices to check an unenriched lifting condition), it seems reasonable to incorporate it into the theory.

More crucially, however, a key feature of the resolution model structure is that, given a model $\infty$-category $\M$ with chosen set of generators $\G$, we will obtain a cofibrant replacement of an object $X \in \M$ by an object $Y_\bullet \in s\M_\res^c$ which is given in each level as a \textit{coproduct} of elements of $\G$ (see \cite[3.3 and 5.5]{DKS-E2}, Lemma T.5.5.8.13, and Proposition T.5.5.8.10(5)), as opposed to some more elaborate construction involving tensorings with spaces.  This is useful because, denoting by $F : \M \ra \A$ our topology-to-algebra functor of interest (for instance, a homology theory $E_*$), it affords us an approximation of $F(X) \in \A$ by $F^\lw(Y_\bullet) \in s\A$, which is only useful inasmuch as it is algebraically accessible (for instance, a levelwise-projective simplicial $E_*$-module).  If we were to use an $\S$-enriched small object argument to construct the cofibrantly generated resolution model structure, we would cease to have any control whatsoever over the value of the functor $F^\lw : s\M_\res \ra s\A$ on these cofibrant replacements.
\end{rem}

We now come to the main definition of this section.

\begin{defn}\label{define cofgen}
A \bit{cofibrantly generated model $\infty$-category} is a model $\infty$-category $\M$ such that there exist sets of homotopy classes of maps $I$ and $J$, respectively called the \bit{generating cofibrations} and the \bit{generating acyclic cofibrations}, both permitting the small object argument, such that $\bW \cap \bF = \rlp(I)$ and $\bF = \rlp(J)$.  (It follows from \cref{obvious lifting criteria} that also $\bC = I \dashcof$ and $\bW \cap \bC = J \dashcof$.)
\end{defn}

\begin{ex}
Let $\S_\triv$ denote the trivial model structure of \cref{trivial model structure} on the $\infty$-category $\S$ of spaces.  Combined with a few basic observations coming from the theory of CW complexes, \cref{detect equivalences of spaces} shows that $\S_\triv$ is cofibrantly generated by the sets $I^\S_\triv = \{ S^{n-1} \ra \pt_\S \}_{n\geq 0}$ and $J^\S_\triv = \{ \id_{\es_\S} \}$.
\end{ex}

\begin{ex}\label{pi-0 model structure on spaces is almost cofgen}
Recall the model structure $\S_{\pi_0}$ on the $\infty$-category $\S$ of \cref{pi-0 model structure on spaces}: $\bW_{\pi_0}$ is created by $\pi_0 : \S \ra \Set$, $\bC_{\pi_0} = \S$, and $\bF_{\pi_0} = \rlp(\bW_{\pi_0})$.  This model structure is not cofibrantly generated (or at least, not obviously so), but we nevertheless have that
\[ \bF_{\pi_0} = \rlp((I^\S_\triv)_{\geq 2}) \cap \rlp'(\{S^0 \ra \pt_\S \}), \]
where
\begin{itemizesmall}
\item we define
\[ (I^\S_\triv)_{\geq 2} = \{S^{n-1} \ra \pt_\S \}_{n \geq 2} = \{ S^1 \ra \pt_\S , S^2 \ra \pt_\S , \ldots\}, \] and
\item by $\rlp'(\{S^0 \ra \pt_\S \})$, we mean to only require a lift in those commutative squares for which the upper map factors through the terminal map $S^0 \ra \pt_\S$.
\end{itemizesmall}
This observation generalizes to the model $\infty$-category $\S_{\tau_{\leq n}}$ of \cref{n-truncation model structure on spaces}, as indicated there.
\end{ex}

We have the following \bit{recognition theorem} for cofibrantly generated model $\infty$-categories.

\begin{thm}\label{recognize cofgen}
Let $\M$ be an $\infty$-category which is both cocomplete and finitely complete, and let $\bW \subset \M$ be a subcategory which is closed under retracts and satisfies the two-out-of-three property.  Suppose $I$ and $J$ are sets of homotopy classes of maps in $\M$, both permitting the small object argument, such that
\begin{enumeratesmall}
\item\label{recognition theorem 1} $J \dashcof \subset (I \dashcof \cap \bW)$,
\item\label{recognition theorem 2} $I \dashinj \subset (J \dashinj \cap \bW)$, and
\item\label{recognition theorem 3} either
\begin{enumeratesmallsub}
\item\label{recognition theorem 3a} $(I \dashcof \cap \bW) \subset J\dashcof$ or
\item\label{recognition theorem 3b} $(J \dashinj \cap \bW) \subset I\dashinj$.
\end{enumeratesmallsub}
\end{enumeratesmall}
Then the sets $I$ and $J$ define a cofibrantly generated model structure on $\M$ whose weak equivalences are $\bW$.
\end{thm}

\begin{proof}
With the small object argument in hand, the proof runs identically to that of \cite[Theorem 11.3.1]{Hirsch}.
\end{proof}

We also have the following \bit{lifting theorem} for cofibrantly generated model $\infty$-categories.

\begin{thm}\label{lift cofgen}
Let $\M$ be a cofibrantly generated model $\infty$-category with generating cofibrations $I$ and generating acyclic cofibrations $J$, and let $F:\M \adjarr \N:G$ be an adjunction with $\N$ finitely bicomplete.  If $FI$ and $FJ$ both permit the small object argument and $G$ takes relative $FJ$-cell complexes into $\bW_\M$, then $FI$ and $FJ$ define a cofibrantly generated model structure on $\N$ in which $\bW_\N$ is created by $G$.  Moreover, with respect to this lifted model structure, the adjunction $F \adj G$ becomes a Quillen adjunction.
\end{thm}

\begin{proof}
The proof runs identically to that of \cite[Theorem 11.3.2]{Hirsch}.
\end{proof}

\section{The definition of the Kan--Quillen model structure}\label{section define kan--quillen model structure}

We are now in a position to state the main result of this paper (\cref{kan--quillen model structure on sspaces}), which gives a systematic way of manipulating simplicial spaces in their capacity as ``presentations of spaces'' via the geometric realization functor.  This sits in precise analogy with the 1-category of simplicial sets, and so we begin with the following recollection.

\begin{defn}\label{define kan--quillen model structure on ssets}
The \bit{Kan--Quillen model structure} on $s\Set$, denoted $s\Set_\KQ$, is the proper model structure which is cofibrantly generated by the sets $I^{s\Set}_\KQ = \{ \partial \Delta^n \ra \Delta^n \}_{n \geq 0}$ and $J^{s\Set}_\KQ = \{ \Lambda^n_i \ra \Delta^n \}_{0 \leq i \leq n \geq 1}$ (see e.g.\! \cite[Example 11.1.6 and Theorem 13.1.13]{Hirsch}).
\end{defn}

In order to be precise regarding model $\infty$-categories (and in case the reader has forgotten the classical definition), we also give the following.

\begin{defn}
A model $\infty$-category is called \bit{left proper} if its weak equivalences are preserved under pushout along cofibrations, and dually is called \bit{right proper} if its weak equivalences are preserved under pullback along fibrations.  A model $\infty$-category is called \bit{proper} if it is both left proper and right proper.
\end{defn}

The categories of simplicial sets and simplicial spaces are related in the following way.

\begin{notn}\label{adjunction between sspaces and ssets}
Recall the adjunction $\pi_0 : \S \adjarr \Set : \disc$.  Applying $\Fun(\bD^{op},-)$, this induces an adjunction which we again denote by $\pi_0 : s\S \adjarr s\Set : \disc$.  We will generally omit this right adjoint from the notation unless we mean to emphasize it.
\end{notn}

Using this terminology and notation, we can now state the main result of this paper.

\begin{thm}\label{kan--quillen model structure on sspaces}
The sets $I^{s\S}_\KQ = \disc(I^{s\Set}_\KQ)$ and $J^{s\S}_\KQ = \disc(J^{s\Set}_\KQ)$ define a proper, cofibrantly generated model structure on $s\S$, in which the weak equivalences are created by the geometric realization functor $|{-}|:s\S \ra \S$.
\end{thm}

\begin{proof}
We appeal to \ref{recognize cofgen}, verifying
\begin{itemizesmall}
\item condition \ref{recognition theorem 1} in \cref{detect acyclic cofibrations of sspaces},
\item condition \ref{recognition theorem 2} in \cref{detect acyclic fibrations of sspaces},
\item condition \ref{recognition theorem 3}\ref{recognition theorem 3b} in \cref{acyclic fibrations have rlp(I)}, and
\item that $I^{s\S}_\KQ$ and $J^{s\S}_\KQ$ both permit the small object argument in \cref{permit the small object argument}.
\end{itemizesmall}
The weak equivalences are closed under retracts and satisfy the two-out-of-three property because they are pulled back from a class of equivalences.  Lastly, left properness is immediate since the weak equivalences are created by a left adjoint (which commutes with pushouts), and right properness is proved as \cref{kan--quillen model structure is right proper}.
\end{proof}

\begin{defn}\label{define kan--quillen model structure on sspaces}
In analogy with \cref{define kan--quillen model structure on ssets}, we also refer to the model structure on $s\S$ defined by \cref{kan--quillen model structure on sspaces} as the \bit{Kan--Quillen model structure}, denoted $s\S_\KQ$.
\end{defn}

\begin{rem}\label{one of the most useful consequences of right properness}
In the theory of model 1-categories, one of the most useful consequences of right properness is that a pullback in which just one of the maps is a fibration is already a homotopy pullback (and dually for left properness).  This remains true in the theory of model $\infty$-categories; with the theory of homotopy co/limits in model $\infty$-categories in hand (see \cref{qadjns:subsection ho-co-lims}), the proof runs essentially identically (see e.g.\! \cite[\sec 13.3]{Hirsch}).  However, for the sake of self-containment, we will also directly prove this consequence of the right properness of $s\S_\KQ$ as \cref{consequence of right properness of kan--quillen model structure}.  (In fact, we use this as an \textit{input} to the proof of right properness in \cref{kan--quillen model structure is right proper}.)  On the other hand, the dual corollary of the left properness of $s\S_\KQ$ is as trivial to verify as the left properness of $s\S_\KQ$ itself.
\end{rem}

\begin{rem}
In the theory of model 1-categories, there are many other adjectives that one might attach to a model structure: combinatorial, cellular, tractable, etc.  The model $\infty$-category $s\S_\KQ$ enjoys completely analogous properties to those enjoyed by $s\Set_\KQ$.  However, we won't need these observations for now, so we just leave them here as a remark.  (The model $\infty$-category $s\S_\KQ$ will also be a \textit{simplicio-spatial model $\infty$-category}, the analogous notion to that of a simplicial model 1-category (see \cref{qadjns:defn sspatial model str}).  (This is simply to say that it is a \textit{symmetric monoidal model $\infty$-category} (see \cref{qadjns:defn of symm monoidal model infty-cat} and \cref{qadjns:ex KQ is a monoidal model str}).))
\end{rem}

\begin{rem}
Note that if we apply the small object argument for $I^{s\S}_\KQ$ or $J^{s\S}_\KQ$ to a map in $s\S$ whose source is in $s\Set$, then the intermediate object will also be in $s\Set$.  In particular, by the usual transfinite induction arguments, $\rlp(I^{s\S}_\KQ) = \rlp(\disc(\bC^{s\Set}_\KQ))$ and $\rlp(J^{s\S}_\KQ) = \rlp(\disc((\bW \cap \bC)^{s\Set}_\KQ))$.   We will use these facts without further comment.
\end{rem}

\begin{rem}\label{rem could lift KQ but circular}
The adjunction $\pi_0 : s\S \adjarr s\Set : \disc$ could be used to lift the cofibrantly generated Kan--Quillen model structure on $s\S$ to the one on $s\Set$ via the lifting theorem (\ref{lift cofgen}), except of course that that would be totally circular: the construction of $s\S_\KQ$ takes the existence of $s\Set_\KQ$ as input.  As observed in \cref{quillen equivalence between kan--quillen model structures}, this adjunction is even a Quillen equivalence, whose derived adjunction is the identity adjunction on their common localization $\S$.
\end{rem}

\begin{rem}\label{rem sSet-KQ is model subcat of sS-KQ}
In fact, extending \cref{rem could lift KQ but circular}, note that all three defining subcategories of the model category $s\Set_\KQ$ are pulled back from the corresponding defining subcategories of the model $\infty$-category $s\S_\KQ$ along the inclusion $s\Set \subset s\S$.
\end{rem}

To codify the observation of \cref{rem sSet-KQ is model subcat of sS-KQ}, we introduce the following.

\begin{defn}\label{defn model subcat}
Let $\M$ be a model $\infty$-category.  We say that a model structure on a subcategory $\N \subset \M$ defines a \bit{model subcategory} of $\M$ if we have
\begin{itemizesmall}
\item $\bW_\N = \N \cap \bW_\M \subset \M$,
\item $\bC_\N = \N \cap \bC_\M \subset \M$, and
\item $\bF_\N = \N \cap \bF_\M \subset \M$.
\end{itemizesmall}
\end{defn}

\begin{notn}
As $s\Set_\KQ$ is a model subcategory of $s\S_\KQ$ by \cref{rem sSet-KQ is model subcat of sS-KQ}, it will usually be unambiguous to omit the superscripts $s\Set$ and $s\S$ from their defining subcategories: we will usually just write $\bW_\KQ$, $\bC_\KQ$, or $\bF_\KQ$, leaving the ambient $\infty$-category implicit unless we mean to draw specific attention to it, as promised in \cref{notn for decorating model structure data}.  (Of course, since the purpose of this paper is to construct the model $\infty$-category $s\S_\KQ$, one would be safest to assume that we mean to refer to the subcategories of $s\S$ instead of those of $s\Set$ when no superscript is included.)  Similarly, we will henceforth generally simply denote
\begin{itemizesmall}
\item both $I^{s\Set}_\KQ$ and $I^{s\S}_\KQ$ by $I_\KQ$, and
\item both $J^{s\Set}_\KQ$ and $J^{s\S}_\KQ$ by $J_\KQ$.
\end{itemizesmall}
\end{notn}

\begin{notn}
In the course of proving \cref{kan--quillen model structure on sspaces}, it will be important to take care in distinguishing which facts have been proved and which facts have not.  Otherwise, our arguments might appear to be circular (for example as discussed in \cref{circularity of potential proof of detect acyclic fibrations of sspaces}).  Thus, in order to be totally clear about this distinction, we take the following conventions regarding maps of simplicial sets and of simplicial spaces.
\begin{itemize}
\item We will only decorate our arrows if the corresponding property is relevant to the argument.
\item We will only use the decoration $\we$ if we've actually proved that the map becomes a weak equivalence upon geometric realization.
\item When working in $s\Set_\KQ \subset s\S_\KQ$, we will use all standard results and notation.
\item By the nature of the arguments, the only cofibrations in $s\S_\KQ$ that appear will actually lie in the subcategory $s\Set_\KQ$.  On the other hand, rather than write $\fibn$ or $\wfibn$ for the indicated maps in $s\S_\KQ$, we will instead label the arrows with their lifting properties (so $\rlp(J_\KQ)$ or $\rlp(I_\KQ)$, respectively).
\item For convenience, we will still write $s\S^f_\KQ$ for those objects whose terminal map has $\rlp(J_\KQ)$.  (On the other hand, we will simply have $s\S^c_\KQ = s\Set \subset s\S$.)
\end{itemize}
\end{notn}

\section{Auxiliary results on spaces and simplicial spaces}\label{section auxiliary results}

In this section, we collect some auxiliary results regarding spaces and simplicial spaces which will be necessary for our proof of \cref{kan--quillen model structure on sspaces} (the bulk of which will be given in \cref{section proof of kan--quillen model structure}).

We begin with an easy folkloric result, which gives a criterion for a map of spaces to be an equivalence.

\begin{lem}\label{detect equivalences of spaces}
Let $Y \xra{\varphi} Z$ be a map in $\S$.  Then $\varphi$ is $n$-connected iff it has $\rlp(\{S^{i-1} \ra \pt_\S\}_{0 \leq i \leq n})$.  In particular, $\varphi$ is an equivalence iff it has $\rlp(\{S^{n-1} \ra \pt_\S\}_{n\geq 0})$.
\end{lem}

\begin{proof}
First, note that if $Y \ra Z$ has $\rlp(\{S^{i-1} \ra \pt_\S\})$, then the map $[S^{i-1},Y]_\S \ra [S^{i-1},Z]_\S$ is an inclusion.  Since a map off of $S^{i-1}$ is basedly nullhomotopic iff it's freely so, this means that for any basepoint $y \in Y$, the map $\pi_{i-1}(Y,y) \ra \pi_{i-1}(Z,\varphi(y))$ is an inclusion.

On the other hand, considering the map $S^{i-1} \ra \pt_\S$ as the standard inclusion $S^{i-1} \ra D^i$, we see that if we begin with the constant map $S^{i-1} \ra Y$ at some point $y \in Y$, then an extension of the composite $S^{i-1} \ra Y \ra Z$ over $S^{i-1} \ra D^i$ really just selects a map $S^i \ra Z$ which is based at $\varphi(y)$.  So, if $Y \ra Z$ has $\rlp(\{S^{i-1} \ra \pt_\S\})$, then also $\pi_i(Y,y) \ra \pi_i(Z,\varphi(y))$ is a surjection for any basepoint $y \in Y$.

Combining these two consequences of $Y \ra Z$ having $\rlp(\{S^{i-1} \ra \pt_\S\})$ proves both claims.
\end{proof}

We now turn to simplicial spaces.  We begin with the necessary smallness results.

\begin{lem}\label{finite ssets are small sspaces}
An object of $s\Set$ with finitely many nondegenerate simplices is $\omega$-small as an object of $s\S$.
\end{lem}

\begin{proof}
This follows from the fact that finite sets are small in $\S$ and from the inductive definition of a map of simplicial objects.
\end{proof}

\begin{cor}\label{permit the small object argument}
The sets $I_\KQ$ and $J_\KQ$ permit the small object argument.
\end{cor}

\begin{proof}
This follows from \cref{finite ssets are small sspaces}.
\end{proof}

Many of the remaining results of this paper will rely critically on the following one.  Its proof is rather involved, and so we postpone it to \cref{section proof of crazy lemma}.  Roughly speaking, this result concerns the inductive construction of a presentation of a map in $\S$ by a map in $s\S$ whose source lies in $s\Set \subset s\S$.

\begin{lem}\label{crazy lemma}
Suppose we are given any $W \in s\S$, any $K \in s\Set$, and any pushout
\[ \begin{tikzcd}
\partial \Delta^n \arrow{r} \arrow{d} & K \arrow{d} \\
\Delta^n \arrow{r} & L
\end{tikzcd} \]
in $s\Set$.  Suppose further that we are given any point of the pullback
\[ \lim \left( \begin{tikzcd}
& \hom_{s\S}(K,W) \arrow{d} \\
\hom_\S(|L|,|W|) \arrow{r} & \hom_\S(|K|,|W|)
\end{tikzcd} \right) . \]
Then there exists some $i \geq 0$ such that if the front square in the cube
\[ \begin{tikzcd}
& \partial \Delta^n \arrow{rr} \arrow[tail]{dd} & & K \arrow{dd} \\
\sd^i(\partial \Delta^n) \arrow{ru}[sloped, pos=0.6]{\approx} \arrow[crossing over]{rr} \arrow[tail]{dd} & & K \arrow[equals]{ru} \\
& \Delta^n \arrow{rr} & & L \\
\sd^i(\Delta^n) \arrow{ru}[sloped, pos=0.6]{\approx} \arrow{rr} & & L' \arrow[leftarrow, crossing over]{uu} \arrow{ru}
\end{tikzcd} \]
is also a pushout in $s\Set$, then $L' \we L$ in $s\Set_\KQ$ and the map
\[ \hom_{s\S}(L',W) \ra \lim \left( \begin{tikzcd}
& & \hom_{s\S}(K,W) \arrow{d} \\
\hom_\S(|L'|,|W|) & \hom_\S(|L|,|W|) \arrow{r} \arrow{l}[swap]{\sim} & \hom_\S(|K|,|W|)
\end{tikzcd} \right) \]
in $\S$ is surjective onto the chosen point.
\end{lem}

\begin{rem}\label{crazy lemma is not as strong as one might hope}
\cref{crazy lemma} is not as strong as one might hope.  First of all, it would be nice if the last map in its statement were actually a surjection, but to deduce this we would need to be able to bound the number $i$ as we run through the path components of the pullback, which does not appear to be possible.  But even more seriously, if instead we have a cofibration $K \cofibn L$ in $s\Set_\KQ$ which can only be obtained through multiple pushouts of maps in $I_\KQ$, then \cref{crazy lemma} cannot be made to guarantee the existence of an extension $L' \ra W$ in $s\S$ (for some $L' \we L$ in $s\Set_\KQ$) modeling the chosen extension $|L| \ra |W|$ in $\S$.

For instance, suppose that $L = \Delta^1 \times \Delta^1$, and that the map $K \ra L$ is the inclusion of its boundary (so that $K$ is a simplicial square, and the map $K \ra L$ in $s\Set_\KQ$ presents the map $S^1 \ra \pt_\S$ in $\S$).  The minimal way to present this as a composition of pushouts of maps in $I_\KQ$ is as
\[ K \cofibn M \cofibn N \cofibn L , \]
where $M$ is the 1-skeleton of $L$ and the latter two maps are each obtained by attaching a 2-simplex.  However, when we attempt to extend our given map $K \ra W$ along $K \ra M$ using \cref{crazy lemma}, we may need to subdivide the 1-simplex that we're attaching, and so we only obtain an extension $M' \ra W$ in $s\S$ for some factorization $K \ra M' \we M$ in $s\Set_\KQ$.  If \textit{any} such subdivisions are required, then the two remaining holes to be filled in $M'$ will now have at least four edges each, and so we are no closer to ``filling the hole'' than when we started.

On the other hand, in the case that the map $K \ra L$ is $\es_{s\Set} \ra \Delta^0$, then no subdivisions are required (indeed, subdivision preserves both $\es_{s\Set}$ and $\Delta^0$, or alternatively we can see this from the fact that the map $W_0 \ra |W|$ is a surjection).  Thus, if we are given any $W \in s\S$ and any map $S^n \ra |W|$ in $\S$, we can first present the composite $\pt_\S \ra S^n \ra |W|$ in $\S$ by some map $\Delta^0 \ra W$ in $s\S$, and then taking $K = \Delta^0$ and forming the pushout
\[ \begin{tikzcd}
\partial \Delta^n \arrow{r} \arrow{d} & \Delta^0 \arrow{d} \\
\Delta^n \arrow{r} & L
\end{tikzcd} \]
in $s\Set$, we are guaranteed a factorization $\Delta^0 \ra L' \we L$ in $s\Set_\KQ$ and a map $L' \ra W$ in $s\S$ presenting the chosen map $|L| \simeq S^n \ra |W|$ in $\S$.  Since so many arguments in $\S$ go by considering arbitrary maps into a space from a sphere (e.g.\! recall \cref{detect equivalences of spaces}), the existence of such a minimal model $\Delta^n / \partial \Delta^n \in s\Set_\KQ$ for the object $S^n \in \S$ seems like a real stroke of luck.

In any case, we do not expect \cref{crazy lemma} to be particularly useful in the long run: it is effectively supplanted by the fundamental theorem of model $\infty$-categories (\Cref{fundthm:fundamental theorem}).  (In particular, see \cref{fibrant sspaces are fibrant} (and \cref{special case of fibrant sspaces are fibrant}).)
\end{rem}

\section{Fibrancy, fibrations, and the $\Ex^\infty$ functor}\label{section fibrations}

In this section, we undertake a study of
\begin{itemizesmall}
\item the subcategory $s\S_\KQ^f \subset s\S_\KQ$ of fibrant objects in \cref{subsection fibrancy in KQ model str on sspaces},
\item the subcategory $\bF_\KQ \subset s\S_\KQ$ of fibrations in \cref{subsection fibrations in KQ model str}, and
\item an $\Ex^\infty$ endofunctor on $s\S$ in \cref{subsection Ex-infty}.
\end{itemizesmall}
As we will see, all of these behave quite analogously to their classical counterparts in $s\Set_\KQ$.

\subsection{Fibrancy}\label{subsection fibrancy in KQ model str on sspaces}

We begin by studying fibrant objects.  First of all, \cref{crazy lemma} admits a much cleaner analog when $W \in s\S_\KQ$ is fibrant.

\begin{lem}\label{reduce crazy lemma for fibrant W}
In \cref{crazy lemma}, if $W \in s\S_\KQ^f$ then we may take $i=0$.  Hence, the map
\[ \hom_{s\S}(L,W) \ra \lim \left( \begin{tikzcd}
& \hom_{s\S}(K,W) \arrow{d} \\
\hom_\S(|L|,|W|) \arrow{r} & \hom_\S(|K|,|W|)
\end{tikzcd} \right) \]
is a surjection.
\end{lem}

\begin{proof}
We will argue using the diagram in $s\Set_\KQ$ shown in \cref{diagram in ssets for the proof that reduces the crazy lemma for fibrant W}, in which some of the objects and morphisms have yet to be constructed.
\begin{figure}[h]
\[ \begin{tikzcd}
& & & \Delta^n \arrow{rr} \arrow[tail]{dd}[sloped, anchor=south, pos=0.65]{\approx} & & L \arrow[tail]{dd}[sloped, anchor=south]{\approx} \\
\sd^i(\partial \Delta^n) \arrow{rr}{\approx} \arrow[tail]{dd} & & \partial \Delta^n \arrow[tail]{ru} \arrow[crossing over]{rr} \arrow[tail]{dd} & & K \arrow[tail]{ru} \\
& & & M' \arrow{rr} & & L'' \\
\sd^i(\Delta^n) \arrow{rr}[swap]{\approx} & & M \arrow{rr} \arrow[tail]{ru}[sloped, swap, pos=0.4]{\approx} & & L' \arrow[leftarrowtail, crossing over]{uu} \arrow[tail]{ru}[sloped, swap, pos=0.4]{\approx}
\end{tikzcd} \]
\caption{The diagram in $s\Set_\KQ$ used in the proof of \cref{reduce crazy lemma for fibrant W}.}\label{diagram in ssets for the proof that reduces the crazy lemma for fibrant W}
\end{figure}
First of all, recall that the top square of the cube is a pushout by the definition of $L$.  Also, observe that we can also build $L'$ via the iterated pushout in the front two squares, where we have $\sd^i(\Delta^n) \we M$ since $s\Set_\KQ$ is left proper.  Next, choose any acyclic object $M' \in s\Set_\KQ$ admitting a cofibration from $M \coprod_{\partial \Delta^n} \Delta^n$, and use it to form the left face of the cube.  (The maps from $M$ and $\Delta^n$ to this pushout are cofibrations, which is why the maps from $M$ and $\Delta^n$ to $M'$ are also cofibrations.)  Then, define $L''$ by declaring that the bottom square of the cube is a pushout; by an easy diagram chase, the back square of the cube is therefore a pushout as well.

Now, since by assumption $W \ra \pt_{s\S}$ has $\rlp(J_\KQ)$, we are guaranteed an extension
\[ \begin{tikzcd}
K \arrow{rr} \arrow{rd} \arrow{dd} & & W \\
& L' \arrow{ru} \arrow[tail]{d}[sloped, anchor=north]{\approx} \\
L \arrow{r}[swap]{\approx} & L'' \arrow[dashed, bend right=10]{ruu}
\end{tikzcd} \]
in $s\S$, and the composite map $L \ra L'' \ra W$ satisfies the same hypotheses as were required of the map $L' \ra W$.  Thus, we may take $i=0$, as claimed.

Carrying out this same argument for all path components of the pullback implies that the indicated map is indeed a surjection.
\end{proof}

\begin{cor}\label{fibrant sspaces are fibrant}
If $W \in s\S_\KQ^f$, then for any $K \cofibn L$ in $s\Set_\KQ$, the map
\[ \hom_{s\S}(L,W) \ra \lim \left( \begin{tikzcd}
& \hom_{s\S}(K,W) \arrow{d} \\
\hom_\S(|L|,|W|) \arrow{r} & \hom_\S(|K|,|W|)
\end{tikzcd} \right) \]
is a surjection.
\end{cor}

\begin{proof}
First, we present the map $K \cofibn L$ as a transfinite composition of pushouts of maps in $I_\KQ$.  Then, the result follows by transfinite induction, applying \cref{reduce crazy lemma for fibrant W} at each successor ordinal and using the universal property of the colimit (which here is a colimit both in $s\Set$ and in $s\S$) at each limit ordinal.
\end{proof}

\begin{rem}\label{special case of fibrant sspaces are fibrant}
In the special case that $K = \es_{s\Set}$, \cref{fibrant sspaces are fibrant} reduces to the statement that for any $L \in s\Set = s\S^c_\KQ$ and any $W \in s\S^f_\KQ$, the map $\hom_{s\S}(L , W) \ra \hom_\S(|L|,|W|)$ is a surjection.  This is a hint of the fundamental theorem of model $\infty$-categories (\Cref{fundthm:fundamental theorem}) as applied to $s\S_\KQ$ (recall \cref{remark fundamental theorem}).
\end{rem}

\begin{rem}
\cref{fibrant sspaces are fibrant} provides a basis for \cite[30.10]{BEBdBP}, which gives a complete characterization of the subcategory $\bW_\KQ^f \subset s\S$ of weak equivalences between fibrant objects (and which is in turn the crucial ingredient of that paper).
\end{rem}

\subsection{Fibrations}\label{subsection fibrations in KQ model str}

We now turn from fibrant objects to fibrations:
\begin{itemizesmall}
\item in \cref{subsubsection fibns and geom realizns} we lay out some general results on the interplay between fibrations and geometric realizations,
\item in \cref{subsubsection necessary condition for fibn} we show that the left Quillen equivalence $\pi_0 : s\S_\KQ \ra s\Set_\KQ$ preserves fibrations, and
\item in \cref{subsubsection fibration comparisons} we give some comparisons with existing literature.
\end{itemizesmall}

\subsubsection{Fibrations and geometric realizations}\label{subsubsection fibns and geom realizns}

The following result is crucial, and provides a basis for many of the convenient properties enjoyed by the model $\infty$-category $s\S_\KQ$.  Its proof is relatively straightforward, though somewhat long (although not as long as it looks, since it contains so many diagrams).

\begin{prop}\label{the fiber of a fibration of sspaces is homotopically correct}
Suppose the map $Y \ra Z$ in $s\S$ has $\rlp(J_\KQ)$, and suppose we are given any point $\pt_{s\S} \xra{z} Z$.  Let $F_z \in s\S$ be the fiber of $Y \ra Z$ over $z$, and let $F_{|z|} \in \S$ be the fiber of $|Y| \ra |Z|$ over $|z|$.  Then the natural map $|F_z| \ra F_{|z|}$ is an equivalence in $\S$.
\end{prop}

\begin{proof}
We use the criterion of \cref{detect equivalences of spaces}.  So, suppose that
\[ \begin{tikzcd}
S^{n-1} \arrow{r} \arrow{d} & |F_z| \arrow{d} \\
\pt_\S \arrow{r} & F_{|z|}
\end{tikzcd} \]
is any commutative square in $\S$, for any $n \geq 0$.  Since $F_z \ra \pt_{s\S}$ also has $\rlp(J_\KQ)$ as this property is closed under pullbacks, by \cref{fibrant sspaces are fibrant} we may present the upper map in the above diagram as a map $\partial \Delta^n \ra F_z$ in $s\S$.

From here, our argument will play back and forth between the diagrams shown in Figures \ref{diagram in sspaces for the proof that the fiber of a fibration of sspaces is homotopically correct} \and \ref{diagram in spaces for the proof that the fiber of a fibration of sspaces is homotopically correct}.  The former takes place in $s\S_\KQ$, while the latter takes place in $\S$; in both, many of the objects (and all of the dotted arrows) have yet to be constructed.
\begin{figure}[h]
\[ \begin{tikzcd}[row sep=2cm, column sep=2cm]
& & \partial \Delta^n \arrow{r} \arrow{ld} & F_z \arrow{r} \arrow[leftarrow, dashed, bend left=10]{llddd} & Y \arrow{ddd}{\rlp(J_\KQ)} \\
& (\Delta^n)' \arrow{r}[swap]{\approx} \arrow[tail]{dl}[sloped, swap, pos=0.4]{\approx} \arrow[dashed, bend right=10]{rrru} & \Delta^n \arrow[leftarrow, crossing over]{u} \\
(\Delta^n)''' \arrow{r}{\approx} \arrow[dashed, bend left=5]{rrrrd} \arrow[dashed, bend right=15]{rrrruu} & (\Delta^n)'' \arrow[leftarrow, crossing over]{u}[sloped, anchor=north]{\approx} \\
& \Delta^0 \arrow{lu}[sloped, pos=0.4]{\approx} \arrow[crossing over]{u}[sloped, anchor=north]{\approx} \arrow{rr}[swap]{\sim} & & \pt_{s\S} \arrow{r}[swap]{z} \arrow[leftarrow, crossing over]{uuu} & Z
\end{tikzcd} \]
\caption{The diagram in $s\S_\KQ$ used in the proof of \cref{the fiber of a fibration of sspaces is homotopically correct}.}
\label{diagram in sspaces for the proof that the fiber of a fibration of sspaces is homotopically correct}
\[ \begin{tikzcd}[row sep=2cm, column sep=2cm]
& & |\partial \Delta^n| \arrow{r} \arrow{d} \arrow{ld} & |F_z| \arrow{r} \arrow{d} \arrow[dashed, leftarrow, bend left=10]{llddd} & |Y| \arrow{ddd} \\
& |(\Delta^n)'| \arrow{r}[swap]{\sim} \arrow{d}[sloped, anchor=south]{\sim} \arrow{dl}[sloped, swap, pos=0.4]{\sim} & |\Delta^n| \arrow[crossing over]{r} & F_{|z|} \arrow{ru} \\
|(\Delta^n)'''| \arrow{r}{\sim} & |(\Delta^n)''| \arrow[dashed, bend left=10]{rrrd} \\
& |\Delta^0| \arrow{lu}[sloped, pos=0.4]{\sim} \arrow{u}[sloped, anchor=north]{\sim} \arrow{rr}[swap]{\sim} & & |\pt_{s\S}| \arrow{r}[swap]{|z|} \arrow[leftarrow, crossing over]{uu} & |Z| ,
\end{tikzcd} \]
\caption{The diagram in $\S$ used in the proof of \cref{the fiber of a fibration of sspaces is homotopically correct}.}
\label{diagram in spaces for the proof that the fiber of a fibration of sspaces is homotopically correct}
\end{figure}
For clarity, we proceed in steps.

\begin{enumerate}

\item 

Given the composite map $\partial \Delta^n \ra F_z \ra Y$ in $s\S$ and its chosen extension
\[ \begin{tikzcd}
|\partial \Delta^n| \arrow{r} \arrow{d} & |F_z| \arrow{r} \arrow{d} & |Y| \\
|\Delta^n| \arrow{r} & F_{|z|} \arrow{ru}
\end{tikzcd} \]
in $\S$, by \cref{crazy lemma} there exists a factorization $\partial \Delta^n \ra (\Delta^n)' \we \Delta^n$ in $s\Set_\KQ$ and a dotted arrow $(\Delta^n)' \ra Y$ as in \cref{diagram in sspaces for the proof that the fiber of a fibration of sspaces is homotopically correct} which models this extension in $\S$.

\item

For expository convenience, we consider the object $\Delta^0 \in s\S$ with its unique map $\Delta^0 \xra{\sim} \pt_{s\S}$ as selecting a composite map $|\Delta^0| \xra{\sim} |\pt_{s\S}| \xra{|z|} |Z|$.

\item

Choose any vertex of $(\Delta^n)'$, and use this to define $(\Delta^n)'' \in s\Set$ by the pushout diagram
\[ \begin{tikzcd}
\partial \Delta^1 \arrow{r}{\approx} \arrow[tail]{d} & (\Delta^n)' \sqcup \Delta^0 \arrow[tail]{d} \\
\Delta^1 \arrow{r}[swap]{\approx} & (\Delta^n)''
\end{tikzcd} \]
in $s\Set_\KQ$.  Observe that both induced maps $\Delta^0 \ra (\Delta^n)''$ and $(\Delta^n)' \ra (\Delta^n)''$ are in $\bW_\KQ$.

\item\label{item subdiagram in spaces}

Now, we have the solid commutative diagram in $\S$ of \cref{subdiagram in spaces for the proof that the fiber of a fibration of sspaces is homotopically correct},
\begin{figure}[h]
\[ \begin{tikzcd}
& & & |Y| \arrow{ddd} \\
|(\Delta^n)'| \arrow{r}{\sim} \arrow{d}[sloped, anchor=north]{\sim} & |\Delta^n| \arrow{r} & F_{|z|} \arrow{ru} \\
|(\Delta^n)''| \arrow[dashed, bend left=10]{rrrd} \\
|\Delta^0| \arrow{u}[sloped, anchor=south]{\sim} \arrow{rr}[swap]{\sim} & & |\pt_{s\S}| \arrow{r}[swap]{|z|} \arrow[leftarrow, crossing over]{uu} & |Z|
\end{tikzcd} \]
\caption{The subdiagram of the diagram in $\S$ of \cref{diagram in spaces for the proof that the fiber of a fibration of sspaces is homotopically correct} used in part \cref{item subdiagram in spaces} of the proof of \cref{the fiber of a fibration of sspaces is homotopically correct}.}
\label{subdiagram in spaces for the proof that the fiber of a fibration of sspaces is homotopically correct}
\end{figure}
and hence we can obtain a dotted arrow $|(\Delta^n)''| \ra |Z|$ therein making the entire diagram commute, as in \cref{diagram in spaces for the proof that the fiber of a fibration of sspaces is homotopically correct}.

\item

Thus, we have a map $(\Delta^n)' \sqcup \Delta^0 \ra Z$ in $s\S$ and a chosen extension
\[ \begin{tikzcd}
|(\Delta^n)'| \sqcup |\Delta^0| \arrow{r} \arrow{d} & |Z| \\
|(\Delta^n)''| \arrow[dashed]{ru}
\end{tikzcd} \]
in $\S$.  (Note that the geometric realization functor $|{-}| : s\S \ra \S$ commutes with colimits (being a left adjoint), and in particular with coproducts.)  So by \cref{crazy lemma}, there is a factorization $(\Delta^n)' \sqcup \Delta^0 \ra (\Delta^n)''' \we (\Delta^n)''$ in $s\Set_\KQ$ and a dotted arrow $(\Delta^n)''' \ra Z$ as in \cref{diagram in sspaces for the proof that the fiber of a fibration of sspaces is homotopically correct} which models this extension in $\S$.

\item

It is easy to see that that in fact, we have $(\Delta^n)' \wcofibn (\Delta^n)'''$ in $s\Set_\KQ$ (for instance because $s\Set_\KQ$ is left proper, and so the defining map to $(\Delta^n)'''$ from some subdivision of $\Delta^1$ is in $\bW_\KQ$).  Since the map $Y \ra Z$ in $s\S$ has $\rlp(J_\KQ)$, it follows that there exists a lift $(\Delta^n)''' \ra Y$ as in \cref{diagram in sspaces for the proof that the fiber of a fibration of sspaces is homotopically correct}.

\item\label{item subdiagram in sspaces}

We now have the solid commutative diagram in $s\S$ of \cref{subdiagram in sspaces for the proof that the fiber of a fibration of sspaces is homotopically correct},
\begin{figure}[h]
\[ \begin{tikzcd}
& & \partial \Delta^n \arrow{r} \arrow{ddll} & F_z \arrow{r} \arrow[dashed, leftarrow, bend left=10]{llddd} & Y \arrow{ddd} \\
\\
(\Delta^n)''' \arrow[bend right=15, crossing over]{rrrruu} \\
& \Delta^0 \arrow{lu} \arrow{rr} & & \pt_{s\S} \arrow{r}[swap]{z} \arrow[leftarrow, crossing over]{uuu} & Z
\end{tikzcd} \]
\caption{The subdiagram of the diagram in $s\S$ of \cref{diagram in sspaces for the proof that the fiber of a fibration of sspaces is homotopically correct} used in part \cref{item subdiagram in sspaces} of the proof of \cref{the fiber of a fibration of sspaces is homotopically correct}.}
\label{subdiagram in sspaces for the proof that the fiber of a fibration of sspaces is homotopically correct}
\end{figure}
and so by the universal property of the pullback we can obtain a dotted arrow $\Delta^0 \ra F_z$ making the entire diagram commute, as in \cref{diagram in sspaces for the proof that the fiber of a fibration of sspaces is homotopically correct}.

\item

Now, taking geometric realization gives the entire diagram in $\S$ of \cref{diagram in spaces for the proof that the fiber of a fibration of sspaces is homotopically correct} (including the dotted arrows).  In particular, we obtain the desired lift
\[ \begin{tikzcd}
| \partial \Delta^n| \arrow{r} \arrow{d} & |F_z| \arrow{d} \\
|\Delta^n| \arrow{r} \arrow[dashed]{ru} & F_{|z|}
\end{tikzcd} \]
in $\S$.  (It is straightforward to see that this does indeed commute, as $F_{|z|}$ is defined as a pullback.)
\qedhere

\end{enumerate}

\end{proof}

\begin{rem}\label{observe that the fiber of a fibration of sspaces is homotopically correct}
In either $s\Set$ or $s\S$, one can always take the fiber of a map over a given point in its target.  However, inasmuch as we are interested in simplicial sets and simplicial spaces as presenting spaces via geometric realization, we can view \cref{the fiber of a fibration of sspaces is homotopically correct} as saying that in either case, this fiber is only ``homotopically meaningful'' -- that is, it only computes the fiber in $\S$ -- if the original map is a \textit{fibration} in the corresponding Kan--Quillen model structure.
\end{rem}

While \cref{the fiber of a fibration of sspaces is homotopically correct} only addresses the question of when taking fibers commutes with geometric realization, we can use it to address the same question regarding more general pullbacks.

\begin{cor}\label{consequence of right properness of kan--quillen model structure}
Suppose the map $Y \ra Z$ in $s\S$ has $\rlp(J_\KQ)$.  Then for any map $W \ra Z$ in $s\S$, the natural map
\[ |W \times_Z Y| \ra |W| \times_{|Z|} |Y| \]
is an equivalence in $\S$, i.e.\! the pullback of $Y \ra Z$ along any map commutes with geometric realization.
\end{cor}

\begin{proof}
It suffices to show that in the diagram
\[ \begin{tikzcd}
|W \times_Z Y| \arrow{r} \arrow{d} & |Y| \arrow{d} \\
|W| \arrow{r} & |Z|
\end{tikzcd} \]
in $\S$, for every point of $|W|$ the upper map induces an equivalence on the corresponding fibers of the vertical maps.

To begin, note that by \cref{crazy lemma} (or since the map $W_0 \ra |W|$ is a surjection), every point $\pt_\S \ra |W|$ in $\S$ is represented by a point $\pt_{s\S} \simeq \Delta^0 \ra W$ in $s\S$.  For such a point $\pt_{s\S} \xra{w} W$, denote by $F_w \in s\S$ the fiber of the map $W \times_Z Y \ra W$ over $w$.

Now, as $\rlp(J_\KQ)$ is closed under pullbacks, then the map $W \times_Z Y \ra W$ also has $\rlp(J_\KQ)$.  So in the diagram
\[ \begin{tikzcd}
F_w \arrow{r} \arrow{d} & W \times_Z Y \arrow{r} \arrow{d}{\rlp(J_\KQ)} & Y \arrow{d}{\rlp(J_\KQ)} \\
\pt_{s\S} \arrow{r}[swap]{w} & W \arrow{r} & Z
\end{tikzcd} \]
in $s\S_\KQ$, both the left square and the large rectangle are pullbacks, and by \cref{the fiber of a fibration of sspaces is homotopically correct} these both remain pullbacks when we apply $|{-}| : s\S \ra \S$.  This proves the indicated sufficient condition.
\end{proof}

As stated in \cref{one of the most useful consequences of right properness}, \cref{consequence of right properness of kan--quillen model structure} already spells out one of the most useful consequences of the right properness of $s\S_\KQ$.  However, for the sake of completeness, since we have not actually proved that this model structure is right proper, we do so now.

\begin{cor}\label{kan--quillen model structure is right proper}
$\bW_\KQ$ is preserved under pullback along maps that have $\rlp(J_\KQ)$.
\end{cor}

\begin{proof}
Suppose that we have a diagram
\[ \begin{tikzcd}
& Y \arrow{d}{\rlp(J_\KQ)} \\
W \arrow{r}[swap]{\approx} & Z
\end{tikzcd} \]
in $s\S_\KQ$.  By \cref{consequence of right properness of kan--quillen model structure}, the induced diagram
\[ \begin{tikzcd}
|W \times_Z Y| \arrow{r} \arrow{d} & |Y| \arrow{d} \\
|W| \arrow{r}[swap]{\sim} & |Z|
\end{tikzcd} \]
in $\S$ is a pullback square.  But this implies that the upper map is an equivalence in $\S$, i.e.\! that the map $W \times_Z Y \ra Y$ is in $\bW_\KQ$.
\end{proof}

\subsubsection{Fibrations are preserved by $\pi_0$}\label{subsubsection necessary condition for fibn}

We now proceed to give a necessary condition (\cref{fibn implies pi-0 is a fibn}) for a map to be a fibration in $s\S_\KQ$.

\begin{lem}\label{map to pi-0 is a fibn}
For any $Y \in s\S$ and any $K \wcofibn L$ in $s\Set_\KQ$, the canonical map $Y \ra \pi_0(Y)$ in $s\S$ induces a $\pi_0$-isomorphism
\[ \Match_L(Y) \ra \lim \left( \begin{tikzcd}
& \Match_K(Y) \arrow{d} \\
\Match_L(\pi_0(Y)) \arrow{r} & \Match_K(\pi_0(Y))
\end{tikzcd} \right) \]
in $\S$.  In particular, for any $Y \in s\S$, the canonical map $Y \ra \pi_0(Y)$ has $\rlp(J_\KQ)$.
\end{lem}

\begin{proof}
Letting $J$ denote the collection of maps $K \ra L$ in $s\S$ for which this induced map in $\S$ is a $\pi_0$-isomorphism, it suffices to prove that $J$ contains the set $J_\KQ = \{ \Lambda^n_i \ra \Delta^n \}_{0 \leq i \leq n \geq 1}$.  Note first that any map $\Delta^i \ra \Delta^j$ is contained in $J$, since we have an equivalence $\Match_{\Delta^n}(-) \simeq (-)_n$ in $\Fun(s\S,\S)$ and pullbacks over discrete spaces commute with $\pi_0 : \S \ra \Set$.  Note too that $J$ is closed under pushouts and has the two-out-of-three property.

We now argue by induction: writing $J_{ \leq n} = \{ \Lambda^m_j \ra \Delta^m \}_{0 \leq j \leq m \leq n} \subset J_\KQ$, we will show that $J_{\leq n} \subset J$ for all $n \geq 1$.  We have already shown that $J_{\leq 1} \subset J$, since both maps $\Lambda^1_i \ra \Delta^1$ are of the form $\Delta^0 \ra \Delta^1$.  So, suppose that $J_{\leq (n-1)} \subset J$, and let $\Lambda^n_i \ra \Delta^n$ be any map in $J_{\leq n} \backslash J_{\leq (n-1)}$.  Observe that the composite $\Delta^{\{i\}} \ra \Lambda^n_i \ra \Delta^n$ lies in $J$, and observe that the first map can be constructed as an iterated pushout of maps in $J_{\leq (n-1)}$ (note that pushouts along cofibrations in $s\Set_\KQ$ are also pushouts in $s\S$) and hence by assumption also lies in $J$ since it is closed under pushouts.  Since $J$ also has the two-out-of-three property, it follows that the second map lies in $J$ as well.  This proves the claim.
\end{proof}

\begin{prop}\label{fibn implies pi-0 is a fibn}
If a map $Y \ra Z$ in $s\S$ has $\rlp(J_\KQ)$, then so does $\pi_0(Y) \ra \pi_0(Z)$.
\end{prop}

\begin{proof}
A map $\Lambda^n_i \ra \Delta^n$ in $J_\KQ$ gives rise to a diagram
\[ \begin{tikzcd}
& Z_n \arrow{rr} \arrow{dd} & & \pi_0(Z_n) \arrow{dd} \\
Y_n \arrow{ru} \arrow[crossing over]{rr} \arrow{dd} & & \pi_0(Y_n) \arrow{ru} \arrow{dd} \\
& \Match_{\Lambda^n_i}(Z) \arrow{rr} & & \Match_{\Lambda^n_i}(\pi_0(Z)) \\
\Match_{\Lambda^n_i}(Y) \arrow{ru} \arrow{rr} & & \Match_{\Lambda^n_i}(\pi_0(Y)) \arrow[leftarrow, crossing over]{uu} \arrow{ru}
\end{tikzcd} \]
in $\S$.  All four horizontal maps in this cube are $\pi_0$-isomorphisms by \cref{map to pi-0 is a fibn}, and moreover all of their targets are discrete since the inclusion $\Set \subset \S$ commutes with limits (being a right adjoint).  Taking pullbacks of the cospans contained in the left and right faces, we obtain a commutative square
\[ \begin{tikzcd}[row sep=1.5cm, column sep=1.5cm]
Y_n \arrow{r} \arrow{d} & \pi_0(Y_n) \arrow{d} \\
\Match_{\Lambda^n_i}(Y) \underset{\Match_{\Lambda^n_i}(Z)}{\times} Z_n \arrow{r} &  \Match_{\Lambda^n_i}(\pi_0(Y)) \underset{\Match_{\Lambda^n_i}(\pi_0(Z))}{\times} \pi_0(Z_n)
\end{tikzcd} \]
in which the upper map is a $\pi_0$-isomorphism and the lower map is a component of the canonical comparison map
\[ \pi_0 \left( {\lim}^\S(-) \right) \ra {\lim}^\Set \left( \pi_0^\lw(-) \right) \]
in $\Fun ( \Fun(\Nerve^{-1}(\Lambda^2_2) , \S),\Set )$.  As this comparison map is a componentwise surjection, the surjectivity of the left map in the commutative square implies that of its right map.  This proves the claim.
\end{proof}

\begin{rem}\label{failure of pi-0 fibn to imply fibn}
The proof of \cref{fibn implies pi-0 is a fibn} clearly illustrates the reason that its converse fails: the canonical comparison map
\[ \pi_0 \left( {\lim}^\S(-) \right) \ra {\lim}^\Set \left( \pi_0^\lw(-) \right) \]
in $\Fun ( \Fun(\Nerve^{-1}(\Lambda^2_2) , \S),\Set )$ is a componentwise surjection but not a componentwise isomorphism.  On the other hand, if its component at the object
\[ \left( \Match_{\Lambda^n_i}(Y) \ra \Match_{\Lambda^n_i}(Z) \la Z_n \right) \in \Fun(\Nerve^{-1}(\Lambda^2_2),\S) \]
happens to be an isomorphism (for instance, if $\Match_{\Lambda^n_i}(Z) \in \Set \subset \S$), then $\pi_0(Y) \ra \pi_0(Z)$ having the right lifting property against the map $\Lambda^n_i \ra \Delta^n$ implies that $Y \ra Z$ has it as well.  Assembling this observation over all maps in $J_\KQ$ then gives a partial converse to \cref{fibn implies pi-0 is a fibn}.
\end{rem}

\subsubsection{Comparisons with existing literature}\label{subsubsection fibration comparisons}

\begin{rem}\label{compare with rezk's realization fibrations}
In the unpublished note \cite{Rezk-pi-star-kan}, Rezk defines a \textit{realization fibration} to be a map $Y \ra Z$ in $s\S$ such that all pullbacks commute with geometric realization (see \cite[Definition 1.1]{Rezk-pi-star-kan}), and he completely characterizes them as being detected by iterated pullbacks along all possible composites $\Delta^{\{i\}} \ra \Delta^n \ra Z$ (see \cite[Proposition 5.10]{Rezk-pi-star-kan}).  In this language, we can restate \cref{consequence of right properness of kan--quillen model structure} as asserting that all maps in $\bF_\KQ \subset s\S$ are realization fibrations.  On the other hand, as geometric realization (being a sifted colimit) commutes with finite products in $\S$, every terminal map $Y \ra \pt_{s\S}$ is a realization fibration, but it clearly need not be in $\bF_\KQ$ in general.
\end{rem}

\begin{rem}\label{compare with anderson and bousfield--friedlander}
There are two results, both of which appeared in the literature in 1978, which are strongly reminiscent of \cref{consequence of right properness of kan--quillen model structure}:
\begin{itemize}

\item one introduced by Bousfield--Friedlander in \cite[Appendix B]{BF} (and recapitulated in \cite[Chapter IV, \sec 4]{GJ}) based on the notion of a simplicial space satisfying the \textit{$\pi_*$-Kan condition},

\item the other introduced by Anderson in \cite{AndFibGeomReal} based on the notion of a simplicial groupoid being \textit{fully fibrant}.

\end{itemize}
In fact, these two results are extremely similar to one another, and both rely on a common construction: given a space $Y \in \S$, a Grothendieck construction (see e.g.\! \cref{gr:section gr}) applied to the resulting the functor
\[ \Pi_1(Y) \xra{ \underset{n \geq 2}{\prod} \pi_n(Y,-)} \Grp \]
(from the fundamental (1-)groupoid of $Y$ to the category of groups) yields a new groupoid, which we will denote by $\Pi_{\geq 1}(Y) \in \strgpd$.  Using this, we can describe the conditions appearing in these results as follows.

\begin{itemize}

\item The \textit{$\pi_*$-Kan condition} on a simplicial space $Y \in s\S$ demands of the simplicial groupoid $\Pi_{\geq 1}(Y_\bullet) \in s\strgpd$ that for every horn inclusion $\Lambda^n_i \wcofibn \Delta^n$ in $J_\KQ$, the induced map
\[ \Pi_{\geq 1}(Y_n) \cong \Match_{\Delta^n} \left( \Pi_{\geq 1}(Y) \right) \ra \Match_{\Lambda^n_i} \left( \Pi_{\geq 1}(Y) \right) \]
is full.

\item The demand that the simplicial groupoid $\Pi_{\geq 1}(Y)_\bullet \in s\strgpd$ be \textit{fully fibrant} amounts to the additional requirement that for every boundary inclusion $\partial \Delta^n \cofibn \Delta^n$ in $I_\KQ$, the induced map
\[ \Pi_{\geq 1}(Y_n) \cong \Match_{\Delta^n} \left( \Pi_{\geq 1}(Y) \right) \ra \Match_{\partial \Delta^n} \left( \Pi_{\geq 1}(Y) \right) = \Match_n \left( \Pi_{\geq 1}(Y) \right) \]
is full.\footnote{This additional requirement can be rephrased as requiring that the simplicial groupoid be fibrant in $s(\strgpd_\can)_\Reedy$, the Reedy model structure built on the canonical model structure on $\strgpd$ (which explains the presence of the word ``fibrant'' in the terminology ``fully fibrant'').  In turn, the canonical model structure, which in fact seems to have first appeared in \cite[\sec 5]{AndFibGeomReal}, has that $\bW_\can \subset \strgpd$ consists of the equivalences of groupoids, $\bC_\can \subset \strgpd$ consists of those maps that are injective on objects, and $\bF_\can \subset \strgpd$ consists of the isofibrations (i.e.\! those maps satisfying an ``isomorphism lifting property'').}

\end{itemize}
Of course, a map in $\strgpd$ is full precisely if the induced maps on automorphism groups are all surjective, and so these conditions ultimately boil down to certain lifting criteria among the various homotopy groups $\{ \pi_i(Y_j) \}_{j \geq 0, i \geq 1}$.

Then, \cite[Theorem B.4]{BF} (resp.\! the main theorem of \cite{AndFibGeomReal}) asserts that if a map $Y \ra Z$ in $s\S$ has that
\begin{itemizesmall}
\item the induced map $\pi_0(Y) \ra \pi_0(Z)$ lies in $\bF_\KQ \subset s\Set$ and
\item both $Y$ and $Z$ satisfy the $\pi_*$-Kan condition (resp.\! both $\Pi_{\geq 1}(Y)$ and $\Pi_{\geq 1}(Z)$ are fully fibrant),
\end{itemizesmall}
then the map $Y \ra Z$ is a realization fibration (in the sense of \cref{compare with rezk's realization fibrations}).

In light of \cref{fibn implies pi-0 is a fibn} and \cref{failure of pi-0 fibn to imply fibn}, it appears quite likely that these results are actually somehow secretly asking for the map $Y \ra Z$ to be a fibration between fibrant objects, i.e.\! to lie in $\bF^f_\KQ \subset s\S$.\footnote{Anderson's notion of an \textit{epifibration} (introduced in \cite[\sec 6]{AndFibGeomReal}) provides a model-categorical counterpart to our notion of a fibration in $s\S_\KQ$ (recall \cref{model infty-cats are better than double model cats}), via which observation \cite[Theorem 6.2]{AndFibGeomReal} appears to more-or-less imply \cref{consequence of right properness of kan--quillen model structure}.  Unfortunately, there seem to be a number of issues with the proof given there.  For instance, the two paragraphs following the statement imply that epifibrations present maps in $s\S$ which have not just $\rlp(J_\KQ)$ but also have $\rlp(I_\KQ)$.  And then, the last sentence of the fourth paragraph of the proof of \cite[Lemma 6.5]{AndFibGeomReal} has a counterexample given by the inclusion $\Delta^{\{i\}} \subset \Lambda^n_i$.}

Both Bousfield--Friedlander and Anderson give classes of examples where their respective criteria hold:
\begin{itemizesmall}
\item on the one hand, $Y \in s\S$ satisfies the $\pi_*$-Kan condition
\begin{itemizesmall}
\item if each $Y_n \in \S$ is connected, or
\item if each $Y_n \in \S$ is simple and for all $i \geq 1$ the map
\[ [S^i,Y_\bullet]^\lw_\S \ra [\pt_\S,Y_\bullet]^\lw_\S \cong \pi_0(Y) \]
lies in $\bF_\KQ \subset s\Set$, so in particular
\begin{itemizesmall}
\item if it can be presented by a bisimplicial group,
\end{itemizesmall}
\end{itemizesmall}
while
\item on the other hand, $Y \in s\S$ has that $\Pi_{\geq 1}(Y)$ is fully fibrant
\begin{itemizesmall}
\item if each $Y_n \in \S$ is connected,
\item if it lies in $s\Set \subset s\S$, or
\item if it can be presented by a simplicial topological group.
\end{itemizesmall}
\end{itemizesmall}
\end{rem}

\begin{rem}\label{compare with result for levelwise-connected sspaces}
It is well known that for a simplicial space which is levelwise connected, taking loopspaces (with respect to any compatible choices of basepoints) commutes with geometric realization.  (This follows from the results discussed in \cref{compare with anderson and bousfield--friedlander}, see e.g.\! \cite[Chapter IV, Corollary 4.11]{GJ}.)  In fact, \textit{any} pullback in $s\S$ in which the common target in the cospan is levelwise connected commutes with geometric realization (see Lemma A.5.5.6.17).  Of course, there are many interesting simplicial spaces which are not levelwise connected -- for instance, any nontrivial simplicial set -- and so this result is of somewhat limited use when manipulating simplicial spaces and their geometric realizations.\end{rem}

\begin{rem}\label{compare with seymour and BS}
There are two papers which study certain classes of maps which are closely related to our notion of a fibration in $s\S_\KQ$.

\begin{itemize}

\item In \cite{Seymour}, Seymour studies a \textit{continuous} lifting condition -- that is, an enriched lifting condition with respect to the enrichment of $s\Top$ over $\Top$ -- declaring that a map $x \xra{i} y$ has the ``left lifting property'' with respect to a map $z \xra{p} w$ if the induced map
\[ \enrhom_{s\Top}(y,z) \to \enrhom_{\Fun([1],s\Top)}(i,p) \]
admits a section (instead of just being surjective).  A ``Kan fibration'' in this sense -- which for disambiguation we'll refer to as an ``S-Kan fibration'' -- is then defined to be a map in $s\Top$ which satisfies this cotinuous right lifting property against the usual set of horn inclusions $J_\KQ = \{ \Lambda^n_i \ra \Delta^n \}_{0 \leq i \leq n \geq 1}$ in $s\Set \subset s\Top$.  Thus, aside from issues of homotopy coherence (which can presumably be handled using an appropriate model structure on $s\Top$ (recall \cref{model infty-cats are better than double model cats})), it appears that these morphisms present a strict subset of those in the subcategory $\bF_\KQ \subset s\S$ of fibrations.

The main result is then that S-Kan fibrations are stable under taking the internal hom into them from any other object; taking that source object to be $\Delta^1 \in s\Set \subset s\Top$, a ``covering homotopy theorem'' immediately follows (see \cite[Theorems 4.1 and 4.2]{Seymour}).  Morally speaking, this is the case because the ``acyclic cofibrations'' in this setup are closed under taking the product with any identity map, a fact which is not true in $s\S_\KQ$ (but see \cref{other KQ model structures} for an explanation of why this feature is in a certain sense undesirable).

\item In \cite{BrownSzczarba}, Brown--Szczarba study a continuous lifting condition which is similar to that of \cite{Seymour} but is yet more restrictive: they additionally require lifting for all ``sub-horns'' (see \cite[Definition 6.1]{BrownSzczarba}).  Thus, it appears that their resulting ``fibrations'' -- which for disambiguation we'll refer to as ``BS-Kan fibrations'' -- present a strict subset of even those morphisms in $s\S$ which are presented by S-Kan fibrations.

First of all, Brown--Szczarba prove an analogous result to Seymour's (see \cite[Theorem 6.2]{BrownSzczarba}).  Moreover, given a BS-Kan object $Y_\bullet \in s\Top$ equipped with a basepoint $\pt_{s\Top} \xra{y} Y_\bullet$, they define its ``homotopy groups'' to be those of the underlying pointed simplicial set, so that $\pi_n(Y,y)$ is a quotient of the subset
\[ \{ y_n \in Y_n : \delta^n_0 (y_n) = \delta^n_1(y_n) = \cdots = \delta^n_n(y_n) = y \} \subset Y_n , \]
but is additionally topologized via the quotient topology.  With respect to these homotopy groups, they obtain a ``continuous long exact sequence associated to a fibration of Kan simplicial topological spaces'' (in which the (strict) fiber is also a BS-Kan object) (see \cite[Theorem 6.5 and Proposition 6.6]{BrownSzczarba}).  They also develop notions of continuous (singular and de Rham) cohomology and of real homotopy type.  Of course, all but the first of these accomplishments lie outside of the scope of what we seek to achieve here (and a comparison of the first with the present work is no different from that given above).

\end{itemize}
\end{rem}

\subsection{The $\Ex^\infty$ functor}\label{subsection Ex-infty}

We now return to the general theory.  In the proof of \cref{detect acyclic fibrations of sspaces} we will need to have a version of the $\Ex^\infty$ functor for simplicial spaces, so we take a moment to develop that now.  For simplicial sets, this was originally defined and explored in \cite[\sec 3-4]{KanEx}; it is developed in more modern terminology in \cite[Chapter III, \sec 4]{GJ}.

\begin{defn}\label{define sd on sspaces}
Recall that any $\Delta^n \in s\Set$ admits a \bit{subdivision}, denoted $\sd(\Delta^n) \in s\Set$; this is the nerve of its poset of nondegenerate simplices.  Recall further that this admits a map $\sd(\Delta^n) \ra \Delta^n$, called the \bit{last vertex} map, induced by the map of posets given by taking a simplex to its last vertex.  Recall still further that we can extend this definition to any $K \in s\Set$ by defining
\[ \sd(K) = \colim_{(\Delta^n \ra K) \in \left( \bD \underset{s\Set}{\times} s\Set_{/K} \right)} \sd(\Delta^n) , \]
and that we obtain an induced last vertex map $\sd(K) \ra K$.  We now extend this even further to any $Y \in s\S$ by defining
\[ \sd(Y) = \colim_{(\Delta^n \ra Y) \in \left( \bD \underset{s\S}{\times} s\S_{/Y} \right)} \sd(\Delta^n) ; \]
in the same way, this also admits a last vertex map $\sd(Y) \ra Y$.  Note that this does indeed extend the functor $\sd : s\Set \ra s\Set$, as $s\Set \subset s\S$ is a full subcategory (so that for any $K \in s\Set \subset s\S$, we have an equivalence
\[ \bD \underset{s\Set}{\times} s\Set_{/K} \xra{\sim} \bD \underset{s\S}{\times} s\S_{/K} \]
of $\infty$-categories).  This clearly defines a functor $\sd:s\S \ra s\S$.
\end{defn}

\begin{defn}\label{define Ex on sspaces}
We define the \bit{extension} of $Y \in s\S$ to be the object $\Ex(Y) \in s\S$ defined by $\Ex(Y)_n = \hom_{s\S}(\sd(\Delta^n),Y)$, with simplicial structure maps corepresented by the cosimplicial structure maps of $\sd(\Delta^\bullet) \in c(s\Set)$.  This defines a functor $\Ex : s\S \ra s\S$, which extends the usual functor $\Ex : s\Set \ra s\Set$ (again since $s\Set \subset s\S$ is a full subcategory).
\end{defn}

\begin{notn}\label{iterate sd and Ex}
For any $i \geq 0$, we write $\sd^i = \sd^{\circ i}$ and $\Ex^i = \Ex^{\circ i}$ for the iterated composites of the indicated endofunctors on $s\S$ of \cref{define sd on sspaces}.
\end{notn}

\begin{lem}\label{sd and Ex are adjoint}
The functors $\sd$ and $\Ex$ define an adjunction $\sd : s\S \adjarr s\S : \Ex$, and hence the functors $\sd^i$ and $\Ex^i$ define an adjunction $\sd^i : s\S \adjarr s\S : \Ex^i$ for any $i \geq 0$.
\end{lem}

\begin{proof}
The first statement follows directly from the definitions and the fact that, as in any presheaf category, any $Y \in s\S = \Fun(\bD^{op},\S)$ is recoverable as a colimit
\[ Y \simeq \colim_{(\Delta^n \ra Y) \in \left( \bD \underset{s\S}{\times} s\S_{/Y} \right) } \Delta^n \]
of representable presheaves.  The second statement is obtained by composing the adjunction $i$ times.
\end{proof}

\begin{notn}\label{notn Ex-infty}
By \cref{sd and Ex are adjoint}, for any $Y \in s\S$ the last vertex map $\sd(Y) \ra Y$ is adjoint to a map $Y \ra \Ex(Y)$.  We write
\[ \Ex^\infty(Y) = \colim(Y \ra \Ex(Y) \ra \Ex^2(Y) \ra \cdots) . \]
This defines an endofunctor $\Ex^\infty : s\S \ra s\S$.
\end{notn}

\begin{rem}
Using the classical theory of $s\Set_\KQ$, we can see that $\Ex^\infty$ cannot be a right adjoint.  For instance, it does not commute with the countably infinite product of copies of the ``simplicial infinite line'', i.e.\! the nerve of the poset $(\bbZ,\leq)$.  (This product is not acyclic, but the countably infinite product of any acyclic Kan complexes is again acyclic.)
\end{rem}

We now give the result which we will need in the proof of \cref{detect acyclic fibrations of sspaces}.

\begin{prop}\label{Ex-infty is fibrant}
For any $Y \in s\S$, $\Ex^\infty(Y) \in s\S_\KQ^f$.
\end{prop}

\begin{proof}
The proof of \cite[Chapter III, Lemma 4.7]{GJ}, given as it is by a universal computation involving the map $\Lambda^n_i \ra \Delta^n$, works equally well in our setting to show that for any $Y \in s\S$ and for any map $\Lambda^n_i \ra \Ex(Y)$, there exists an extension
\[ \begin{tikzcd}
\Lambda^n_i \arrow{r} \arrow{d} & \Ex(Y) \arrow{d} \\
\Delta^n \arrow[dashed]{r} & \Ex^2(Y) .
\end{tikzcd} \]
By \cref{permit the small object argument}, it follows that $\Ex^\infty(Y) \ra \pt_{s\S}$ is in $\rlp(J_\KQ)$, i.e.\! that $\Ex^\infty(Y) \in s\S_\KQ^f$.
\end{proof}

\begin{rem}\label{bonus properties of sd and Ex}
Many of the usual results regarding the classical endofunctors $\sd : s\Set \ra s\Set$ and $\Ex : s\Set \ra s\Set$ extend to our setting.

For instance, the functor $\Ex^\infty : s\S \ra s\S$ (with its canonical map from $\id_{s\S}$) is a fibrant replacement functor in $s\S_\KQ$.  To see this, in light of \cref{Ex-infty is fibrant} it suffices to show that the map $Y \ra \Ex^\infty(Y)$ is in $\bW_\KQ$.  Since the subcategory $\bW_\KQ \subset s\S$ is closed under transfinite composition, it suffices to show that the map $Y \ra \Ex(Y)$ is in $\bW_\KQ$.  For this, we first use the small object argument to produce a map $Y' \ra Y$ in $s\S$ that has $\rlp(I_\KQ)$ with $Y' \in s\Set \subset s\S$.  Then, we see that the map $\Ex(Y') \ra \Ex(Y)$ also has $\rlp(I_\KQ)$ since a commutative square
\[ \begin{tikzcd}
\partial \Delta^n \arrow{r} \arrow[tail]{d} & \Ex(Y') \arrow{d} \\
\Delta^n \arrow{r} & \Ex(Y)
\end{tikzcd} \]
is adjoint to a commutative square
\[ \begin{tikzcd}
\sd(\partial \Delta^n) \arrow{r} \arrow[tail]{d} & Y' \arrow{d}{\rlp(I_\KQ)} \\
\sd(\Delta^n) \arrow{r} & Y ,
\end{tikzcd} \]
and there is always a lift in the latter square (which is equivalent to a lift in the former square).  It follows from \cref{detect acyclic fibrations of sspaces} below that we have both $Y' \we Y$ and $\Ex(Y') \we \Ex(Y)$, and hence from the diagram
\[ \begin{tikzcd}
Y' \arrow{r}{\approx} \arrow{d}[sloped, anchor=north]{\approx} & \Ex(Y') \arrow{d}[sloped, anchor=south]{\approx} \\
Y \arrow{r} & \Ex(Y)
\end{tikzcd} \]
we deduce that also $Y \we \Ex(Y)$ since $\bW_\KQ$ satisfies the two-out-of-three property.  On the other hand, our choice to use an unenriched lifting condition (and hence to be relatively restrictive about which maps are cofibrations) means that the map $Y \ra \Ex(Y)$ is not generally in $\bC_\KQ$.

From here, it is not hard to see that in fact we have a Quillen equivalence $\sd : s\S_\KQ \adjarr s\S_\KQ : \Ex$.  Indeed, it is straightforward to check that $\sd : s\S \ra s\S$ preserves both $I_\KQ \dashcell$ and $J_\KQ \dashcell$, so that this adjunction is a Quillen adjunction.  Then, the condition for being a Quillen equivalence follows from the facts
\begin{itemizesmall}
\item that $\sd(K) \we K$ for any $K \in s\Set = s\S^c_\KQ$,
\item that $Y \we \Ex(Y)$ for any $Y \in s\S$ (as shown above), and
\item that $\bW_\KQ \subset s\S$ satisfies the two-out-of-three property.
\end{itemizesmall}

One can similarly verify using standard arguments that $\Ex^\infty$ preserves $\bF_\KQ$, finite limits, filtered colimits, and $0\th$ spaces.
\end{rem}

\section{The proof of the Kan--Quillen model structure}\label{section proof of kan--quillen model structure}

We now turn to the components of the proof of the main result of this paper, \cref{kan--quillen model structure on sspaces}.  Recall that this appeals to the recognition theorem for cofibrantly generated model $\infty$-categories (\ref{recognize cofgen}); we verify the various criteria in turn, as itemized in the proof of \cref{kan--quillen model structure on sspaces} above.

\begin{prop}\label{detect acyclic cofibrations of sspaces}
$J_\KQ \dashcof \subset (I \dashcof \cap \bW)_\KQ$.
\end{prop}

\begin{proof}
First, since $J_\KQ \subset I_\KQ \dashcell$, then $J_\KQ \dashinj \supset (I_\KQ \dashcell) \dashinj = I_\KQ \dashinj$, so $J_\KQ \dashcof \subset I_\KQ \dashcof$.  So it remains to show that $J_\KQ \dashcof \subset \bW_\KQ$.

To show that $J_\KQ \dashcof \subset \bW_\KQ$, we claim that it suffices to show that $J_\KQ \dashcell \subset \bW_\KQ$; more precisely, we claim that any map in $J_\KQ \dashcof$ is a retract of a map in $J_\KQ \dashcell$, and so the result follows from the fact that $\bW_\KQ$ is closed under retracts.  Indeed, by \cref{permit the small object argument}, we can apply the small object argument for $J_\KQ$ to any map in $J_\KQ \dashcof$ to factor it as a map in $J_\KQ \dashcell$ followed by a map in $J_\KQ \dashinj$.  Then, by the retract argument (in the form of \cite[Proposition 7.2.2(1)]{Hirsch}, whose proof for 1-categories carries over verbatim to $\infty$-categories), it follows that the original map in $J_\KQ \dashcof$ is a retract of the map in $J_\KQ \dashcell$.

Finally, to see that $J_\KQ \dashcell \subset \bW_\KQ$, since a sequential colimit of equivalences is an equivalence, by transfinite induction it suffices to show that if we have a pushout square
\[ \begin{tikzcd}
\Lambda^n_i \arrow{r} \arrow{d} & Y \arrow{d} \\
\Delta^n \arrow{r} & Z
\end{tikzcd} \]
in $s\S$, then $|Y| \xra{\sim} |Z|$ in $\S$.  But this follows from the fact that geometric realization (being a colimit) commutes with pushouts, so the induced square
\[ \begin{tikzcd}
|\Lambda^n_i| \arrow{r} \arrow{d}[sloped, anchor=north]{\sim} & |Y| \arrow{d} \\
|\Delta^n| \arrow{r} & |Z|
\end{tikzcd} \]
is a pushout in $\S$.
\end{proof}

\begin{prop}\label{detect acyclic fibrations of sspaces}
$I_\KQ \dashinj \subset (J \dashinj \cap \bW)_\KQ$.
\end{prop}

\begin{proof}
First, since $J_\KQ \subset I_\KQ \dashcell$, then $I_\KQ \dashinj = (I_\KQ \dashcell) \dashinj \subset J_\KQ \dashinj$.  So it remains to show that $I_\KQ \dashinj \subset \bW_\KQ$.

So, suppose that the map $Y \ra Z$ is in $I_\KQ \dashinj$, i.e.\! that it has $\rlp(I_\KQ)$.  By \cref{detect equivalences of spaces}, it suffices to show that any diagram
\[ \begin{tikzcd}
S^{n-1} \arrow{r} \arrow{d} & |Y| \arrow{d} \\
\pt_\S \arrow{r} & |Z|
\end{tikzcd} \]
in $\S$ admits a lift, for any $n \geq 0$.  Using the observation in \cref{crazy lemma is not as strong as one might hope}, by \cref{crazy lemma} there exists a factorization $\es_{s\Set} \ra K \we (\Delta^{n-1}/\partial \Delta^{n-1})$ in $s\Set_\KQ$ and a map $K \ra Y$ in $s\S$ presenting the upper map in this diagram, where $K$ has only finitely many nondegenerate simplices.

Now, the above diagram gives us a nullhomotopy of the composite $S^{n-1} \simeq |K| \ra |Y| \ra |Z|$ in $\S$, and we would like to extend the composite $K \ra Y \ra Z$ over a cofibration from $K$ into a acyclic object of $s\Set_\KQ$ in a way which presents this nullhomotopy.  To do this, we write $M = (K \times \Delta^1) /( K \times \Delta^{\{1\}})$ (with its natural inclusion $K \cong K \times \Delta^{\{0\}} \cofibn M$ in $s\Set_\KQ$), and then by \cref{Ex-infty is fibrant} and \cref{fibrant sspaces are fibrant} we conclude that there must exist an extension
\[ \begin{tikzcd}
K \arrow{r} \arrow[tail]{d} & Z \arrow{r} & \Ex^\infty(Z) \\
M \arrow[dashed]{rru}
\end{tikzcd} \]
in $s\S_\KQ$ modeling the above nullhomotopy.  However, since $M$ also has only finitely many nondegenerate simplices, then by \cref{finite ssets are small sspaces} there must exist a factorization
\[ \begin{tikzcd}
K \arrow{r} \arrow[tail]{d} & L \arrow{r} & \Ex^i(Z) \arrow{r} & \Ex^\infty(Z) \\
M \arrow[dashed]{rru} \arrow{rrru}
\end{tikzcd} \]
for some $i<\infty$.  Via the adjunction $\sd^i : s\S \adjarr s\S : \Ex^i$ of \cref{sd and Ex are adjoint}, the above extension yields the extension
\[ \begin{tikzcd}
\sd^i(K) \arrow{r}{\approx} \arrow[tail]{d} & K \arrow{d} \\
\sd^i(M) \arrow[dashed]{r} & Z
\end{tikzcd} \]
in $s\S_\KQ$, and plugging this back into the original diagram gives us the diagram
\[ \begin{tikzcd}
\sd^i(K) \arrow{r}{\approx} \arrow[tail]{d} & K \arrow{r} & Y \arrow{d} \\
\sd^i(M) \arrow{rr} & & Z
\end{tikzcd} \]
in $s\S_\KQ$.  Now, since by assumption $Y \ra Z$ has $\rlp(I_\KQ)$, then there must exist a lift
\[ \begin{tikzcd}
\sd^i(K) \arrow{r}{\approx} \arrow[tail]{d} & K \arrow{r} & Y \arrow{d} \\
\sd^i(M) \arrow{rr} \arrow[dashed, bend right=10]{rru} & & Z
\end{tikzcd} \]
in $s\S_\KQ$, and upon extracting the outer rectangle and taking geometric realizations, this yields the desired lift
\[ \begin{tikzcd}
S^{n-1} \arrow{r} \arrow{d} & |Y| \arrow{d} \\
\pt_\S \arrow{r} \arrow[dashed]{ru} & |Z|
\end{tikzcd} \]
in $\S$.
\end{proof}

\begin{rem}\label{compare with moerdijk}
In the same spirit as \cref{compare with anderson and bousfield--friedlander}, the criterion of \cref{detect acyclic fibrations of sspaces} is comparable to the lifting criterion coming from the generating cofibrations in the \textit{Moerdijk model structure} on $ss\Set$ (originally introduced in \cite[\sec 1]{Moer}, but see also \cite[Chapter IV, \sec 3.3]{GJ}).  However, to actually make a direct comparison requires a bit of care, and so we explain this in some detail.

First of all, the diagonal functor $\diag : \bD^{op} \ra \bD^{op} \times \bD^{op}$ induces an adjunction $\diag_! : s\Set \adjarr ss\Set : \diag^*$, and the Moerdijk model structure on $ss\Set$ is induced by the standard lifting theorem (on which \cref{lift cofgen} is based) applied to the model structure $s\Set_\KQ$; in fact, this yields a Quillen equivalence $\diag_! : s\Set_\KQ \adjarr ss\Set_\Moer : \diag^*$.  Denoting the external product by $- \extprod - : s\Set \times s\Set \ra ss\Set$, the generating cofibrations for this model structure are thus given by
\[ I^{ss\Set}_\Moer = \{ \diag_!(\partial \Delta^n \ra \Delta^n) \}_{n \geq 0} = \{ \partial \Delta^n \extprod \partial \Delta^n \ra \Delta^n \extprod \Delta^n \}_{n \geq 0} , \]
and we have that $\rlp(I^{ss\Set}_\Moer) \subset \bW^{ss\Set}_\Moer$.

Next, recall that we can present the $\infty$-category $s\S$ using the model category $s(s\Set_\KQ)_\Reedy$.  Thus, any map $Y \ra Z$ in $s\S$ can be presented as a map $\ttY \ra \ttZ$ in $s(s\Set_\KQ)_\Reedy$.  If the latter map happens to have $\rlp(I^{ss\Set}_\Moer)$, then we will have that the induced map $\diag^*(\ttY) \ra \diag^*(\ttZ)$ will be in $\bW^{s\Set}_\KQ$; since in $s(s\Set_\KQ)_\Reedy$ the diagonal always computes the homotopy colimit (see e.g.\! \cite[Chapter IV, Exercise 1.6]{GJ} or Example T.A.2.9.31), then this map in $s\Set_\KQ$ presents the map $|Y| \ra |Z|$ in $\S$, which is therefore an equivalence.

Now, observe that the maps in $I^{ss\Set}_\Moer$ are cofibrations (between cofibrant objects) when considered in the model category $s(s\Set_\KQ)_\Reedy$.  Moreover, note that if
\begin{itemizesmall}
\item the maps $A \xra{i} B$ and $Y \xra{p} Z$ in $s\S$ are respectively presented by the maps $\ttA \stackrel{\tti}{\cofibn} \ttB$ and $\ttY \stackrel{\ttp}{\fibn} \ttZ$  in $s(s\Set_\KQ)_\Reedy$,
\item $\ttZ \in s(s\Set_\KQ)_\Reedy^f$, and
\item $p \in \rlp(\{i\})$ in $s\S$,
\end{itemizesmall}
then $\ttp \in \rlp(\{\tti\})$ in $s(s\Set)$.  (This follows from the fact that under these hypotheses, the model category $(s(s\Set_\KQ)_\Reedy)_{\ttA /\! / \ttZ}$ presents the $\infty$-category $s\S_{A/ \! /Z}$.)  Together, these imply that if the map $Y \ra Z$ in $s\S$ has the right lifting property against the set
\[ I^{s\S}_\Moer = \{ S^{n-1} \tensoring \partial \Delta^n \ra \pt_\S \tensoring \Delta^n \}_{n \geq 0} = \{ S^{n-1} \tensoring \partial \Delta^n \ra \Delta^n \}_{n \geq 0} , \]
then it can be presented by a map $\ttY \ra \ttZ$ in the model category $s(s\Set_\KQ)_\Reedy$ which has $\rlp(I^{ss\Set}_\Moer)$, and therefore $|Y| \xra{\sim} |Z|$ in $\S$.  Hence, we obtain that $\rlp(I^{s\S}_\Moer) \subset \bW_\KQ$.  This, finally, allows us to make a direct comparison.

Of course, what we come to see now is that it appears much easier to have $\rlp(I^{s\S}_\KQ)$ than to have $\rlp(I^{s\S}_\Moer)$.  The sets $I^{s\S}_\Moer$ and $I^{s\S}_\KQ$ of homotopy classes of maps in $s\S$ are illustrated in \cref{picture of Moer and KQ sets}, in which the various shapes and their positions are meant to vaguely indicate the different simplicial levels at which these spaces live as well as the simplicial structure maps between them:
\begin{figure}[h]
\begin{tikzpicture}[scale=2]


\draw (0,0) circle (2pt);
\draw (1,0) circle (2pt);
\draw (0.5,0.866) circle (2pt);
\draw (0.5,0) ellipse (6pt and 1.5pt);
\draw [rotate around={60:(0.25,0.433)}] (0.25,0.433) ellipse (6pt and 1.5pt);
\draw [rotate around={120:(0.75,0.433)}] (0.75,0.433) ellipse (6pt and 1.5pt);

\path (1.25,0.5) edge[style={-stealth}] node[above] {$I^{s\S}_\Moer$} (1.75,0.5);

\draw [fill] (2,0) circle (2pt);
\draw [fill] (3,0) circle (2pt);
\draw [fill] (2.5,0.866) circle (2pt);
\draw [fill] (2.5,0) ellipse (6pt and 1.5pt);
\draw [rotate around={60:(2.25,0.433)}, fill] (2.25,0.433) ellipse (6pt and 1.5pt);
\draw [rotate around={120:(2.75,0.433)}, fill] (2.75,0.433) ellipse (6pt and 1.5pt);
\draw [fill] (2.413,0.239) -- (2.587,0.239) -- (2.5,0.389) -- (2.413, 0.239);


\draw [fill] (4.5,0) circle (2pt);
\draw [fill] (5.5,0) circle (2pt);
\draw [fill] (5,0.866) circle (2pt);
\draw [fill] (5,0) ellipse (6pt and 1.5pt);
\draw [rotate around={60:(4.75,0.433)}, fill] (4.75,0.433) ellipse (6pt and 1.5pt);
\draw [rotate around={120:(5.25,0.433)}, fill] (5.25,0.433) ellipse (6pt and 1.5pt);

\path (5.75,0.5) edge[style={-stealth}] node[above] {$I^{s\S}_\KQ$} (6.25,0.5);

\draw [fill] (6.5,0) circle (2pt);
\draw [fill] (7.5,0) circle (2pt);
\draw [fill] (7,0.866) circle (2pt);
\draw [fill] (7,0) ellipse (6pt and 1.5pt);
\draw [rotate around={60:(6.75,0.433)}, fill] (6.75,0.433) ellipse (6pt and 1.5pt);
\draw [rotate around={120:(7.25,0.433)}, fill] (7.25,0.433) ellipse (6pt and 1.5pt);
\draw [fill] (6.913,0.239) -- (7.087,0.239) -- (7,0.389) -- (6.913, 0.239);


\node at (0.5,-0.3) {$S^{n-1} \tensoring \partial \Delta^n$};
\node at (2.5,-0.3) {$\Delta^n$};

\node at (5,-0.3) {$\partial \Delta^n$};
\node at (7,-0.3) {$\Delta^n$};

\end{tikzpicture}
\caption{The maps in $I^{s\S}_\Moer$ and $I^{s\S}_\KQ$ at $n=2$.}\label{picture of Moer and KQ sets}
\end{figure}
\begin{itemizesmall}
\item the maps in $I^{s\S}_\Moer$ are obtained by simultaneously coning off a $\partial \Delta^n$ worth of $(n-1)$-spheres and adding a new nondegenerate point in the $n^{\rm th}$ simplicial level, while
\item the maps in $I^{s\S}_\KQ$ are simply maps of discrete simplicial spaces.
\end{itemizesmall}
In particular, the maps in $I^{s\S}_\Moer$ have real homotopical content, and thus it appears that checking that a map has $\rlp(I^{s\S}_\Moer)$ is indeed much more difficult than checking that it has $\rlp(I^{s\S}_\KQ)$.
\end{rem}

\begin{rem}\label{rem ceg-rem}
There is also the ``$\ol{W}$ model structure'' on $ss\Set$ of \cite{CegRem}, which admits a left Quillen equivalence to $ss\Set_\Moer$ (see \cite[Theorem 9]{CegRem}).  However, this is of course also inherently 1-categorical, and hence any $\infty$-categorical lifting criteria that come of it will likewise necessarily contain far more geometric content than the maps in $I_\KQ$ (compare with Remarks \ref{compare with moerdijk} \and \ref{ssets in infty-cat of sspaces know about more than just levelwise 0-simplices}).
\end{rem}

\begin{rem}\label{ssets in infty-cat of sspaces know about more than just levelwise 0-simplices}
If one thinks of the 1-category of topological spaces or of simplicial sets in place of the $\infty$-category $\S$, then \cref{detect acyclic fibrations of sspaces} may seem somewhat implausible.  For instance, the functor $\disc : s\Set \ra s\S$ of $\infty$-categories is modeled in $\strrelcat_\BarKan$ by an evident functor $\const : s\Set_\triv \ra s(s\Set_\KQ)_\Reedy$.  On underlying 1-categories, this functor participates in an adjunction $\const : s\Set \adjarr ss\Set : ((-)_0)^\lw$ with the ``levelwise 0-simplices'' functor, which takes each constituent simplicial set to its set of 0-simplices.  This right adjoint clearly doesn't know anything about the higher homotopical information in the bisimplicial set, and in particular cannot recover its geometric realization.  Hence, one might deduce that asking for the right lifting property in $s\S$ against maps in the image of $\disc : s\Set \ra s\S$ could not possibly tell us about the functor $|{-}| : s\S \ra \S$.  However, asking for the 0-simplices of a simplicial set isn't a homotopical operation in $s\Set_\KQ$, and so we cannot expect this maneuver to tell us anything $\infty$-categorical.  Indeed, the above right adjoint is certainly not a relative functor, nor is it even a right Quillen functor (with respect to the indicated model structures), corresponding to the fact that the functor $\disc : s\Set \ra s\S$ isn't a left adjoint.
\end{rem}

\begin{ex}
For a concrete nonexample of \cref{detect acyclic fibrations of sspaces}, we show that the map $\const(S^1) \ra \const(\pt_\S) \simeq \pt_{s\S}$ in $s\S$ (which is not in $\bW_\KQ$) does not have $\rlp(\{ \partial \Delta^2 \ra \Delta^2 \})$.  This illustrates the capability of ``simplicial'' spheres to detect ``geometric'' spheres in $s\S_\KQ$.

The quickest way to proceed is to use the adjunction $|{-}| = \colim : s\S \adjarr \S : \const$.  This gives a canonical commutative square
\[ \begin{tikzcd}
\partial \Delta^2 \arrow{r} \arrow{d} & \const(S^1) \arrow{d} \\
\Delta^2 \arrow{r} & \const(\pt_\S)
\end{tikzcd} \]
in $s\S$ which corresponds to the evident commutative square
\[ \begin{tikzcd}
|\partial \Delta^2| \arrow{r}{\sim} \arrow{d} & S^1 \arrow{d} \\
|\Delta^2| \arrow{r}[swap]{\sim} & \pt_\S
\end{tikzcd} \]
in $\S$.  Moreover, a lift in either square yields a lift in the other, but a lift in the latter diagram would imply that its vertical maps are also equivalences, which is clearly false.

But we can also describe the above commutative square in $s\S$ more explicitly.  Namely, we can define the upper map $\partial \Delta^2 \ra \const(S^1)$ in $s\S$ by giving a \textit{weak} natural transformation of simplicial topological spaces, as illustrated in \cref{weak natural transformation} (using the same schematics as were employed in \cref{picture of Moer and KQ sets}).
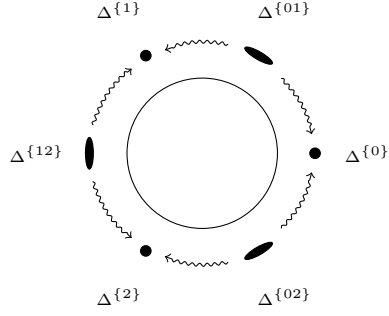
\begin{figure}[h]
\begin{tikzpicture}


\draw (0,0) circle (1cm);


\draw [fill] (1*1.5,0*1.5) circle (2pt);
\draw [fill] (-0.5*1.5,0.866*1.5) circle (2pt);
\draw [fill] (-0.5*1.5,-0.866*1.5) circle (2pt);


\draw [rotate around={-30:(0.5*1.5,0.866*1.5)}, fill] (0.5*1.5,0.866*1.5) ellipse (6pt and 1.5pt);
\draw [rotate around={90:(-1*1.5,0*1.5)}, fill] (-1*1.5,0*1.5) ellipse (6pt and 1.5pt);
\draw [rotate around={30:(0.5*1.5,-0.866*1.5)}, fill] (0.5*1.5,-0.866*1.5) ellipse (6pt and 1.5pt);


\draw (0.7*1.5,0.666*1.5) edge[snake arrow, bend left=15] (0.97*1.5,0.17*1.5);

\draw (0.23*1.5,0.966*1.5) edge[snake arrow, bend right=15] (-0.33*1.5,0.916*1.5);

\draw (-0.97*1.5,0.25*1.5) edge[snake arrow, bend left=15] (-0.626*1.5,0.75*1.5);


\draw (0.7*1.5,-0.666*1.5) edge[snake arrow, bend right=15] (0.97*1.5,-0.17*1.5);

\draw (0.23*1.5,-0.966*1.5) edge[snake arrow, bend left=15] (-0.33*1.5,-0.916*1.5);

\draw (-0.97*1.5,-0.25*1.5) edge[snake arrow, bend right=15] (-0.626*1.5,-0.75*1.5);


\node at (1*2.2,0*2.2) {\tiny $\Delta^{\{0\}}$};
\node at (0.5*2.2,0.866*2.2) {\tiny $\Delta^{\{01\}}$};
\node at (-0.5*2.2,0.866*2.2) {\tiny $\Delta^{\{1\}}$};
\node at (-1*2.2,0*2.2) {\tiny $\Delta^{\{12\}}$};
\node at (-0.5*2.2,-0.866*2.2) {\tiny $\Delta^{\{2\}}$};
\node at (0.5*2.2,-0.866*2.2) {\tiny $\Delta^{\{02\}}$};

\end{tikzpicture}
\caption{A weak natural transformation of simplicial topological spaces from $\partial \Delta^2$ to $\const(S^1)$.}\label{weak natural transformation}
\end{figure}
Let us parametrize the circle as the group-theoretic quotient $\mathbb{R} / \mathbb{Z}$.  Then, we begin at level 0 by sending $\Delta^{\{0\}}$ to $0$, $\Delta^{\{1\}}$ to $\nicefrac{1}{3}$, and $\Delta^{\{2\}}$ to $\nicefrac{2}{3}$.  Since $\partial \Delta^2$ is 1-skeletal, it remains to fill in the commutative diagram
\[ \begin{tikzcd}
L_1(\partial \Delta^2) \arrow{r} \arrow{d} & (\partial \Delta^2)_1 \arrow{r} \arrow[dashed]{d} & M_1(\partial \Delta^2) \arrow{d} \\
L_1(\const(S^1)) \arrow{r} & \const(S^1)_1 \arrow{r} & M_1(\const(S^1))
\end{tikzcd} \]
in $\S$.  To do this, we map the degenerate elements of $(\partial \Delta^2)_1$ so that the left square commutes on the nose.  Then, we map the nondegenerate elements of $(\partial \Delta^2)_1$ to $\const(S^1)_1 = S^1$ by sending $\Delta^{\{01\}}$ to $\nicefrac{1}{6}$, sending $\Delta^{\{12\}}$ to $\nicefrac{1}{2}$, and sending $\Delta^{\{02\}}$ to $\nicefrac{5}{6}$.  To select a homotopy witnessing the homotopy commutativity of the right square, for $i \not= j$ we choose the evident paths of length $\nicefrac{1}{6}$ from the image of each $\Delta^{\{ij\}}$ to the images of $\Delta^{\{i\}}$ and $\Delta^{\{j\}}$ (as indicated by the squiggly arrows in \cref{weak natural transformation}).  Again, it is clear that this map cannot be extended over $\Delta^2$.\footnote{This can also be realized as an actual natural transformation of simplicial topological spaces if we're willing to use a fatter model for the object $\partial \Delta^2 \in s\S$, despite the fact that the simplicial topological space which is constant at the circle isn't actually fibrant in the corresponding Reedy model structure.}
\end{ex}

\begin{rem}\label{circularity of potential proof of detect acyclic fibrations of sspaces}
In the proof of \cref{detect acyclic fibrations of sspaces} above, one might be tempted to apply the small object argument for $I_\KQ$ to the map $K \ra Z$ to obtain a factorization $K \ra L \ra Z$ with the map $K \ra L$ in $I_\KQ \dashcell$ (so that $L \in s\Set \subset s\S$) and with the map $L \ra Z$ in $I_\KQ \dashinj$; then, we could proceed with the proof using standard techniques in $s\Set_\KQ$, using $L \in s\Set_\KQ$ as a replacement for $Z \in s\S_\KQ$.  If this worked, it would allow us to sidestep the extension of the functor $\Ex^\infty$ from $s\Set$ to $s\S$.  However, such an argument would be circular: we can certainly obtain such a factorization $K \ra L \ra Z$, but to conclude that the map $L \ra Z$ is in $\bW_\KQ$ because it is in $I_\KQ \dashinj$ uses precisely the result that we are trying to prove.
\end{rem}

\begin{rem}
Whereas \cref{detect acyclic fibrations of sspaces} only allows us to detect acyclic fibrations in $s\S_\KQ$, the ``weak equivalence criterion'' of \cite[30.10]{BEBdBP} gives a complete characterization of the subcategory $\bW_\KQ^f \subset s\S$ of weak equivalences between fibrant objects (and thence also a complete (albeit rather abstract) characterization of the entire subcategory $\bW_\KQ \subset s\S$ of weak equivalences).
\end{rem}

In contrast with \cref{circularity of potential proof of detect acyclic fibrations of sspaces}, now that we have \cref{detect acyclic fibrations of sspaces} in hand, we can use this technique of reducing to $s\Set_\KQ$.  We employ it in proving the following result, the last of this section.

\begin{prop}\label{acyclic fibrations have rlp(I)}
$(J \dashinj \cap \bW)_\KQ \subset I_\KQ\dashinj$.
\end{prop}

\begin{proof}
Suppose that the map $Y \ra Z$ in $s\S$ has $\rlp(J_\KQ)$ and geometrically realizes to an equivalence in $\S$.  We must show that $Y \ra Z$ also has $\rlp(I_\KQ)$, i.e.\! that any commutative square
\[ \begin{tikzcd}
\partial \Delta^n \arrow{r} \arrow{d} & Y \arrow{d} \\
\Delta^n \arrow{r} & Z
\end{tikzcd} \]
in $s\S$ admits a lift, for any $n \geq 0$.  We argue by constructing the diagram (and in particular the dotted arrow, which solves the above lifting problem) in $s\S_\KQ$ given in \cref{diagram in the proof that acyclic fibrations have rlp(I)}, beginning with only the outermost square.
\begin{figure}[h]
\[ \begin{tikzcd}[row sep=0.5cm, column sep=2cm]
\partial \Delta^n \arrow{rrr} \arrow[tail]{dr} \arrow[tail]{ddddddd} & & & Y \arrow{ddddddd}{\tiny \rlp(J_\KQ)}[swap, sloped, anchor=north]{\approx} \\
& Y' \arrow[tail]{dr}[sloped, pos=0.45]{\approx} \arrow[bend right=10]{rru}[sloped, pos=0.55]{\approx}[sloped, swap, pos=0.4]{\rlp(I_\KQ)} \\
& & Y'' \arrow[two heads]{dddd}[sloped, anchor=south]{\approx} \arrow[bend right=10]{ruu}[sloped, swap, pos=0.4]{\approx} \arrow[leftarrow, dashed]{dddddll} \\ \\ \\ \\
& & Z' \arrow{dr}[sloped, pos=0.25]{\rlp(I_\KQ)}[sloped, swap, pos=0.6]{\approx} \arrow[leftarrowtail, crossing over]{uuuuul} \\
\Delta^n \arrow{rrr} \arrow{rru} & & & Z
\end{tikzcd} \]
\caption{The diagram in $s\S_\KQ$ used in the proof of \cref{acyclic fibrations have rlp(I)}.}\label{diagram in the proof that acyclic fibrations have rlp(I)}
\end{figure}
For clarity, we proceed in steps.

\begin{enumerate}

\item 

We use the small object argument for $I_\KQ$ to obtain the the factorization $\partial \Delta^n \cofibn Y' \ra Y$ in $s\S_\KQ$, where $Y' \in s\Set \subset s\S$ and the latter map has $\rlp(I_\KQ)$; by \cref{detect acyclic fibrations of sspaces}, this latter map is also in $\bW_\KQ$.

\item

We use the small object argument for $I_\KQ$ to obtain the factorization $Y' \cofibn Z' \ra Z$ in $s\S_\KQ$ of the composite map $Y' \ra Y \ra Z$, where $Z' \in s\Set \subset s\S$ and the latter map has $\rlp(I_\KQ)$; again by \cref{detect acyclic fibrations of sspaces}, this latter map is also in $\bW_\KQ$.

\item

Since the map $Z' \ra Z$ has $\rlp(I_\KQ)$, we are guaranteed a lift in the square
\[ \begin{tikzcd}
\partial \Delta^n \arrow{r} \arrow[tail]{d} & Y' \arrow{r} & Z' \arrow{d}{\rlp(I_\KQ)} \\
\Delta^n \arrow{rr} \arrow[dashed, bend right=10]{rru} & & Z .
\end{tikzcd} \]

\item

We use the small object argument for $J_\KQ$ to obtain the factorization $Y' \wcofibn Y'' \fibn Z'$ in $s\Set_\KQ \subset s\S_\KQ$ of the map $Y' \ra Z'$.

\item

Since the map $Y \ra Z$ has $\rlp(J_\KQ)$, we are guaranteed a lift in the square
\[ \begin{tikzcd}
Y' \arrow{rr}{\approx} \arrow[tail]{d}[sloped, anchor=north]{\approx} & & Y \arrow{d}{\rlp(J_\KQ)} \\
Y'' \arrow{r} \arrow[dashed, bend right=10]{rru} & Z' \arrow{r} & Z .
\end{tikzcd} \]
Since $\bW_\KQ$ has the two-out-of-three property, then this lift is in $\bW_\KQ$.

\item

Again since $\bW_\KQ$ has the two-out-of-three property, the map $Y'' \ra Z'$ must be in $\bW_\KQ$.

\item

In $s\Set_\KQ$ we have that $(\bW \cap \bF)^{s\Set}_\KQ = \rlp(I^{s\Set}_\KQ)$, so we must have a lift in the square
\[ \begin{tikzcd}
\partial \Delta^n \arrow{r} \arrow[tail]{d} & Y' \arrow{r} & Y'' \arrow[two heads]{d}[sloped, anchor=south]{\approx} \\
\Delta^n \arrow{rr} \arrow[dashed, bend right=10]{rru} & & Z' ,
\end{tikzcd} \]
which gives us the dotted arrow in the diagram in \cref{diagram in the proof that acyclic fibrations have rlp(I)}.
\qedhere

\end{enumerate}

\end{proof}

\section{The proof of \cref{crazy lemma}}\label{section proof of crazy lemma}

We now give the proof of \cref{crazy lemma}, completing the proof of \cref{kan--quillen model structure on sspaces}.

\begin{proof}[Proof of \cref{crazy lemma}]
We begin by observing that we have $L' \we L$ in $s\Set_\KQ$ (for any choice of $i \geq 0$) because this model category is left proper, so that these objects are both homotopy pushouts.

To prove the rest of the statement, we proceed in steps for clarity.  To fix notation, suppose that the chosen point of the pullback selects a pair of maps
\[ (\varphi, \varepsilon) \in \hom_{s\S}(K,W) \times \hom_\S(|L|,|W|) . \]
We present the $\infty$-category $s\S$ via the model category $s(s\Set_\KQ)_\Reedy$, and we denote by $\const : s\Set_\triv \ra s(s\Set_\KQ)_\Reedy$ the evident right Quillen functor modeling the right adjoint $\disc:s\Set \ra s\S$ (though we will continue to suppress the latter).

\begin{enumerate}

\item\label{choose representatives} We choose an arbitrary fibrant representative $\ttW \in s(s\Set_\KQ)^f_\Reedy$ for the object $W \in s\S$, and then we choose an arbitrary map $\const(K) \xra{\phi} \ttW$ in $s(s\Set_\KQ)_\Reedy$ which presents the map $K \xra{\varphi} W$ in $s\S$.

\item Recall that the diagonal functor $\diag : \bD^{op} \ra \bD^{op} \times \bD^{op}$ induces an adjunction
\[ \diag_! : s\Set \adjarr ss\Set : \diag^* .\]
Applying its right adjoint to $\phi$ yields a map
\[ K \cong \diag^*(\const(K)) \xra{\diag^*(\phi)} \diag^*(\ttW) \]
in $s\Set$.  Since this right adjoint $\diag^* : s(s\Set_\KQ)_\Reedy \ra s\Set_\KQ$ is a relative functor (see e.g.\! \cite[Chapter IV, Proposition 1.9]{GJ}) and considered in $\strrelcat_\BarKan$ models the functor $|{-}| : s\S \ra \S$ (see e.g.\! \cite[Chapter IV, Exercise 1.6]{GJ} or Example T.A.2.9.31), then the map $\diag^*(\phi)$ in $s\Set_\KQ$ models the map $|K| \xra{|\varphi|} |W|$ in $\S$.

\item\label{choose map from L to fibt replacement of diagonal of W'} By assumption, we have an extension
\[ \begin{tikzcd}
|K| \arrow{r}{|\varphi|} \arrow{d} & |W| \\
|L| \arrow[dashed]{ru}[swap]{\varepsilon}
\end{tikzcd} \]
in $\S$.  Since the $\infty$-category $\S_{|K|/}$ is presented by the model category $(s\Set_{K/})_\KQ$, then this can be modeled as an extension
\[ \begin{tikzcd}
\partial \Delta^n \arrow{r} \arrow{d} & K \arrow{r}{\diag^*(\phi)} \arrow{d} & \diag^*(\ttW) \arrow{d}[sloped, anchor=south]{\approx} \\
\Delta^n \arrow{r} & L \arrow[dashed]{r}[swap]{\epsilon} & \Ex^\infty(\diag^*(\ttW))
\end{tikzcd} \]
in $s\Set_\KQ$.  To simplify our diagrams we will henceforth omit $L$, since it's defined as a pushout anyways.

\item\label{factor through a finite stage of Ex} Recall that the map $\diag^*(\ttW) \we \Ex^\infty(\diag^*(\ttW))$ is defined as a transfinite composition
\[ \diag^*(\ttW) \we \Ex(\diag^*(\ttW)) \we \Ex^2(\diag^*(\ttW)) \we \cdots \]
in $s\Set_\KQ$.  Since $\Delta^n$ is small as an object of $s\Set_{\partial \Delta^n/}$, there must exist a factorization
\[ \begin{tikzcd}
\partial \Delta^n \arrow{r} \arrow{d} & K \arrow{r}{\diag^*(\phi)} & \diag^*(\ttW) \arrow{d}[sloped, anchor=south]{\approx} \\
\Delta^n \arrow[dashed]{rr} \arrow{rrd} & & \Ex^i(\diag^*(\ttW)) \arrow{d}[sloped, anchor=south]{\approx} \\
& & \Ex^\infty(\diag^*(\ttW))
\end{tikzcd} \]
in $s\Set_\KQ$ for some $i < \infty$.

\item\label{use adjunction to get from Ex to sd} Via the adjunction $\sd^i : s\Set \adjarr s\Set : \Ex^i$, the extension in step \ref{factor through a finite stage of Ex} is equivalent to the extension
\[ \begin{tikzcd}
\sd^i(\partial \Delta^n) \arrow{r} \arrow{d} & \partial \Delta^n \arrow{r} & K \arrow{rr}{\diag^*(\phi)} & & \diag^*(\ttW) \\
\sd^i(\Delta^n) \arrow{r} \arrow[dashed]{rrrru} & \Delta^n , \arrow[leftarrow, crossing over]{u}
\end{tikzcd} \]
i.e.\! an extension
\[ \begin{tikzcd}
K \arrow{rr}{\diag^*(\phi)} \arrow{d} & & \diag^*(\ttW) . \\
L' \arrow[dashed]{rru}
\end{tikzcd} \]

\item\label{use adjunction induced by diagonal} Via the adjunction $\diag_! : s\Set \adjarr ss\Set : \diag^*$, the extension in step \ref{use adjunction to get from Ex to sd} is equivalent to the extension
\[ \begin{tikzcd}
\diag_!(\sd^i(\partial \Delta^n)) \arrow{r} \arrow{d} & \diag_!(\partial \Delta^n) \arrow{r} & \diag_!(K) \arrow{rr}{\diag^*(\phi)^\sharp} & & \ttW . \\
\diag_!(\sd^i(\Delta^n)) \arrow{r} \arrow[dashed, bend right=3]{rrrru} & \diag_!(\Delta^n) . \arrow[leftarrow, crossing over]{u}
\end{tikzcd} \]

\item\label{extend to a diagram including constant sssets} The diagram in step \ref{use adjunction induced by diagonal} extends to a diagram
\[ \begin{tikzcd}[column sep=0cm, row sep=1.5cm]
\diag_!(\sd^i(\partial \Delta^n)) \arrow{rr} \arrow{dd} \arrow{rd} & & \diag_!(\partial \Delta^n) \arrow{rr} \arrow{rd} & & \diag_!(K) \arrow{rr}{\diag^*(\phi)^\sharp} & & \ttW \arrow[leftarrow, dashed, bend left=15]{ddllllll} \\
& \const(\sd^i(\partial \Delta^n)) \arrow{rr} \arrow[crossing over]{dd} & & \const(\partial \Delta^n) \arrow[crossing over]{rr} \arrow[crossing over]{dd} & & \const(K) \arrow{ru}[swap, pos=0.6]{\phi} \arrow[leftarrow, crossing over]{lu} \\
\diag_!(\sd^i(\Delta^n)) \arrow[crossing over]{rr} \arrow{rd} & & \diag_!(\Delta^n) \arrow[leftarrow, crossing over]{uu} \arrow{rd} \\
& \const(\sd^i(\Delta^n)) \arrow{rr} & & \const(\Delta^n)
\end{tikzcd} \]
in $s(s\Set_\KQ)_\Reedy$, in which all unlabeled oblique arrows are components of the natural transformation
\[ \diag_! \cong \diag_! \diag^* \const \ra \const \]
in $\Fun(s\Set,ss\Set)$ induced by the counit of the adjunction $\diag_! \adj \diag^*$; indeed, the counit is initial among maps of the form $\diag^*(-)^\sharp$.

\item The map $\diag_!(\sd^i(\Delta^n)) \ra \const(\sd^i(\Delta^n))$ in the diagram of step \ref{extend to a diagram including constant sssets} is a weak equivalence in $s(s\Set_\KQ)_\Reedy$ by \cref{acyclic ssets have left kan extensions equivalent to const} below.  Hence, upon applying the localization $s(s\Set_\KQ)_\Reedy \ra s\S$ to that diagram, we obtain the desired extension
\[ \begin{tikzcd}
\sd^i(\partial \Delta^n) \arrow{r} \arrow{d} & \partial \Delta^n \arrow{r} & K \arrow{r}{\varphi} & W \\
\sd^i(\Delta^n) \arrow{r} \arrow[dashed, bend right=5]{rrru} & \Delta^n \arrow[leftarrow, crossing over]{u}
\end{tikzcd} \]
in $s\S$.  (Working back through the proof, it is clear that this does indeed model the extension
\[ \begin{tikzcd}
& |K| \arrow{r}{|\varphi|} \arrow{d} \arrow{dl} & |W| \\
|L'| \arrow{r}[swap]{\sim} & |L| \arrow[dashed]{ru}[swap]{\varepsilon}
\end{tikzcd} \]
in $\S$.)
\qedhere

\end{enumerate}

\end{proof}

\begin{lem}\label{acyclic ssets have left kan extensions equivalent to const}
For any acyclic object $M \in s\Set_\KQ$, the component
\[ \diag_!(M) \cong \diag_!(\diag^*(\const(M))) \ra \const(M) \]
of the counit of the adjunction $\diag_! : s\Set \adjarr ss\Set : \diag^*$ is a weak equivalence in $s(s\Set_\KQ)_\Reedy$.
\end{lem}

\begin{proof}
We begin by choosing a presentation of $M$ as a transfinite composition of pushouts of the generating acyclic cofibrations $J_\KQ = \{ \Lambda^n_i \ra \Delta^n \}_{0 \leq i \leq n \geq 1}$.  Note that $\diag_!$ is a left adjoint, so it commutes with pushouts; thus, this also gives us a presentation of $\diag_!(M)$ as a transfinite composition of pushouts of maps in $\diag_!(J_\KQ) = \{ \diag_!(\Lambda^n_i) \ra \diag_!(\Delta^n)\}_{0 \leq i \leq n \geq 1}$.

We now argue by transfinite induction.  Clearly the result holds if $M = \Delta^0$.  To obtain the inductive step at any successor ordinal, we will show below that we have a commutative square
\[ \begin{tikzcd}
\diag_!(\Lambda^n_i) \arrow{r}{\approx} \arrow[tail]{d} & \const(\Lambda^n_i) \arrow[tail]{d} \\
\diag_!(\Delta^n) \arrow{r}[swap]{\approx} & \const(\Delta^n)
\end{tikzcd} \]
in $s(s\Set_\KQ)_\Reedy$.  Then, since $s(s\Set_\KQ)_\Reedy$ is left proper (for instance because all its objects are cofibrant), the induced map between the pushouts of the front and back faces in the diagram
\[ \begin{tikzcd}
& \const(\Lambda^n_i) \arrow{rr} \arrow[tail]{dd} & & \const(M) \\
\diag_!(\Lambda^n_i) \arrow{ru}[sloped, pos=0.6]{\approx} \arrow[crossing over]{rr} \arrow[tail]{dd} & & \diag_!(M) \arrow{ru}[sloped, pos=0.6]{\approx} \\
& \const(\Delta^n) \\
\diag_!(\Delta^n) \arrow{ru}[sloped, pos=0.6]{\approx}
\end{tikzcd} \]
in $s(s\Set_\KQ)_\Reedy$ will again be a weak equivalence.  To obtain the inductive step at any limit ordinal, we observe that both colimits and weak equivalences in $s(s\Set_\KQ)_\Reedy$ are defined levelwise, and that weak equivalences in $s\Set_\KQ$ are closed under transfinite composition (for instance by arguments in the style of steps \ref{choose map from L to fibt replacement of diagonal of W'}-\ref{use adjunction to get from Ex to sd} in the proof of \cref{crazy lemma} above).

So, it only remains to show that we have a commutative square in $s(s\Set_\KQ)_\Reedy$ as claimed above.  We verify the illustrated assertions in turn.

\begin{itemize}

\item

Both vertical maps are monomorphisms and hence are cofibrations in $s(s\Set_\KQ)_\Reedy$.

\item

We have an isomorphism $\diag_!(\Delta^n) \cong \Delta^n \extprod \Delta^n$, under which identification the lower map is given by $\Delta^n \extprod \Delta^n \ra \Delta^n \extprod \Delta^0$.  In level $j$ this is just the map $\coprod_{(\Delta^n)_j} \Delta^n \ra \coprod_{(\Delta^n)_j} \Delta^0$, which is a weak equivalence in $s\Set_\KQ$.  So the lower map is indeed a weak equivalence in $s(s\Set_\KQ)_\Reedy$.

\item

To see that the upper map is also a weak equivalence, we recall the explicit description of $\diag_!(\Lambda^n_i)$ (given both in the proof of \cite[Lemma 1.3]{Moer} and in the text leading up to \cite[Chapter IV, Lemma 3.10]{GJ}), that
\[ \diag_!(\Lambda^n_i)_{j,k} \cong \left\{ (\alpha,\beta) \in \hom_\bD([j],[n]) \times \hom_\bD([k],[n]) :
\begin{array}{c}
\textup{there exists some } l \in [n] \textup{ with } l \not= i  \\ \textup{which is not in the image of } \alpha \textup{ or } \beta
\end{array}
\right\} . \]
For any fixed $j \geq 0$, we must show that the map $\diag_!(\Lambda^n_i)_j \ra \const(\Lambda^n_i)_j$ is a weak equivalence in $s\Set_\KQ$.  The latter object is discrete (i.e.\! its only nondegenerate simplices are 0-simplices), and so this is equivalent to showing that the preimage of each such 0-simplex is acyclic in $s\Set_\KQ$.  Such a 0-simplex is precisely the datum of a map $\alpha \in \hom_\bD([j],[n]) \cong \hom_{s\Set}(\Delta^j,\Delta^n)$ whose image on 0-simplices does not cover $(\Delta^n)_0 \backslash \{ \Delta^{\{i\}} \}$.  Define the subset $T \subset (\Delta^n)_0$ to be the union of $\{\Delta^{\{i\}}\}$ and the image of $\alpha$, and let $T^c \subset (\Delta^n)_0$ be its complement.  Then, the preimage of the 0-simplex of $\const(\Lambda^n_i)_j$ corresponding to $\alpha$ is the subobject of $\Delta^n \in s\Set$ consisting of those simplices whose 0-simplices do not contain all of $T^c \subset (\Delta^n)_0$, which is indeed acyclic in $s\Set_\KQ$ since $T^c$ is nonempty.
\qedhere
\end{itemize}
\end{proof}

\begin{rem}
\cref{acyclic ssets have left kan extensions equivalent to const} fails drastically if we do not assume that $M \in s\Set_\KQ$ is acyclic.  In fact, in \cref{compare with moerdijk}, the stark difference between the sets of homotopy classes of maps $I^{s\S}_\Moer$ and $I^{s\S}_\KQ$ in $s\S$ is precisely due to the difference between (the weak equivalences classes of) the objects $\diag_!(\partial \Delta^n) \cong \partial \Delta^n \extprod \partial \Delta^n$ and $\const(\partial \Delta^n)$ in $s(s\Set_\KQ)_\Reedy$.
\end{rem}

\begin{rem}
Using the arguments of \cref{compare with moerdijk}, one can use \cref{acyclic ssets have left kan extensions equivalent to const} to give an alternative proof of \cref{consequence of right properness of kan--quillen model structure}.  The key point is that we have a diagram
\[ \begin{tikzcd}
\diag_!(\Lambda^n_i) \arrow{r}{\approx} \arrow[tail]{d} & \const(\Lambda^n_i) \arrow[tail]{d} \\
\diag_!(\Delta^n) \arrow{r}[swap]{\approx} & \const(\Delta^n)
\end{tikzcd} \]
in $s(s\Set_\KQ)_\Reedy$.  Hence, if the map $Y \ra Z$ in $s\S$ has $\rlp(J_\KQ)$, then it can be presented by a fibration $\ttY \fibn \ttZ$ in $s(s\Set_\KQ)_\Reedy$ that additionally has $\rlp(\diag_!(J^{s\Set}_\KQ))$, i.e.\! that is also in $\bF^{ss\Set}_\Moer$.  Since both $s(s\Set_\KQ)_\Reedy$ and $ss\Set_\Moer$ are right proper, it follows that all pullbacks of the map $\ttY \ra \ttZ$ in $ss\Set$ simultaneously compute homotopy pullbacks in both model structures.
\end{rem}

\appendix

\section{Notation, terminology, and conventions}\label{section conventions}

As this is the first paper in its series, in this appendix we take the opportunity to spell out once and for all the precise foundations on which the project is built.

\subsection{}\label{subsection foundations}

We begin with our philosophy surrounding the semantics of the signifier ``$\infty$-category''.

\begin{enumerate}

\item\label{ground in qcats}

For definiteness, we ground ourselves in the theory of \textit{quasicategories}: an $\infty$-category is a quasicategory.  We will refer to these as ``quasicategories'' only when we mean to make specific reference to their properties or manipulation as such, which we will avoid doing to the largest extent possible.  We use \cite{LurieHTT} as our primary reference, but we note that many of the ideas given there have their origins in \cite{Joyal-qcats-1, Joyal-qcats-2, Joyal-notes-on-qcats}.

In order to proceed with the enumeration of our foundations, we must immediately lay out the following basic conventions.

\begin{enumeratesub}

\item

We will be ignoring all set-theoretic issues.  They are irrelevant to our aims, and in any case can be dispensed with by appealing to the usual device of \textit{Grothendieck universes} (see e.g.\! \sec T.1.2.15).

\item\label{initial and terminal objs}

If an $\infty$-category $\C$ has an initial (resp.\! terminal) object, we will write $\es_\C$ (resp.\! $\pt_\C$) for any such object, or we will simply write $\es$ (resp.\! $\pt$) if the ambient $\infty$-category $\C$ is clear from the context.  For $\infty$-categories of co/pointed objects, we will make the abbreviations $\C_{\es} = \C_{/\es}$ and $\C_* = \C_{\pt/}$.

\item\label{conventions on spaces}

We write $\S$ for the $\infty$-category of spaces.  Up to equivalence, we can take this to be either $\loc{\Top}{\bW_\whe}$ or $\loc{s\Set}{\bW_\whe}$, where in both cases the symbol $\bW_\whe$ denotes the \textit{weak homotopy equivalences}.\footnote{Of course, by $\Top$ we mean to denote any ``convenient'' category of topological spaces.}\footnote{More invariantly, one can also characterize the $\infty$-category of spaces as the free cocompletion of the terminal $\infty$-category (see item \ref{yoneda}).}  In particular, by ``space'' we will mean an object of $\S$; when we mean to refer to an object of $\Top$, we will instead use the term ``topological space''.  The $\infty$-category $\S$ of spaces plays the same fundamental role in the theory of $\infty$-categories that the category $\Set$ of sets plays in the theory of categories: whereas categories are naturally enriched in sets, $\infty$-categories are naturally enriched in spaces (see item \ref{twisted arrow gives hom bifunctor}).

We adopt the following conventions regarding $\S$.

\begin{itemize}

\item

A map in $\S$ is called

\begin{itemizesmall}

\item \textit{\'{e}tale} if it induces a $\pi_{\geq 1}$-isomorphism for every basepoint of the source;

\item a \textit{monomorphism} (or, more informally, the inclusion of a \textit{subspace}) if it is \'{e}tale and additionally induces a $\pi_0$-monomorphism;

\item a \textit{surjection} if it induces a $\pi_0$-surjection.\footnote{Surjections are also called ``effective epimorphisms'', but note that they are \textit{not} generally epimorphisms in $\S$ (see item \ref{mono and epi}).  Rather, the only epimorphisms in $\S$ are the equivalences.}

\end{itemizesmall}

\item

There is an evident adjunction $\pi_0 : \S \adjarr \Set : \disc$, and we call a space \textit{discrete} if it is in the image of the right adjoint (see item \ref{discrete object}).  We will only include this right adjoint in the notation if we mean to emphasize it.

\item

More generally, for any $n \geq 0$ we have a \textit{truncation} adjunction $\tau_{\leq n} : \S \adjarr \S^{\leq n} : \forget_{\leq n}$ and a \textit{cotruncation} adjunction $\forget_{\geq n} : \S^{\geq n}_* \adjarr \S^{\geq 1}_* : \tau_{\geq n}$.  (In the special case that $n=0$, the truncation adjunction reduces to the adjunction $\pi_0 : \S \adjarr \Set : \disc$ given above.)

\item

We write $\S^\fin \subset \S$ for the full subcategory on the \textit{finite} spaces.\footnote{A space is finite exactly when it can be presented either by a finite CW complex or by a finite simplicial set (i.e.\! a simplicial set with finitely many nondegenerate simplices).  More invariantly, one can also characterize $\S^\fin$ as the initial $\infty$-category admitting all \textit{finite} colimits.}

\item

We will refer to spaces as \textit{$\infty$-groupoids} when we mean to emphasize the fact that they are just particular examples of $\infty$-categories.

\end{itemize}

\item\label{conventions on infty-cats}

We write $\Cati$ for the $\infty$-category of $\infty$-categories.  We adopt the following conventions regarding $\Cati$.

\begin{itemize}

\item

A functor $\C \xra{F} \D$ of $\infty$-categories is an \textit{equivalence} precisely if it is (homotopically) fully faithful and surjective: that is,
\begin{itemizesmall}
\item for all $x,y \in \C$, the induced map
\[ \hom_\C(x,y) \ra \hom_\D(F(x),F(y)) \]
is an equivalence in $\S$, and moreover
\item for every $z \in \D$ there is some $w \in \C$ and an equivalence $F(w) \simeq z$ in $\D$.\footnote{This is essentially Definition T.1.1.5.14, which makes use of the left Quillen equivalence $s\Set_\Joyal \ra (\strcat_{s\Set})_\Bergner$ of Theorem T.2.2.5.1.  (Note that all objects of $s\Set_\Joyal$ are cofibrant.)}
\end{itemizesmall}

\item

To say that an $\infty$-category $\C$ is a \textit{subcategory} of some other $\infty$-category $\D$ means, in the most invariant possible language, that we have a chosen functor $\C \xra{F} \D$ which is (homotopically) faithful: that is, for all $x,y \in \C$, the induced map
\[ \hom_\C(x,y) \ra \hom_\D(F(x),F(y)) \]
is a monomorphism in $\S$.  We will call the functor $F$ the \textit{inclusion} of a subcategory, but we will usually suppress it from the notation and simply write $\C \subset \D$ as shorthand.\footnote{Note that these are not quite the monomorphisms in $\Cati$ (see item \ref{mono and epi}).  Rather, the monomorphisms are precisely the \textit{pseudomonic} functors, i.e.\! the inclusions of subcategories which are full on equivalences.  This is perhaps most easily seen by appealing to the equivalence $\Cati \simeq \CSS$ of item \ref{different model cats}\ref{CSS as infty-cat} below: as the inclusion $\CSS \subset s\S$ preserves limits (being a right adjoint), a map in $\CSS$ is a monomorphism precisely if it is a monomorphism when considered in $s\S$.}  A subcategory $\C \subset \D$ is uniquely specified by the resulting subcategory $\ho(\C) \subset \ho(\D)$ of its homotopy category.

\item More generally, if $I$ is a class of maps in $\S$, then a functor in $\Cati$ is called a \textit{local $I$} if all the induced maps on hom-spaces are in $I$.  (So for instance, the inclusion of a subcategory might otherwise be called a \textit{local monomorphism}.)

\item

An $\infty$-category will be called a \textit{category}, or sometimes a \textit{1-category} for emphasis, if its hom-spaces are discrete, i.e.\! they lie in the full subcategory $\Set \subset \S$.  These form a full subcategory $\Cat \subset \Cati$, the inclusion of which we will denote by $\forget_\Cat : \Cat \hookra \Cati$.  This inclusion is the right adjoint in an adjunction
\[ \ho : \Cati \adjarr \Cat : \forget_\Cat \]
whose left adjoint is given by the \textit{homotopy category} functor.  Given an $\infty$-category $\C$ and any pair of objects $c,d \in \C$, we will sometimes write
\[ [c,d]_\C = \hom_{\ho(\C)}(c,d) \]
for the corresponding hom-set in the homotopy category $\ho(\C)$ of $\C$.  By definition, the map
\[ \hom_\C(c,d) \ra \hom_{\ho(\C)}(c,d) \]
in $\S$ induced by projection $\C \ra \ho(\C)$ (i.e.\! the unit of the adjunction $\ho \adj \forget_\Cat$) is precisely the projection to the set of path components (i.e.\! the unit of the adjunction $\pi_0 \adj \disc$).

\item We write $\forget_\S : \S \hookra \Cati$ for the inclusion of spaces as $\infty$-groupoids.
\begin{itemize}
\item This inclusion is the left adjoint in an adjunction
\[ \forget_\S : \S \adjarr \Cati : (-)^\simeq \]
whose right adjoint is given by the \textit{maximal subgroupoid} functor.\footnote{The adjunction $\forget_\S : \S \adjarr \Cati : (-)^\simeq$ is presented by an adjunction of $s\Set$-enriched categories between that of Kan complexes and that of quasicategories, whose left adjoint is the canonical inclusion and whose right adjoint takes a quasicategory to the largest Kan complex that it contains (see Proposition T.1.2.5.3 and Corollary T.5.2.4.5).}
\item This inclusion is the right adjoint in an adjunction
\[ (-)^\gpd : \Cati \adjarr \S : \forget_\S \]
whose left adjoint is given by the (\textit{$\infty$-})\textit{groupoid completion} functor.\footnote{The adjunction $(-)^\gpd : \Cati \adjarr \S : \forget_\S$ is presented by the Quillen adjunction $\id_{s\Set} : s\Set_\Joyal \adjarr s\Set_\KQ : \id_{s\Set}$ (see item \ref{quillen adjunction induces underlying adjunction} (and Remark T.1.2.5.6)).}
\end{itemize}

\item We write $\Fun(\C,\D) \in \Cati$ for the $\infty$-category of functors from $\C$ to $\D$.  This is the internal hom in $(\Cati,\times)$, and admits a canonical equivalence $\Fun(\C,\D)^\simeq \simeq \hom_\Cati(\C,\D)$ in $\S$.\footnote{As the model category $s\Set_\Joyal$ is cartesian (as can easily be seen from Corollary T.2.2.5.4), the $\infty$-category of functors is presented therein by the internal hom in $(s\Set,\times)$.}

\item We write $(-)^{op} : \Cati \xra{\sim} \Cati$ for the involution given by taking opposites.

\end{itemize}

\end{enumeratesub}

\item\label{different model cats}

Despite our grounding declared in item \ref{ground in qcats}, our notion of ``$\infty$-category'' is nevertheless a rather flexible one: over the course of these papers, we will interchange fluidly between a number of distinct but essentially equivalent notions thereof.  In accordance with current best practices, those that we will employ all appear naturally as objects in various model categories.  For the reader's convenience, we itemize these notions and their ambient model categories here, and we give some indication of the roles that they will play in this series of papers.

\begin{enumeratesub}

\item\label{qcat as infty-cat}

The notion of a \textit{quasicategory} plays a distinguished role in these papers, as indicated in item \ref{ground in qcats}.  These are precisely the bifibrant objects in $s\Set_\Joyal$, the category of simplicial sets equipped with the \textit{Joyal model structure} of Theorem T.2.2.5.1.  We view these as the most convenient of the notions to employ as an ambient framework, which advantage is surely in large part due to the abundance of theory that has been built up around them.

\item\label{sset-enr cat as infty-cat}

The notion which most closely adheres to the intuition of a ``category enriched in spaces'' is that of a \textit{category enriched in simplicial sets}, or simply a \textit{$s\Set$-enriched category} for short.  These organize into the model category $(\strcat_{s\Set})_\Bergner$ under the \textit{Bergner model structure} of \cite[Theorem 1.1]{Bergner} (or see Proposition T.A.3.2.4 for a generalization).  Just as when one uses simplicial sets to present spaces one should generally be working with Kan complexes, when considering a $s\Set$-enriched category as an $\infty$-category one should generally be working with a category which is in fact enriched in Kan complexes: indeed, these are precisely the fibrant objects of $(\strcat_{s\Set})_\Bergner$ (which sits in a Quillen equivalence $\mf{C} : s\Set_\Joyal \adjarr (\strcat_{s\Set})_\Bergner : \Nervehc$ (see Theorem T.2.2.5.1)), though note that not all objects are cofibrant.  This model category provides an explicit bridge from $\strrelcat_\BarKan$ to $s\Set_\Joyal$ (see subitem \ref{different model cats}\ref{relcat as infty-cat}).

\item\label{CSS as infty-cat}

The notion which is most ``homotopy invariant'' is that of a \textit{complete Segal space}.  These are actually bisimplicial sets, thought of as simplicial spaces via choices of distinguished ``simplicial'' and ``geometric'' directions.  They are precisely the bifibrant objects in $ss\Set_\Rezk$, the category of bisimplicial sets equipped with the \textit{Rezk model structure} of \cite[Theorem 7.2]{RezkCSS} (there called the ``complete Segal space'' model structure).

However, it is also fruitful to consider a theory of complete Segal spaces \textit{internally} to the world of $\infty$-categories, i.e.\! to define them as a subcategory $\CSS \subset s\S$ of the $\infty$-category of simplicial spaces.\footnote{This perspective is explored in detail (and in greater generality) in \cite[\sec 1]{LurieGoo}.}  From this viewpoint, a complete Segal space can be thought of as a homotopical analog of the \textit{nerve} of a category: the equivalence $\Nervei : \Cati \xra{\sim} \CSS$ takes an $\infty$-category $\C$ to its \textit{$\infty$-categorical nerve}, namely the simplicial space
\[ \Nervei(\C)_\bullet = \hom^\lw_{\Cati}([\bullet],\C) \]
(i.e.\! the levelwise hom-space from the standard cosimplicial category $[\bullet] : \bD \hookra \Cat$).\footnote{Indeed, a simplicial set is the nerve of a category precisely if it satisfies the Segal condition.}\footnote{Note that the $\infty$-categorical nerve of a 1-category does \textit{not} generally coincide with its 1-categorical nerve.}\footnote{Throughout \cref{section conventions}, for the sake of clarity we will exclusively refer to this construction as the ``$\infty$-categorical nerve''.  However, it appears quite frequently in the remaining papers in this sequence (beginning with its reintroduction in \cref{rnerves:section CSSs}), and so for brevity we will omit the modifier ``$\infty$-categorical'' in those sequels except when we mean to emphasize the distinction.}  The inverse equivalence takes a complete Segal space $Y_\bullet \in \CSS$ to the $\infty$-category
\[ \int^{[n] \in \bD} Y_n \times [n] , \]
(where we implicitly consider $Y_n \in \Cati$ via the inclusion $\forget_\S : \S \hookra \Cati$).\footnote{This formula follows from \cite[Corollary 4.3.15]{LurieGoo}, but note that this is ultimately just an instance of the ``generalized nerve/realization'' Quillen equivalence first proved as \cite[Theorem 3.1]{DKSingReal}.}  In fact, this inclusion is the right adjoint in an adjunction $\leftloc_\CSS : s\S \adjarr \CSS : \forget_\CSS$.\footnote{This adjunction is presented by the left Bousfield localization $\id_{ss\Set} : s(s\Set_\KQ)_\Reedy \adjarr ss\Set_\Rezk : \id_{ss\Set}$ (see item \ref{quillen adjunction induces underlying adjunction}).}  It is fruitful to think of the resulting composite adjunction
\[ \begin{tikzcd}[column sep=1.5cm]
s\S 
{\arrow[transform canvas={yshift=0.7ex}]{r}{\leftloc_\CSS}[swap, transform canvas={yshift=0.25ex}]{\scriptstyle \bot} \arrow[transform canvas={yshift=-0.7ex}, hookleftarrow]{r}[swap]{\forget_\CSS}}
& \CSS
{\arrow[transform canvas={yshift=0.7ex}]{r}{\Nervei^{-1}}[swap, transform canvas={yshift=0.05ex}]{\sim} \arrow[transform canvas={yshift=-0.7ex}, leftarrow]{r}[swap]{\Nervei}}
& \Cati
\end{tikzcd} \]
as being a homotopical analog of the usual ``nerve/homotopy category'' adjunction
\[ \begin{tikzcd}[column sep=1.5cm]
s\Set
{\arrow[transform canvas={yshift=0.7ex}]{r}{\leftloc_\strcatsup}[swap, transform canvas={yshift=0.25ex}]{\scriptstyle \bot} \arrow[transform canvas={yshift=-0.7ex}, hookleftarrow]{r}[swap]{\Nerve}}
& \strcat
\end{tikzcd} \]
(see subitem \ref{work invariantly with qcats}\ref{strict cats}).


\item\label{relcat as infty-cat}

Finally, the simplest notion is that of a \textit{relative category}.  These organize into the model category $\strrelcat_\BarKan$ under the \textit{Barwick--Kan model structure} of \cite[Theorem 6.1]{BK-relcats}.  We write $\hamd : \strrelcat_\BarKan \ra (\strcat_{s\Set})_\Bergner$ for the hammock localization functor, which is a relative functor (see \cite[Theorem 1.8]{BK-simploc}); in fact, it is even a weak equivalence in $\strrelcat_\BarKan$ (see \cite[Theorem 1.7]{BK-simploc} and item \ref{equivalence of notions}).\footnote{In fact, the Rezk nerve functor $\NerveRezk : \strrelcat_\BarKan \ra ss\Set_\Rezk$ (see \cite[3.3]{RezkCSS}, where it is called the ``classification diagram'' functor) \textit{also} creates the weak equivalences by \cite[Theorem 6.1(i)]{BK-relcats}.  However, the objects in its image are not generally fibrant, even up to weak equivalence in $s(s\Set_\KQ)_\Reedy$ (see \cite{LowMG}).}\footnote{The letter $\delta$ in the notation $\hamd$ stands for ``discrete'': in \cite{MIC-hammocks} we study an $\infty$-categorical version of this functor, which we denote simply by $\ham$.}

We mainly use relative categories (and the Barwick--Kan model structure) as a technical device that allows us to make rigorous sense of the \textit{underlying $\infty$-category} of a relative category (in particular of a model category).  In this situation, we say that the model category gives a \textit{presentation} of its underlying $\infty$-category.  See \cref{subsection model cats present infty-cats} for details regarding our usage of model categories as presentations of $\infty$-categories.

\end{enumeratesub}

\item\label{equivalence of notions}

The assertion made in item \ref{different model cats} that these various notions of $\infty$-categories are all ``essentially equivalent'' is rather multifaceted.  We therefore give a careful account of this assertion.  Our perspective is espoused in a number of relatively recent papers, notably \cite{BSP} (from the introduction of which this item is more or less directly lifted), and seems to represent the emerging consensus among practitioners of higher category theory.

First of all, these four model categories are all connected by a diagram of Quillen equivalences along with the weak equivalence $\hamd : \strrelcat_\BarKan \we (\strcat_{s\Set})_\Bergner$ in $\strrelcat_\BarKan$ (see \cite[Figure 1]{BSP} and the references cited therein).  Thus, any homotopically meaningful manipulations that we might make using one of these notions can equally well be made using any other notion.

However, there is still cause for potential concern: the diagram of \cite[Figure 1]{BSP} does not commute, even up to natural isomorphism.  However, a moment's reflection should reassure us that this is a stronger request than we should really be making: after all, we would generally like to consider objects of a model category up to \textit{weak equivalence}, not up to isomorphism. Thus, it is helpful to reinterpret this diagram \textit{within} one of the given model categories.  Rather than choose a particular one, we will simply refer to objects of this model category as `$\infty$-categories' (with scare-quotes) for the remainder of the item; as explained in item \ref{quillen adjunction induces underlying adjunction}, Quillen equivalences between model categories induce weak equivalences of underlying `$\infty$-categories'.

This conceptual leap leads us to the alternative point of view that what we are looking at is a not-necessarily-commutative diagram of weak equivalences of `$\infty$-categories'.  This may not seem like an improvement in and of itself, but in fact we are saved by the following remarkable facts (\cite[Th\'{e}or\`{e}me 6.3]{Toen}, reproved as \cite[Theorem 4.4.1]{LurieGoo} and generalized as \cite[Theorem 8.2]{BSP}), originally stated within the model category $ss\Set_\Rezk$ of complete Segal spaces.
\begin{itemize}
\item The `$\infty$-category' of complete Segal spaces (i.e.\! the `$\infty$-category' corresponding to $ss\Set_\Rezk$) -- and hence any `$\infty$-category' weakly equivalent to it -- has a \textit{discrete} derived automorphism space, which is equivalent as a group to $\bbZ/2$.
\item Furthermore, the unique nontrivial derived automorphism of this `$\infty$-category' is given by the involution of taking opposites, and is therefore detected by considering its restriction to the full subcategory generated by the objects $[0],[1] \in \Cat$ (considered as objects of each of these various model categories).
\end{itemize}
It now follows readily that the diagram of \cite[Figure 1]{BSP} commutes \textit{as a diagram in an $\infty$-category}: more precisely, as a diagram internal to the quasicategory corresponding to the ambient model category of `$\infty$-categories'.

\end{enumerate}

\subsection{}

We now establish some conventions surrounding our usage of $\infty$-categories.

\begin{enumerate}[resume]

\item\label{work invariantly with qcats}

As a rule, the statements we make will generally be invariant under equivalence of $\infty$-categories.  In fact, when we make statements about $\infty$-categories, we will generally mean to be working \textit{in the $\infty$-category of $\infty$-categories}.  However, this is only a matter of taste: the sufficiently motivated reader should readily be able to turn our invariant arguments about $\infty$-categories into simplex-by-simplex arguments about quasicategories and model-categorical arguments in $s\Set_\Joyal$.

The choice of such a foundational regime compels us to lay out the following related conventions.

\begin{enumeratesub}

\item\label{omit essential}

When working $\infty$-categorically, we will omit the modifier ``essential'' (and its variants) wherever it might be used in its technical capacity.  For instance, we simply say \textit{unique} where one might otherwise say ``essentially unique'': in the invariant world, the adjective ``unique'' has no other possible meaning.

\item\label{reserve equals sign}

We reserve the symbol $=$ to indicate nothing other than
\begin{itemizesmall}
\item that some equivalence holds by definition, or
\item the equality of two elements of a set, and in particular
\begin{itemizesmall}
\item the equivalence of two subobjects of a given object (see item \ref{mono and epi}).
\end{itemizesmall}
\end{itemizesmall}
Along these same lines, whereas we generally use the symbol $\simeq$ to denote an equivalence in an arbitrary $\infty$-category, if that $\infty$-category is in fact a 1-category then we may instead write $\cong$ (and refer to the equivalence as an \textit{isomorphism}).

\item\label{strict cats}

We will have to be slightly careful with our definition of ordinary categories: for instance, we will sometimes want to refer to the nerve of a 1-category, but this is not a well-defined operation on the full subcategory $\Cat \subset \Cati$: for example, the notion of ``the set of objects'' is not invariant under equivalence of categories.

Thus, we will use the term \textit{strict category} (or even \textit{strict 1-category}) to mean a simplicial set satisfying the Segal condition.  These assemble into a full subcategory $\strcat \subset s\Set$.  This 1-category of categories now admits a nerve functor
\[ \strcat \xra{\Nerve} s\Set , \]
although this is entirely cosmetic: according to our definition, it is simply the defining inclusion.  Note that this inclusion sits as the right adjoint in a left localization adjunction
\[ \begin{tikzcd}[column sep=1.5cm]
s\Set
{\arrow[transform canvas={yshift=0.7ex}]{r}{\leftloc_\strcatsup}[swap, transform canvas={yshift=0.25ex}]{\scriptstyle \bot} \arrow[transform canvas={yshift=-0.7ex}, hookleftarrow]{r}[swap]{\Nerve}}
& \strcat ,
\end{tikzcd} \]
and hence commutes with limits.

The 1-category of strict categories also admits a functor $\strcat \xra{\forget_\strcatsupsup} \Cat$, namely the factorization of the composite functor
\[ \strcat \xhookra{\Nerve} s\Set \ra \loc{s\Set}{(\bW_\Joyal)} \simeq \Cati \]
through its image, but this is \textit{not} the inclusion of a subcategory. 
On the other hand, the \textit{gaunt} objects of $\strcat$ -- that is, those in which every isomorphism is in fact an identity morphism -- do include as a full subcategory of $\Cat$.  In fact, the map $\hom_\strcatsup(\C,\D) \ra \hom_\Cat(\C,\D)$ is an equivalence in $\S$ whenever $\D$ is gaunt.  Note in particular that we obtain full inclusions
\[ \begin{tikzcd}
\bD \arrow[hook]{rr} \arrow[hook]{rd} & & \Cat . \\
& \strcat \arrow{ru}
\end{tikzcd} \]
In fact, note further that we can consider $\bD$ \textit{itself} as a strict category: in such situations, we will take $\bD$ to be \textit{skeletal}, i.e.\! to be the full subcategory $\bD \subset \strcat$ on the gaunt categories $[n] \in \strcat$ (rather than the full subcategory on \textit{all} finite nonempty totally ordered sets).

The functor $\strcat \ra \Cat$ does not preserve monomorphisms.  At the risk of confusion, we will nevertheless use the same notation $\C \subset \D$ to indicate a monomorphism in $\strcat$. 

In contrast with subitem \ref{work invariantly with qcats}\ref{omit essential}, we \textit{will} use terms such as ``essentially surjective'' to refer to maps in $\strcat$, since otherwise the meaning would be ambiguous.

We will similarly speak of \textit{strict groupoids}, \textit{strict relative categories}, etc., likewise shrinking the capital letters in their names as $\strgpd$, $\strrelcat$, etc.\! to indicate the 1-categories of these.  (For us, $s\Set$-enriched categories will \textit{always} be strict; the 1-category of these is correspondingly denoted by $\strcat_{s\Set}$.)  We will also write
\[ \strfun(-,- ) : \strcat^{op} \times \strcat \ra \strcat \]
for the internal hom bifunctor in $(\strcat,\times)$, which can be computed by the internal hom bifunctor in $(s\Set,\times)$; since $s\Set_\Joyal$ is cartesian, for any $\C,\D \in \strcat$ we have a canonical equivalence
\[ \forget_\strcatsup ( \strfun(\C,\D)) \xra{\sim} \Fun(\forget_\strcatsup (\C),\forget_\strcatsup (\D)) \]
in $\Cat \subset \Cati$.

However, now that we have carefully clarified this distinction, we will often simply ignore it (e.g.\! referring to strict categories just as ``categories'') rather than overburden our terminology, unless it warrants emphasis.  Our meaning should always be clear from context.

\end{enumeratesub}

\item

An object $c \in \C$ defines a functor $\pt_\Cati \ra \C$.  For convenience, we will generally denote this functor by $\{ c \} \hookra \C$ (even though it will not generally be a monomorphism!); that is, we take the notation $\{ c \}$ to denote a terminal object of $\Cati$ which is equipped with a preferred map to $\C$.

\item\label{source and target}

We will generally write $s,t : \Fun([1],\C) \ra \C$ for the \textit{source} and \textit{target} maps, i.e.\! for the evaluation maps at $0 \in [1]$ and $1 \in [1]$ respectively.  Relatedly, as various flavors of categories will sometimes be defined as simplicial objects (recall e.g.\! subitems \ref{different model cats}\ref{CSS as infty-cat} and \ref{work invariantly with qcats}\ref{strict cats}), for such a simplicial object $Y_\bullet \in s\C$ we may write the two structure maps $Y_1 \rra Y_0$ in $\C$ as $\delta_1 = s$ and $\delta_0 = t$.

\item\label{evaluation notation}

In general, given an object $c \in \C$, we will write $\Fun(\C,\D) \xra{\ev_c} \D$ for the functor given by evaluation at $c$ (i.e.\! the pullback along the functor $\{ c \} \hookra \C$).

\item\label{constant diagram functor}

For $\infty$-categories $\I$ and $\C$, we will generally write $\const_\I : \C \ra \Fun(\I,\C)$ for the \textit{constant diagram} functor.  However, when it is clear from context, we may omit the subscript and simply write $\const : \C \ra \Fun(\I,\C)$.

\item\label{diagonal map}

For an object $c \in \C$, we write $\diag : c \ra c \times c$ for the diagonal map (if it exists).  This will usually be applied in the case that $\C = \Cati$.

\item\label{notation for adjunct bifunctor}

We will sometimes want to identify a bifunctor $\C \times \D \ra \E$ with its adjunct $\C \ra \Fun(\D,\E)$.  For clarity, if the original bifunctor is denoted by $F(-,-)$, then this adjunct will be denoted by $F(-,=)$.  That is, we will use the symbols $-$ and $=$ to respectively indicate the slot being filled first and the slot being considered as a free variable.  We will use similar notation for adjuncts of multivariable functors.

\item\label{colimit notn}

Given a functor $\I \xra{F} \C$, we will generally denote its colimit (if it exists), an object of $\C$, by
\begin{itemizesmall}
\item $\colim ( \I \xra{F} \C )$, or
\item $\colim_\I(F)$, or
\item ${\colim}^\C_\I(F)$ if we'd like to emphasize the $\infty$-category $\C$ in which the colimit is being taken, or
\item $\colim_{i \in \I}F(i)$ if we'd like to emphasize the functoriality of $F$ for $i \in \I$, or
\item ${\colim}^\C_{i \in \I}F(i)$ to combine the previous two notations.
\end{itemizesmall}
Dually, we will denote its limit (if it exists), also an object of $\C$, by $\lim(\I \xra{F} \C)$, or $\lim_\I(F) \in \C$, or $\lim^\C_\I(F) \in \C$, or $\lim_{i \in \I}F(i)$, or $\lim^\C_{i \in \I}F(i)$.

For convenience, we will often write $|{-}| = \colim_{\bD^{op}}(-)$ for any colimit functor $s\C = \Fun(\bD^{op},\C) \ra \C$ and refer to it as \textit{geometric realization}.  Similarly, we will often write $\|{-}\| = \colim_{\bD^{op} \times \bD^{op}}(-)$ for any colimit functor $ss\C = \Fun(\bD^{op} \times \bD^{op} , \C) \ra \C$.

\item\label{pushout/pullback and co/product notation}

Given a span $d \xla{\varphi} c \xra{\psi} e$ in an $\infty$-category $\C$, we may denote by
\[ d \coprod_{\varphi , c , \psi} e \]
its colimit (i.e.\! its pushout).  Dually, given a cospan $d \xra{\varphi} c \xla{\psi} e$ in an $\infty$-category $\C$, we may denote by
\[ d \underset{\varphi,c,\psi}{\times} e \]
its limit (i.e.\! its limit).  On the other hand, we may omit either or both of the maps from the subscript if they are clear from context.  Meanwhile, in the absolute cases, given a set of objects $\{ c_i \in \C \}_{i \in I}$, we may write
\[ \coprod_{i \in I} c_i \]
for their coproduct and
\[ \prod_{i \in I} c_i \]
for their product.

\item\label{left kan extn}

Given a functor $\I \xra{F} \J$ and an $\infty$-category $\C$, restriction along $F$ induces a functor $\Fun(\J,\C) \xra{F^*} \Fun(\I,\C)$.  In many cases (for instance if $\C$ is cocomplete) this admits a left adjoint, which we denote by $F_! : \Fun(\I,\C) \ra \Fun(\J,\C)$ and refer to as the \textit{left Kan extension} (along $F$) functor.  (See \sec T.4.3.)

\item\label{notation for co/ends}

We will occasionally use the theory of \textit{coends} and \textit{ends}: given a functor $\C^{op} \times \C \xra{F} \D$, we will denote its coend by
\[ \int^{c \in \C} F(c,c) \in \D \]
and its end by
\[ \int_{c \in \C} F(c,c) \in \D . \]
(We refer the reader to \cite[\sec 2]{GHN} for a brief review of the theory of co/ends in the $\infty$-categorical setting.)

\item

Suppose we are given a bifunctor $\I \times \J \xra{F} \C$.  Then, \textit{Fubini's theorem for colimits} asserts that we have a canonical equivalence
\[ \colim^\C_{(i,j) \in \I \times \J}F(i,j) \simeq \colim^\C_{i \in \I} \left( \colim^\C_{j \in \J} F(i,j) \right) \]
in $\C$, if either side exists.  This can be proved by the juggling of iterated coends (which explains the name), but it is really just a consequence of the observation that the composite
\[ \C \xhookra{\const_\J} \Fun(\J,\C) \xhookra{\const_\I} \Fun(\I , \Fun(\J,\C)) \simeq \Fun(\I \times \J,\C) \]
coincides with the functor
\[ \C \xhookra{\const_{\I \times \J}} \Fun(\I \times \J,\C) \]
(at least when $\C$ is cocomplete, or else embedding $\C$ into its free cocompletion for precisely those colimits that it lacks).

\item\label{natural equivalences are equivalences in functor category}

We will often implicitly use the fact (which is proved as \cite[Chapter 5, Theorem C]{Joyal-qcats-2} (combined with \cite[Proposition 4.8]{Joyal-qcats-2})) that a natural transformation between functors of $\infty$-categories is an equivalence precisely if it is a componentwise equivalence.

\item\label{limits and colimits in a functor category computed objectwise}

We will often implicitly use the fact (which is proved as Corollary T.5.1.2.3 and its dual) that co/limits in a functor $\infty$-category are computed objectwise.

\item

Given a functor $\I \ra \C$ that factors through the inclusion $\C^\simeq \subset \C$ of the maximal subgroupoid (i.e.\! that takes all maps in $\I$ to equivalences in $\C$), there exists a unique induced extension
\[ \begin{tikzcd}
\I \arrow{r} \arrow{d} \arrow{rd} & \C \\
\I^\gpd \arrow[dashed]{r} & \C^\simeq \arrow[hook]{u}
\end{tikzcd} \]
in $\Cati$ over the canonical projection $\I \ra \I^\gpd$ to the groupoid completion.  In other words, restriction induces an equivalence
\[ \Fun(\I^\gpd,\C^\simeq) \xra{\sim} \Fun(\I,\C^\simeq) \hookra \Fun(\I,\C) \]
onto the (non-full) subcategory of $\Fun(\I,\C)$ on such functors (and natural equivalences between them).  In particular, if $\I^\gpd \simeq \pt_\S$, then this subcategory can be canonically identified with the subcategory of \textit{constant} functors (and the natural equivalences between them), which of course is canonically equivalent to $\C^\simeq$ itself.

\item

An $\infty$-category $\I$ is called \textit{sifted} if it is nonempty and its diagonal map $\diag : \I \ra \I \times \I$ is final (see Definition T.5.5.8.1, \cref{gr:define final}, and \cref{gr:rem final is cofinal}).  The most important single example of a sifted $\infty$-category is $\bD^{op}$ (see Lemma T.5.5.8.4), but note too that all filtered $\infty$-categories are also sifted (see Example T.5.5.8.3).

The following facts regarding sifted $\infty$-categories will be important to us.
\begin{itemize}
\item If $\I$ is a sifted $\infty$-category, then $\I^\gpd \simeq \pt_\S$ (see Proposition T.5.5.8.7).
\item If
\begin{itemizesmall}
\item $\I$ is a sifted $\infty$-category,
\item $\C$ is an $\infty$-category admitting finite products and $\I$-indexed colimits, and
\item the product bifunctor $\C \times \C \xra{- \times -} \C$ preserves $\I$-indexed colimits separately in each variable,
\end{itemizesmall}
then $\colim : \Fun(\I,\C) \ra \C$ preserves finite products (see Lemma T.5.5.8.11).  In particular, this holds when $\C = \S$, or more generally when $\C$ is an $\infty$-topos (see Remark T.5.5.8.12).
\end{itemize}

\item\label{free vs left vs right localizations}

Given a relative $\infty$-category $(\R,\bW) \in \RelCati$, its localization $\loc{\R}{\bW} \in \Cati$ might be more carefully termed its \textit{free} localization.  This construction is left adjoint to the functor $\Cati \ra \RelCati$ taking an $\infty$-category $\C$ to its corresponding \textit{minimal} relative $\infty$-category $(\C,\C^\simeq)$, and can hence be constructed explicitly as the pushout
\[ \loc{\R}{\bW} \simeq \colim \left( \begin{tikzcd}
\bW \arrow{r} \arrow{d} & \bW^\gpd \\
\R
\end{tikzcd} \right) . \]
(These notions are all discussed in detail in \cref{rnerves:section rel-infty-cats and localizns}.)

We warn the reader that this notion does \textit{not} generally agree with the definition of ``localization'' studied in \sec T.5.2.7 (see Warning T.5.2.7.3), namely a functor admitting a fully faithful right adjoint.  When we discuss it, we will refer to this latter notion as a \textit{left} localization; its right adjoint may then be referred to as the inclusion of a \textit{reflective} subcategory.  We will denote general such adjunctions by $\leftloc \adj \forget$ (with additional decorations in specific instances).  These are actually a special case of free localizations (see Proposition T.5.2.7.12 or \cref{rnerves:rem re-prove that a left localization is a free localization}).\footnote{Somewhat confusingly, accessible left localizations of presentable $\infty$-categories additionally satisfy a universal property among \textit{left adjoint} functors (see Proposition T.5.5.4.20).}

Of course, there is the dual notion of a \textit{right} localization (into a \textit{coreflective} subcategory), although due to the overall handedness of mathematics (boiling down to the fact that we're generally more comfortable thinking about $\Set$ than about $\Set^{op}$), this arises less frequently in practice and in particular does not appear anywhere in \cite{LurieHTT} (hence the unambiguity of the terminology ``localization'' used there).  We will similarly denote general such adjunctions by $\forget \adj \rightloc$.

Note that a (free) localization which is neither a left nor a right localization can nevertheless admit a section; see for instance \cref{try to put model structure on bounded spectra}.

\item\label{twisted arrow gives hom bifunctor}

Many of our arguments will implicitly rely on the existence of a hom bifunctor
\[ \C^{op} \times \C \xra{\hom_\C(-,-)} \S \]
for an arbitrary $\infty$-category $\C$.  This is achieved by the \textit{twisted arrow $\infty$-category} construction (see Proposition A.5.2.1.11).  If $\C$ is an enriched $\infty$-category, we will write $\ul{\hom}_\C(-,-)$ for the enriched hom-object and continue to write $\hom_\C(-,-)$ for its underlying hom-spaces.  (For the most part, the enriched categories we will encounter will be of the particularly special sort described in item \ref{enriched and bitensored via two-variable adjunction}.)

\item\label{mono and epi}

A morphism $c \ra d$ in an $\infty$-category $\C$ is called a \textit{monomorphism} if for any other object $e \in \C$, the induced map $\hom_\C(e,c) \ra \hom_\C(e,d)$ is a monomorphism in $\S$: these are precisely the morphisms for which it is merely a condition (as opposed to requiring additional data) for there to exist a factorization
\[ \begin{tikzcd}
e \arrow{rd} \arrow[dashed]{d} \\
c \arrow{r} & d
\end{tikzcd} \]
of a given map $e \ra d$ in $\C$.  This is equivalent to the requirement that the commutative square
\[ \begin{tikzcd}
c \arrow{r} \arrow{d} & c \arrow{d} \\
c \arrow{r} & d
\end{tikzcd} \]
in $\C$ is a pullback.

Dually, a morphism $c \ra d$ in an $\infty$-category $\C$ is called an \textit{epimorphism} if for any other object $e \in \C$, the induced map $\hom_\C(d,e) \ra \hom_\C(c,e)$ is a monomorphism in $\S$: similarly, these are precisely the morphisms for which it is merely a condition for there to exist an extension
\[ \begin{tikzcd}
c \arrow{r} \arrow{rd} & d \arrow[dashed]{d} \\
& e
\end{tikzcd} \]
of a given map $c \ra e$ in $\C$.  This is equivalent to the requirement that the commutative square
\[ \begin{tikzcd}
c \arrow{r} \arrow{d} & d \arrow{d} \\
d \arrow{r} & d
\end{tikzcd} \]
in $\C$ is a pushout.

\item\label{yoneda}

For any $\infty$-category $\C$, we will write
\[ \Yo_\C = \hom_\C(=,-) : \C \ra \Fun(\C^{op},\S) \]
for the \textit{Yoneda functor}, namely the indicated adjunct to the hom bifunctor of item \ref{twisted arrow gives hom bifunctor}; we may simply write $\Yo$ if the $\infty$-category $\C$ is clear from context.\footnote{Pronounced ``yo'', the character $\Yo$ is the first letter of ``Yoneda'' in Hiragana.}  (If $\C$ is a 1-category, we may also write $\Yo_\C : \C \ra \Fun(\C^{op},\Set)$ for its factorization through the subcategory of \textit{discrete} objects (see item \ref{discrete object}).)  We generally write
\[ \P(\C) = \Fun(\C^{op},\S) \]
for the $\infty$-category of \textit{presheaves} (of spaces) on $\C$, the target of the Yoneda functor.  An $\infty$-categorical version of \textit{Yoneda's lemma} (see Proposition T.5.1.3.1) asserts that, just as in ordinary category theory, this functor is fully faithful; we therefore will also refer to it as the \textit{Yoneda embedding}.  We will also use the fact (proved as Theorem T.5.1.5.6) that the Yoneda embedding models the \textit{free cocompletion}, i.e.\! that for any cocomplete $\infty$-category $\D$, restriction along the Yoneda embedding defines an equivalence
\[ \Fun^{\colim}(\P(\C),\D) \xra[\sim]{(\Yo_\C)^*} \Fun(\C,\D) \]
(where we write $\Fun^{\colim}$ to denote the full subcategory of the functor $\infty$-category on those functors which preserve colimits), with inverse given by left Kan extension along $\Yo_\C$.

\item\label{discrete object}

An object $c \in \C$ of an $\infty$-category $\C$ is called \textit{discrete} if the functor
\[ \Yo_\C(c) = \hom_\C(-,c) : \C^{op} \ra \S \]
factors through the subcategory $\Set \subset \S$.\footnote{This condition is equivalent to requiring that the diagonal map $\Yo_\C(c) \ra \Yo_\C(c) \times \Yo_\C(c)$ in $\P(\C)$ be a monomorphism.  If the product $c \times c$ exists in $\C$, then we have an equivalence $\Yo_\C(c \times c) \simeq \Yo_\C(c) \times \Yo_\C(c)$ in $\P(\C)$ and it follows easily that this condition can also be checked in $\C$.}  For instance, $\Set \subset \S$ is the inclusion of the full subcategory of discrete objects.  It is not hard to see that in a presheaf $\infty$-category $\P(\D) = \Fun(\D^{op},\S)$, an object $F \in \P(\D)$ is discrete if and only if it itself factors through $\Set$: that is, discreteness is determined objectwise.  Thus, for instance, $s\Set \subset s\S$ is likewise the inclusion of the full subcategory of discrete objects.

\item\label{adjunction}

Given two $\infty$-categories $\C,\D \in \Cati$, an adjunction $F : \C \adjarr \D : G$ is uniquely determined by a bifunctor
\[ \C^{op} \times \D \xra{A} \S , \]
where $A \simeq \hom_\C(-,G(-)) \simeq \hom_\D(F(-),-)$.\footnote{That this agrees with Definition T.5.2.2.1 follows easily from the formalism of \textit{correspondences} (see \sec T.2.3.1 and \sec T.5.2.1 (and the model-independent theory of co/cartesian fibrations laid out in \cite{grjl})).}  In fact, we can \textit{define} an adjunction to be an arbitrary bifunctor $\C^{op} \times \D \xra{A} \S$ which is ``co/representable in each slot''.  More precisely, this means that for any $c \in \C$ the functor $\D \xra{A(c,-)} \S$ must be corepresentable, while for any $d \in \D$ the functor $\C^{op} \xra{A(-,d)} \S$ must be representable.  Since the Yoneda embedding is fully faithful, we recover the adjoint functors via the unique factorizations
\[ \begin{tikzcd}[column sep=1.5cm]
\C^{op} \arrow{r}{A(-,=)} \arrow[dashed]{rd}[swap]{F^{op}} & \P(\D^{op}) \\
& \D^{op} \arrow[hook]{u}[swap]{\Yo_{\D^{op}}}
\end{tikzcd} \]
and
\[ \begin{tikzcd}[column sep=1.5cm]
\D \arrow{r}{A(=,-)} \arrow[dashed]{rd}[swap]{G} & \P(\C) \\
& \C . \arrow[hook]{u}[swap]{\Yo_\C}
\end{tikzcd} \]
Following standard conventions, in our diagrams that involve adjunctions, we keep left adjoints above and/or to the left of their right adjoints to whatever extent possible.  (In-line adjunctions will \textit{always} have their left adjoints on top.)  For added clarity, we often use the ``turnstile'' symbol $\perp$, which sits on the right adjoint and points towards the left adjoint.  Even in the absence of an ambient diagram, we write $F \adj G$ to indicate that $F$ is left adjoint to $G$.

We define the $\infty$-category of \textit{adjunctions from $\C$ to $\D$} to be the full subcategory
\[ \Adjn(\C;\D) \subset \Fun(\C^{op} \times \D , \S) \]
on those bifunctors that define adjunctions.  If the objects $A , A' \in \Adjn(\C;\D)$ determine adjunctions $F \adj G$ and $F' \adj G'$, then a map $A \ra A'$ in $\Adjn(\C;\D)$ is uniquely determined by either datum of a morphism $F' \ra F$ in $\Fun(\C,\D)$ or a morphism $G \ra G'$ in $\Fun(\D,\C)$.\footnote{One might say that these two natural transformations are \textit{conjugates} with respect to the given adjunctions (as in \cite[Chapter IV, \sec 7]{MacLaneCWM}, except that our variance is reversed).}  We will write $\LAdjt(\C,\D) \subset \Fun(\C,\D)$ for the full subcategory on those functors which are left adjoints, and we will write $\RAdjt(\D,\C) \subset \Fun(\D,\C)$ for the full subcategory on those functors which are right adjoints.  We therefore have equivalences
\[ \LAdjt(\C,\D)^{op} \xla{\sim} \Adjn(\C;\D) \xra{\sim} \RAdjt(\D,\C) \]
by the uniqueness of adjoints.

We also note here that given an adjunction $\C \adjarr \D$ and any $\infty$-category $\E$, applying the functor $\Fun(\E,-) : \Cati \ra \Cati$ yields a canonical adjunction $\Fun(\E,\C) \adjarr \Fun(\E,\D)$.  (This follows easily by combining Proposition T.5.2.2.8 with \cite[Proposition 2.3]{SaulHodge} (or with \cite[Proposition 5.1]{GHN}).)

\item\label{multivar adjns}

More generally, we define an \textit{adjunction of $i$ contravariant variables and $j$ covariant variables} to be a multifunctor
\[ (\C_1)^{op} \times \cdots \times (\C_i)^{op} \times \D_1 \times \cdots \times \D_j \xra{A} \S \]
satisfying the condition that fixing all but any one of the slots yields a co/representable functor.\footnote{The 1-categorical version of this notion (in the case $i=1$ and $j=n$, called there an ``adjunction of $n$ variables'') is defined in \cite[Definition 2.1]{CGR-multivar-adjns} (see \cite[Theorem 2.2]{CGR-multivar-adjns}).}  Note that by definition, fixing any number of slots in such an adjunction yields an adjunction in the remaining free variables.\footnote{Thus, by convention, an adjunction in a single variable in just a functor to $\S$, and an adjunction in zero variables is just an object of $\S$.}  We similarly define a full subcategory
\[ \Adjn(\C_1,\ldots,\C_i;\D_1,\ldots,\D_j) \subset \Fun((\C_1)^{op} \times \cdots (\C_i)^{op} \times \D_1 \times \cdots \D_j , \S) \]
on such multivariable adjunctions.

Aside from ordinary adjunctions, we will mainly be interested in (what have come to be called) \textit{adjunctions of two variables} (or simply \textit{two-variable adjunctions}), namely the case $i=2$ and $j=1$.  A two-variable adjunction is thus a trifunctor
\[ \C^{op} \times \D^{op} \times \E \xra{A} \S , \]
in which the co/representability condition furnishes three bifunctors denoted in general as
\[ \twovaradjCDE , \]
which come equipped with uniquely determined natural equivalences
\[ A(c,d,e) \simeq \hom_\C(c,\enrhom_r(d,e)) \simeq \hom_\D(d,\enrhom_l(c,e)) \simeq \hom_\E(c \otimes d,e) \]
in $\Fun(\C^{op} \times \D^{op} \times \E , \S)$ (for $c \in \C$, $d \in \D$, and $e \in \E$).  Just as with ordinary adjunctions, we will often denote a two-variable adjunction simply by listing its constituent bifunctors, leaving the natural equivalences implicit.

\item\label{co/tensoring}

Given an $\infty$-category $\C$, an object $c \in \C$, and a space $Y \in \S$, a \textit{tensor} of $c$ over $Y$ is an object $c \tensoring Y \in \C$ equipped with an equivalence
\[ \hom_\C(c \tensoring Y , -) \simeq \hom_\S(Y,\hom_\C(c,-)) \]
in $\Fun(\C,\S)$.  (We may sometimes write this as $Y \tensoring c \in \C$ for notational convenience.)  Dually, a \textit{cotensor} of $Y$ with $c$ is an object $Y \cotensoring c \in \C$ equipped with an equivalence
\[ \hom_\C(-,Y \cotensoring c) \simeq \hom_\S(Y,\hom_\C(-,c)) \]
in $\Fun(\C^{op},\S)$.  (As this is a sort of generalized mapping object, we will never reverse the order of the objects.)  We will say that $\C$ is \textit{tensored} (over $\S$) if it admits all tensors, and that it is \textit{cotensored} (over $\S$) if it admits all cotensors.  We will say that $\C$ is \textit{bitensored} if it is both tensored and cotensored.

If we denote by $(\S \times \C)^\tensoring \subset \S \times \C$ the full subcategory on those pairs admitting a tensoring, then by definition we can construct the tensoring as a bifunctor via the factorization
\[ \begin{tikzcd}[column sep=3.5cm]
(\S \times \C)^\tensoring \arrow{r}{\hom_\S(-,\hom_\C(-,=))} \arrow[dashed]{rd}[swap]{- \tensoring -} & \P(\C^{op}) \\
& \C \arrow[hook]{u}[swap]{\Yo}
\end{tikzcd} \]
through the fully faithful Yoneda embedding, and we can similarly construct the (maximal) cotensoring as a bifunctor
\[ (\S^{op} \times \C)^\cotensoring \xra{- \cotensoring -} \C . \]
Using the same argument, if $\C$ is tensored (resp.\! cotensored), it is not hard to extend the tensoring (resp.\! cotensoring) bifunctor to an action of the symmetric monoidal $\infty$-category $(\S,\times) \in \CAlg(\Cati)$ (resp.\! $(\S^{op},\times) \in \CAlg(\Cati)$) on the $\infty$-category $\C \in \Cati$.\footnote{If $\C$ is additionally presentable (and hence in particular cocomplete), we can alternatively recover the tensoring action from the symmetric monoidal structure on the $\infty$-category of presentable $\infty$-categories, for which $\S$ is the unit object (see Proposition A.4.8.1.14 and Example A.4.8.1.19).}  If $\C$ is bitensored, then we obtain a two-variable adjunction
\[ \twovaradjSC . \]

By Corollary T.4.4.4.9, considering $Y \in \S \subset \Cati$, we have an equivalence
\[ c \tensoring Y \simeq \colim^\C_Y \const(c) \]
in $\C$ (assuming either side exists).  Thus, a tensoring is a sort of colimit, and hence a cocomplete $\infty$-category is in particular tensored.  Dually, a cotensoring is a sort of limit, and hence a complete $\infty$-category is in particular cotensored.  Similarly, an $\infty$-category which is finitely co/complete is in particular co{\textbackslash}tensored over $\S^\fin \subset \S$.

\item\label{enriched and bitensored via two-variable adjunction}

More generally, suppose that $(\V,\otimes) \in \Alg(\Cati)$ is a closed monoidal $\infty$-category, and suppose that $\C \in \RMod_\V(\Cati)$ is a right $\V$-module.\footnote{For us, a monoidal $\infty$-category being \textit{closed} by definition means that it is both left closed and right closed.  Note that this is actually just a property, not additional structure: left/right closure only demands the existence of certain adjoints.}  Writing
\[ \C \times \V \xra{- \tensoring -} \C \]
for the underlying bifunctor of the right action of $\V$ on $\C$, let us suppose further that this extends to a two-variable adjunction.  Such an extension is precisely the data of an \textit{enrichment} and \textit{bitensoring} of $\C$ over $\V$: the action defines the tensoring, and we write
\[ \C^{op} \times \C \xra{\enrhom_\C(-,-)} \V \]
and
\[ \V^{op} \times \C \xra{- \cotensoring -} \C \]
for the other two constituent bifunctors; as the notation indicates, these come with natural equivalences
\[ \hom_\C(c \tensoring v , d) \simeq \hom_\V(v,\enrhom_\C(c,d)) \simeq \hom_\C(c,v \cotensoring d) \]
in $\Fun(\C^{op} \times \V^{op} \times \C , \S)$ (for $c,d \in \C$ and $v \in \V$).\footnote{If we are simply given a two-variable adjunction of this signature \textit{without} an extension of the first bifunctor to an action of $\V$ on $\C$, then there will not be any compatibility between the symmetric monoidal structure on $\V$ and the bifunctors comprising the two-variable adjunction.}

To see that this gives an enrichment of $\C$, observe first that we have equivalences
\[ \hom_\V(\unit_\V,\enrhom_\C(c,d)) \simeq \hom_\C(c \tensoring \unit_\V , d) \simeq \hom_\C(c,d) \]
by the unitality of the action of $\V$ on $\C$.  The enriched composition maps are obtained from the evaluation maps
\[ \hom_\C( c \tensoring \enrhom_\C(c,d) , d) \simeq \hom_\V(\enrhom_\C(c,d) , \enrhom_\C(c,d)) \xla{\id_{\enrhom_\C(c,d)}} \pt_\S \]
as the composites
\begin{align*}
\hom_\V ( \enrhom_\C(c_0,c_1) \otimes \enrhom_\C(c_1,c_2) , \enrhom_\C(c_0,c_2) )
& \simeq \hom_\C(c_0 \tensoring ( \enrhom_\C(c_0,c_1) \otimes \enrhom_\C(c_1,c_2)) , c_2) \\
& \simeq \hom_\C((c_0 \tensoring \enrhom_\C(c_0,c_1)) \tensoring \enrhom_\C(c_1,c_2) , c_2) \\
& \la \hom_\C(c_1 \tensoring \enrhom_\C(c_1,c_2) , c_2) \\
& \la \hom_\C(c_2,c_2) \\
& \xla{\id_{c_2}} \pt_\S ,
\end{align*}
and the higher composition maps are obtained by essentially this same construction.  It is not hard to see that applying the functor $\hom_\V(\unit_\V,-) : \V \ra \S$, which is canonically lax monoidal, recovers the original composition maps in $\C$.

Then, to see that these functors define \textit{enriched} co/tensors, we check that for an arbitrary test object $w \in \V$,
\begin{align*}
\hom_\V(w,\enrhom_\C(c \tensoring v ,d))
& \simeq \hom_\C((c \tensoring v)\tensoring w,d) \\
& \simeq \hom_\C( c \tensoring (v \otimes w),d)  \\
& \simeq \hom_\V(v\otimes w,\enrhom_\C(c,d)) \\
& \simeq \hom_\V(w,\enrhom_\V(v,\enrhom_\C(c,d)))
\end{align*}
and similarly
\[ \hom_\V(w,\enrhom_\C(c,v \cotensoring d)) \simeq \hom_\V(w,\enrhom_\V(v,\enrhom_\C(c,d))) ; \]
by Yoneda's lemma, we obtain the desired natural equivalences
\[ \enrhom_\C(c \tensoring v , d) \simeq \enrhom_\V(v,\enrhom_\C(c,d)) \simeq \enrhom_\C(c,v\cotensoring d) \]
of enriched hom-objects in $\V$.\footnote{This is the only place we have used that $\V$ is closed; without this assumption, the given data still define a $\V$-enrichment of $\C$ along with \textit{unenriched} co/tensors over $\V$ (in the evident sense).}

Finally, we observe that the cotensoring bifunctor can be canonically extended to a \textit{left} action of $(\V^{op},\otimes^{op}) \in \Alg(\Cati)$ on $\C \in \Cati$ (i.e.\! we can consider $\C \in \LMod_{\V^{op}}(\Cati)$), simply by passing the tensoring action through the adjunction; for instance, for any $c,d \in \C$ and any $v ,w \in \V$ we have a natural string of equivalences
\begin{align*}
\hom_\C(c,w \cotensoring (v \cotensoring d))
& \simeq \hom_\C(c \tensoring w, v \cotensoring d)
\simeq \hom_\C((c \tensoring w) \tensoring v,d)
\\
& \simeq \hom_\C(c \tensoring (w \otimes v),d)
\simeq \hom_\C(c,(w \otimes v) \cotensoring d),
\end{align*}
which by Yoneda's lemma provides a canonical natural equivalence
\[ w \cotensoring (v \cotensoring d) \simeq (w \otimes v) \cotensoring d \]
in $\Fun(\V^{op} \times \V^{op} \times \C,\C)$.

Of course, this same discussion goes through without essential change in the special case that $\V$ is in fact \textit{symmetric} monoidal.

\item\label{generalized matching and latching}

Given a finitely complete $\infty$-category $\C$, its corresponding \textit{generalized matching object} bifunctor
\[ (s\S^{\fin})^{op} \times s\C \xra{\Match_{(-)}(-)} \C \]
is given by
\[ \Match_K(Y) = \int^{[n] \in \bD^{op}} K_n \cotensoring Y_n .  \]
By construction, this comes equipped with an equivalence
\[ \hom_\C( - , \Match_K(Y)) \simeq \hom_{s\S}(K , \hom_\C^\lw(-,Y)) \]
in $\Fun(\C^{op},\S)$, so that in particular when $\C$ is in fact bicomplete we obtain a two-variable adjunction
\[ \twovaradjmatching . \]
Dually, given a finitely cocomplete $\infty$-category $\C$, its corresponding \textit{generalized latching object} bifunctor
\[ s\S^\fin \times c\C \xra{\Latch_{(-)}(-)} \C \]
is given by
\[ \Latch_K(Z) = \int_{[n] \in \bD^{op}} Z^n \tensoring K_n . \]
By construction, this comes equipped with an equivalence
\[ \hom_\C(\Latch_K(Z) , - ) \simeq \hom_{s\S}(K , \hom^\lw_\C(Z,-)) \]
in $\Fun(\C,\S)$, so that in particular when $\C$ is in fact bicomplete we obtain a two-variable adjunction
\[ \twovaradjlatching . \]
These notions are extensions of the usual theory of matching and latching objects in Reedy categories, and correspondingly we make the abbreviations $\Match_n(-) = \Match_{\partial \Delta^n}(-)$ and $\Latch_n(-) = \Latch_{\partial \Delta^n}(-)$.

\end{enumerate}

\subsection{}\label{subsection model cats present infty-cats}

Note that we are considering model categories as objects of study in their own right: they are nothing more than model $\infty$-categories whose hom-spaces are discrete.  However, we will also be using model categories as \textit{presentations} of their underlying $\infty$-categories (as indicated in subitem \ref{different model cats}\ref{relcat as infty-cat}).  Thus, we must also establish our conventions regarding their manipulation in this capacity.

For historical context, we will make some attempt to reference the primary sources for results concerning model categories.  However, the body of literature is vast, and so as catch-all resources we will generally refer to \cite{Hirsch} and \cite{GJ}, especially for the more classical results.\footnote{The former requires that its model categories have functorial factorizations, whereas we do not.  We will never use general results on model categories that depend on functorial factorizations.}

\begin{enumerate}[resume]

\item

As a consistency check, we observe that our consideration of objects of $\strrelcat_\BarKan$ as presentations of $\infty$-categories does indeed identify a relative category $(\R,\bW) \in \strrelcat$ with its $\infty$-categorical localization $\loc{\R}{\bW} \in \Cati$ (as defined in item \ref{free vs left vs right localizations}).  More precisely, the natural commutative diagram
\[ \begin{tikzcd}
\bW \arrow{r} \arrow{d} & \hamd(\bW,\bW) \arrow{d} \\
\R \arrow{r} & \hamd(\R,\bW)
\end{tikzcd} \]
is a homotopy pushout square in $(\strcat_{s\Set})_\Bergner$ (and hence presents a pushout in $\Cati$ (see item \ref{ho-co-lims})), and moreover the object $\hamd(\bW,\bW) \in (\strcat_{s\Set})_\Bergner$ presents the groupoid completion $\bW^\gpd \in \Cati$.\footnote{The proof of this assertion is mostly contained in \cite[3.4]{BK-simploc}.  However, note that there, they do not work in $(\strcat_{s\Set})_\Bergner$, but rather work in the Dwyer--Kan model structure on ``simplicial $O$-categories'' (see \cite[Proposition 7.2]{DKSimpLoc}, though note that their citation for this model structure should actually be to \cite[Chapter II, \sec 4, Theorem 4]{QuillenHA}).  However, the characterization \cite[Proposition 7.6]{DKSimpLoc} of the cofibrations implies that the forgetful functor to $(\strcat_{s\Set})_\Bergner$ preserves cofibrations, so that it also preserves homotopy pushouts (since it also preserves ordinary pushouts).  Moreover, the assertion that $\hamd(\bW,\bW) \in (\strcat_{s\Set})_\Bergner$ presents $\bW^\gpd \in \Cati$ follows from \cite[5.5]{DKSimpLoc} and \cite[Proposition 2.2]{DKCalc}.}\footnote{As described in item \ref{ho-co-lims}, it is only known that homotopy co/limits in \textit{combinatorial simplicial} model categories present co/limits in their underlying $\infty$-categories.  However, even though $(\strcat_{s\Set})_\Bergner$ is not a combinatorial simplicial model category, it is easy enough to show that homotopy pushouts therein coincide up to a zigzag of natural weak equivalences with those computed in the combinatorial simplicial model category $ss\Set_\Rezk$ (using the Quillen equivalence $ss\Set_\Rezk \adjarr (\strcat_{s\Set})_\Bergner$ and the functorial factorizations guaranteed by cofibrant generation).}  More succinctly, we can (apparently circularly but now in fact soundly) summarize this assertion by saying that the localization functor
\[ \strrelcat \ra \loc{\strrelcat}{\bW_\BarKan} \simeq \Cati \]
is \textit{itself} given by localization.

\item\label{choose representative in model category}

In keeping with our general desire for our language to remain independent of any noncanonical choices, when we choose a \textit{representative} in a model category of an object or a map in its underlying $\infty$-category, we will only mean a representative up to equivalence in the underlying $\infty$-category.  When doing so, we indicate this noncanonical choice using ``typewriter text'', so that for instance, given an $\infty$-category $\C \in \Cati$, we might write $\ttC \in s\Set_\Joyal^f$ to denote a quasicategory representing it.

\item\label{sometimes need simplicial}

Given a \textit{simplicial} model category $\M_\bullet$ (with underlying model category $\M$), another notion of ``underlying $\infty$-category'' is given by the full simplicial subcategory $\M^{cf}_\bullet \subset \M_\bullet$ on the bifibrant objects.  By \cite[Proposition 4.8]{DKFunc}, this is weakly equivalent to $\hamd(\M,\bW)$ in $(\strcat_{s\Set})_\Bergner$.\footnote{In the diagram in the statement of \cite[Proposition 4.8]{DKFunc}, the right arrow should also be labeled as a weak equivalence (in $(\strcat_{s\Set})_\Bergner$), as indicated by its proof.}\footnote{This is also proved directly to present the $\infty$-categorical localization $\loc{\M^c}{(\bW^c)}$ as Theorem A.1.3.4.20 (and there is a canonical equivalence $\loc{\M^c}{(\bW^c)} \xra{\sim} \loc{\M}{\bW}$ e.g.\! by \cite[\cref{adjns:inclusion of co/fibrants induces BK weak equivalence}]{adjns}).}  In making connections between model categories and $\infty$-categories, the results of \cite{LurieHTT} generally assume that the given model categories are simplicial.  As a result, some of the connections that we make will carry this same caveat.



\item\label{joyal to JT}

As we have indicated in \cref{ground in qcats}, the primary model category we will use to present the $\infty$-category $\Cati$ will be $s\Set_\Joyal$.  Unfortunately, this does not enjoy all the nice properties that one might hope; in particular, it is not a simplicial model category.  However, all is not lost: there exist both left and right Quillen equivalences to the combinatorial simplicial model category $ss\Set_\Rezk$ given by \cite[Theorems 4.11 and 4.12]{JT}.  These allow us to port many convenient features of $ss\Set_\Rezk$ over to $s\Set_\Joyal$ (such as in items \ref{ho-co-lims} \and \ref{functor model cats present functor infty-cats} below).

\item\label{quillen adjunction induces underlying adjunction}

Suppose that $F : \M \adjarr \N : G$ is a Quillen adjunction.  Note that the functors $F$ and $G$ do not define functors of underlying relative categories: they do not generally preserve weak equivalences.  Nevertheless, we prove as \cite[\cref{adjns:main theorem}]{adjns} that a Quillen adjunction between model categories induces an associated adjunction of quasicategories.\footnote{In the case of a Quillen adjunction of \textit{simplicial} model categories, this result is proved as Proposition T.5.2.4.6.}  By Kenny Brown's lemma (or rather its immediate consequence \cite[Corollary 7.7.2]{Hirsch}), the composites
\[ \M^c \hookra \M \xra{F} \N \]
and
\[ \M \xla{G} \N \hookla \N^f \]
do preserve weak equivalences, and these respectively present the left and right adjoint functors.  As a particular case, we immediately obtain that left Bousfield localizations present left localizations (and dually).

\item

If $\M$ is a model category and $x \xra{f} y$ is any map in $\M$, it is easy to check that the induced adjunction $\M_{x/} \adjarr \M_{y/}$ is automatically a Quillen adjunction.  If $f$ is additionally a weak equivalence, we might hope that this is then a Quillen equivalence.  For this to hold, however, we need for every pushout of $f$ along a cofibration to be a weak equivalence.  This will be true either
\begin{itemizesmall}
\item if $f$ is an \textit{acyclic} cofibration, or
\item if $\M$ is left proper.
\end{itemizesmall}

This observation allows us to partially address the question of when the induced model structure on $\M_{x/}$ present the undercategory $\loc{\M}{\bW}_{x/}$ (or dually, when the induced model structure on $\M_{/y}$ presents the overcategory $\loc{\M}{\bW}_{/y}$): we only establish the connection for \textit{simplicial} model categories, though this will suffice for our purposes.  Namely, let $\M_\bullet$ be a simplicial model category.
\begin{itemize}
\item Suppose that $x \in \M^c$.  If we choose any factorization $x \wcofibn x' \fibn \pt_\M$, then we obtain a Quillen equivalence $(\M_{x/})_\bullet \adjarr (\M_{x'/})_\bullet$ with $x' \in \M^{cf}$.  Since Quillen equivalences induce equivalences of underlying $\infty$-categories, the dual result to Lemma T.6.1.3.13 implies that $(\M_{x/})_\bullet$ (and hence also the underlying model category $\M_{x/}$) presents the undercategory of the object of the underlying $\infty$-category of $\M_\bullet$ corresponding to $x$.
\item On the other hand, if $\M_\bullet$ is left proper, then this statement holds for any $x \in \M$.  Indeed, if we choose any factorization $\es_\M \cofibn x'' \we x$, then we obtain a Quillen equivalence $(\M_{x''/})_\bullet \adjarr (\M_{x/})_\bullet$, which reduces us to the previous case.
\end{itemize}

\item\label{ho-co-lims}

We will use the term \textit{homotopy co/limit} in a model category $\M$ to refer to a (not necessarily commutative) diagram which becomes a (commutative) co/limit diagram in $\loc{\M}{\bW}$.

If $\M_\bullet$ is a combinatorial simplicial model category, then it follows from Remark T.A.3.3.11, Proposition T.A.3.3.12, Remark T.A.3.3.13, and Theorem T.4.2.4.1 that homotopy co/limits in $\M_\bullet$ (in the classical sense) compute co/limits in its underlying $\infty$-category.  (See those results for a precise statement.)

Homotopy co/limits are generally computed using \textit{model structures on functor categories}, of which there are three main examples.\footnote{For details on these, see respectively: \sec T.A.2.8; \cite[\sec 11.6]{Hirsch} and \sec T.A.2.8; \cite[Chapter 15]{Hirsch} and \sec T.A.2.9.}
\begin{itemize}
\item An \textit{injective model structure} on $\Fun(\C,\M)$, denoted $\Fun(\C,\M)_\injective$, has its weak equivalences and cofibrations determined objectwise.  This is guaranteed to exist when $\M$ is combinatorial.
\item A \textit{projective model structure} on $\Fun(\C,\M)$, denoted $\Fun(\C,\M)_\projective$, has its weak equivalences and fibrations determined objectwise.  This is guaranteed to exist when $\M$ is cofibrantly generated.
\item Given a category $\C$ endowed with a Reedy structure, the corresponding \textit{Reedy model structure} on $\Fun(\C,\M)$, denoted $\Fun(\C,\M)_\Reedy$, has its weak equivalences determined objectwise (but its cofibrations and fibrations depend on the Reedy structure), and exists without any additional assumptions on $\M$.\footnote{If the Reedy structure on $\C$ has $\lvec{\C} = \C$, then the Reedy and injective model structures on $\Fun(\C,\M)$ coincide (and both always exist).  Dually, if the Reedy structure on $\C$ has $\rvec{\C} = \C$, then the Reedy and projective model structures on $\Fun(\C,\M)$ coincide (and both always exist).  Thus, in general, the Reedy model structure can be seen as a ``mixture'' of the injective and projective model structures (see Example T.A.2.9.22).}
\end{itemize}
These enjoy the following properties.

\begin{itemize}

\item Whenever these various model structures exist, the identity adjunction gives rise to Quillen equivalences
\[ \begin{tikzcd}[row sep=2cm]
\Fun(\C,\M)_\projective \horizadjntwocolumns \diagadjn & & \Fun(\C,\M)_\injective \\
& \Fun(\C,\M)_\Reedy \antidiagadjn
\end{tikzcd} \]
between them.

\item Applying $\Fun(\C,-)$ to a Quillen adjunction (resp.\! Quillen equivalence) $\M \adjarr \N$ gives rise to another Quillen adjunction (resp.\! Quillen equivalence) with respect to any of these model structures that exist on both $\Fun(\C,\M)$ and $\Fun(\C,\N)$.

\item These various model structures participate in Quillen adjunctions as follows.

\begin{itemize}
\item We have a Quillen adjunction
\[ \const : \M \adjarr \Fun(\C,\M)_\injective : \lim \]
whenever the injective model structure and the limit functor both exist.
\item We have a Quillen adjunction
\[ \colim : \Fun(\C,\M)_\projective \adjarr \M : \const \]
whenever the projective model structure and the colimit functor both exist.
\item If $\C$ is endowed with a Reedy model structure with \textit{cofibrant constants} then we are guaranteed a Quillen adjunction
\[ \const : \M \adjarr \Fun(\C,\M)_\Reedy : \lim , \]
while if $\C$ is endowed with a Reedy model structure with \textit{fibrant constants} then we are guaranteed a Quillen adjunction
\[ \colim : \Fun(\C,\M)_\Reedy \adjarr \M : \const . \]
However, these adjunctions (if they exist) can still be Quillen adjunctions even without these restrictions on $\C$, albeit (necessarily by definition) only for specific choices of $\M$.
\end{itemize}

\end{itemize}
At least when $\M_\bullet$ is a combinatorial simplicial model category, any of these model structures on $\Fun(\C,\M)$ presents the $\infty$-category $\Fun(\C,\loc{\M}{\bW})$ by Proposition T.4.2.4.4 and Remark T.4.2.4.5.\footnote{Moreover, the results of \cite{DuggerSimp} and \cite{RSS} can sometimes be used to replace a model category (via a Quillen equivalence) with a combinatorial simplicial one.}  As the functor $\const : \M \ra \Fun(\C,\M)$ in $\strrelcat_\BarKan$ clearly presents the functor $\const : \loc{\M}{\bW} \ra \Fun(\C,\loc{\M}{\bW})$, combining item \ref{quillen adjunction induces underlying adjunction} with the uniqueness of adjoints shows that the derived functors of these various Quillen adjunctions do indeed compute homotopy co/limits.  (In particular, as foreshadowed in item \ref{joyal to JT}, from here it is straightforward to see that homotopy co/limits in the combinatorial model category $s\Set_\Joyal$ do indeed compute co/limits in $\Cati$.)




\item

As a particular case of item \ref{ho-co-lims}, there is a Reedy structure on the walking span category
\[ \Nerve^{-1}(\Lambda^2_0) = (\bullet \la \bullet \ra \bullet) \]
determined by the degree function described by the picture $(0 \la 1 \ra 2)$.  This has fibrant constants (see e.g.\! the proof of \cite[Proposition 15.10.10]{Hirsch}), so that for any model category $\M$ we obtain a Quillen adjunction
\[ \colim : \Fun(\Nerve^{-1}(\Lambda^2_0),\M)_\Reedy \adjarr \M : \const . \]
Moreover, the cofibrant objects of $\Fun(\Nerve^{-1}(\Lambda^2_0),\M)_\Reedy$ are precisely the diagrams of the form
\[ x \la y \cofibn z \]
for $x,y,z \in \M^c \subset \M$.

Dually, there is a Reedy structure on the walking cospan category
\[ \Nerve^{-1}(\Lambda^2_2) = (\bullet \ra \bullet \la \bullet) \]
determined by the degree function described by the picture $(0 \ra 1 \la 2)$.  This has cofibrant constants, so that for any model category $\M$ we obtain a Quillen adjunction
\[ \const : \M \adjarr \Fun(\Nerve^{-1}(\Lambda^2_2),\M)_\Reedy : \lim . \]
Moreover, the fibrant objects of $\Fun(\Nerve^{-1}(\Lambda^2_2),\M)_\Reedy$ are precisely the diagrams of the form
\[ x \ra y \lfibn z \]
for $x,y,z \in \M^f \subset \M$.

We will refer to either of these dual techniques simply as \textit{the Reedy trick}.

\item\label{functor model cats present functor infty-cats}

We will at times make computations in functor $\infty$-categories using model structures on functor categories; that these present the desired functor $\infty$-categories will always follow from the observations of items \ref{ho-co-lims} \and \ref{joyal to JT}.  We also recall here that the model structures $s(s\Set_\KQ)_\Reedy$ and $s(s\Set_\KQ)_\injective$ coincide by Example T.A.2.9.21.  From this, it follows that the model structures $s(s\Set_\Joyal)_\Reedy$ and $s(s\Set_\Joyal)_\injective$ also coincide: they have the same weak equivalences by definition, and their cofibrations coincide since those of $s\Set_\Joyal$ coincide with those of $s\Set_\KQ$.

\end{enumerate}

\subsection{}

We end \cref{section conventions} by laying out a few other miscellaneous conventions.

\begin{enumerate}[resume]

\item

Whenever we draw a diagram which takes place in a model ($\infty$-)category, we explicitly mention the ambient model structure for emphasis.  However, we will only decorate those aspects of the diagram (e.g.\! a morphism as a co/fibration) which are relevant to the argument.

\item\label{notation opobj}

Given an $\infty$-category $\C$ and an object $c \in \C$, for emphasis we may denote the corresponding object by $c^\opobj \in \C^{op}$: explicitly, if the object $c \in \C$ is selected by a morphism $[0] \xra{\chi} \C$ in $\Cati$, then the object $c^\opobj \in \C^{op}$ is selected by the composite
\[ [0] \xra{\cong} [0]^{op} \xra{\chi^{op}} \C^{op} \]
in $\Cati$.  Similarly, if a morphism $f$ in $\C$ is selected by a morphism $[1] \xra{\varphi} \C$ in $\Cati$, then we may denote by $f^\opobj$ the morphism in $\C^{op}$ selected by the composite
\[ [1] \xra{\cong} [1]^{op} \xra{\varphi^{op}} \C^{op} \]
in $\Cati$, where the isomorphism is determined by the assignments $0 \mapsto 1^\opobj$ and $1 \mapsto 0^\opobj$.  On the other hand, we will sometimes omit these decorations in order not to overburden our notation.\footnote{One might naively hope to simply write e.g.\! $c^{op}$ for the object of $\C^{op}$ corresponding to the object $c \in \C$, but then one would run into trouble as soon as different ``category levels'' begin to mix: for example, the notation $[n]^{op}$ could then either refer to an object of $\Cat$ (which is in fact (equivalent to) the object $[n] \in \bD \subset \Cat$) or to an object of $\bD^{op}$.  Thus, we reserve the superscript $(-)^{op}$ to denote the involution of $\Cati$.  Note, however, that this does not just induce a covariant action on the objects and morphisms of $\Cati$, but also induces a contravariant action on its \textit{2-morphisms}: for any $\C,\D \in \Cati$ we have a canonical identification $\Fun(\C^{op},\D^{op}) \simeq \Fun(\C,\D)^{op}$, so that a pair of functors $F,G : \C \rra \D$ and a natural transformation $\alpha : F \ra G$ in $\Fun(\C,\D)$ corresponds to a pair of functors $F^{op},G^{op} : \C^{op} \rra \D^{op}$ and a natural transformation $\alpha^{op} : G^{op} \ra F^{op}$ in $\Fun(\C^{op},\D^{op})$.}

\item\label{notn lifting properties}

Given a set $I$ of homotopy classes of maps in an $\infty$-category $\C$, we write $\llp(I)$ and $\rlp(I)$ for the sets of (homotopy classes of) maps that have the left or right lifting property with respect to $I$, respectively.  (A lifting property with respect to a subcategory by definition means a lifting property with respect to the homotopy classes of maps contained in that subcategory.)  To be explicit, note that a commutative square in an $\infty$-category is presented by a map from $\Delta^1 \times \Delta^1$ to a quasicategory.  To obtain a lift through that commutative square is then to obtain an extension over the map
\[ \Delta^1 \times \Delta^1 \cong \Delta^{\{013\}} \coprod_{\Delta^{\{03\}}} \Delta^{\{023\}} \cofibn \Delta^3 \]
in $s\Set_\Joyal$.  Alternatively (and invariantly), given a pair of maps $x \xra{i} y$ and $z \xra{p} w$, to say that $i \in \llp(\{p\})$ (or equivalently that $p \in \rlp(\{i\})$) is precisely to say that the induced map
\[ \hom_\C(y,z) \ra
\lim \left(
\begin{tikzcd}
& \hom_\C(y,w) \arrow{d}{i^*} \\
\hom_\C(x,z) \arrow{r}[swap]{p_*} & \hom_\C(x,w)
\end{tikzcd} \right) \]
in $\S$ is a surjection.

\item\label{simplex category}

When working in the cosimplicial indexing category $\bD$, we will often indicate an inclusion simply by specifying its image, so that for instance the notation $[0] \xra{\{i\}} [n]$ refers to the map given by $0 \mapsto i$.  (In particular, we will therefore denote by $\Delta^{\{i_0,\ldots,i_j\}} \subset \Delta^n$ the evident subobject in $s\Set$.)  We will also employ the standard notations
\begin{itemizesmall}
\item $\delta^i_m \in \hom_\bD([m-1],[m])$ for the \textit{coface} maps (for $0 \leq i \leq m$), and
\item $\sigma^j_n \in \hom_\bD([n+1],[n])$ for the \textit{codegeneracy} maps (for $0 \leq j \leq n$),
\end{itemizesmall}
or we may simply write $\delta^i$ or $\sigma^j$ (resp.)\! if the source and/or target are clear from the context.

\item\label{cosimp and simp objs}

Given any $\infty$-category $\C$, we write $c\C = \Fun(\bD,\C)$ for the $\infty$-category of \textit{cosimplicial objects} in $\C$, and we write $s\C = \Fun(\bD^{op},\C)$ for the $\infty$-category of \textit{simplicial objects} in $\C$.  For any objects $Y \in c\C$ and $Z \in s\C$,
\begin{itemizesmall}
\item we denote their constituent objects of $\C$ by $Y^n = Y([n])$ and $Z_n = Z([n]^\opobj)$, and
\item we variously denote their structure maps as follows:
\begin{itemizesmall}
\item a coface map $[m] \xra{\delta^i_m} [m+1]$ in $\bD$ induces
\begin{itemizesmall}
\item a coface map $Y^m \xra{\delta^i_m} Y^{m+1}$ and
\item a face map $Z_{m+1} \xra{\delta_i^m} Z_m$
\end{itemizesmall}
(or simply $\delta^i$ and $\delta_i$, resp.);
\item a codegeneracy map $[n+1] \xra{\sigma^j_n} [n]$ induces
\begin{itemizesmall}
\item a codegeneracy map $Y^{n+1} \xra{\sigma^j_n} Y^n$ and
\item a degeneracy map $Z_n \xra{\sigma_j^n} Z_{n+1}$
(or simply $\sigma^j$ and $\sigma_j$, resp.);
\end{itemizesmall}
\item an arbitrary map $[m] \xra{\varphi} [n]$ (not explicitly identified as a coface or codegeneracy) induces
\begin{itemizesmall}
\item a map $Y^m \xra{\varphi} Y^n$ and
\item a map $Z_n \xra{\varphi} Z_m$
\end{itemizesmall}
(or $Y(\varphi)$ and $Z(\varphi)$ (or even $Z(\varphi^\opobj)$), resp., if we wish to emphasize the functoriality of $Y : \bD \ra \C$ or $Z : \bD^{op} \ra \C$).
\end{itemizesmall}
\end{itemizesmall}

\item\label{notn possibly-omitted decorations}

There are certain decorations which are sometimes useful to include for emphasis or clarity but are at other times useful to exclude for simplicity.  For instance, we may write $(-)^\bullet$ to emphasize that an object is cosimplicial, but we may omit this decoration if we are considering the entire cosimplicial object at once and have no plans to extract its constituents.  We list these here.

\vspace{0pt}

\begin{center}
\bgroup
\def\arraystretch{1.5}
\begin{longtable}{||c|c||}
\hline \hline
\textbf{decoration} & \textbf{meaning} \\ \hline \hline
$(-)^\bullet$ & cosimplicial object \\ \hline
$(-)_\bullet$ & simplicial object \\ \hline
$(-)^\lw$ & functor being taken levelwise \\ \hline
$(-)^\opobj$ & corresponding object or morphism in the opposite $\infty$-category \\
\hline \hline
\end{longtable}
\egroup
\end{center}

\vspace{-20pt}

\noindent (Given a functor $\C \xra{F} \D$, we will sometimes (but not always) write $c\C \xra{F^\lw} c\D$ and $s\C \xra{F^\lw} s\D$ to denote the induced functors on $\infty$-categories of co/simplicial objects given by postcomposition with $F$ (instead of $cF$ or $sF$, resp.).)

\item

We will at times refer to various ``named'' results, both within this sequence of papers and in external citations.  For the reader's convenience, we will always refer to these both by name and by number.  We take the conventions that
\begin{itemize}
\item if the name of the result includes the \textit{type} of result (e.g.\! ``theorem'', ``lemma'', etc.) then we won't repeat it -- so for instance we'll simply refer to ``Kenny Brown's lemma (\Cref{qadjns:kenny brown})'' --, whereas
\item if the name of the result does not include its type, then we will include it -- so for instance we'll refer to ``the small object argument (\cref{small object argument})''.
\end{itemize}

\end{enumerate}

\section{Index of notation}\label{section notation index}

For the reader's convenience, in this appendix we provide an index of all (potentially not-completely-standard) mathematical symbols that we use throughout this sequence of papers.  We list them in alphabetical order to the (greatest extent possible) and indicate where they are defined or first appear (using code names in all cases, including for references in the present paper).  We generally list multi-use decorations as separate entries (e.g.\! the subscript indicating a ``named'' model $\infty$-category) so as to minimize repetition.

\begin{multicols}{2}

{\setlength{\parindent}{0pt}


\indnotn{$|{-}|$}{\Cref{sspaces:section conventions}\ref{colimit notn}, \Cref{hammocks:def geom realizn of sS-cats}}

\indnotn{$\|{-}\|$}{\Cref{sspaces:section conventions}\ref{colimit notn}}


\indnotn{$\cofibn$}{\Cref{sspaces:define model infty-category}}

\indnotn{$\fibn$}{\Cref{sspaces:define model infty-category}}

\indnotn{$\we$}{\Cref{sspaces:define model infty-category}}

\indnotn{$\laxra$}{\Cref{gr:define op/lax nat trans betw fctrs to Cati}}

\indnotn{$\oplaxra$}{\Cref{gr:define op/lax nat trans betw fctrs to Cati}}


\indnotn{$(-)_*$}{\Cref{sspaces:section conventions}\ref{ground in qcats}\ref{initial and terminal objs}}

\indnotn{$(-)_{**}$}{\Cref{hammocks:define doubly-pointed relcats}, \Cref{fundthm:define doubly-pointed model infty-diagram}}

\indnotn{$(-)_{(**)}$}{\Cref{hammocks:define maybe-pointed relcats}, \Cref{fundthm:par ast for maybe}}


\indnotn{$\square$}{\Cref{qadjns:defn various products of maps}}

\indnotn{$\extprod$}{\Cref{sspaces:compare with moerdijk}}  

\indnotn{$\tensoring$}{\Cref{sspaces:section conventions}\ref{co/tensoring}, \Cref{sspaces:section conventions}\ref{enriched and bitensored via two-variable adjunction}, \Cref{hammocks:notn for tensoring of either relcat or relcatp over relcat}, \Cref{fundthm:tensoring of Model over RelCat}}  

\indnotn{$\otimes$}{\Cref{sspaces:section conventions}\ref{multivar adjns}}  

\indnotn{$\cotensoring$}{\Cref{sspaces:section conventions}\ref{co/tensoring}, \Cref{sspaces:section conventions}\ref{enriched and bitensored via two-variable adjunction}}  


\indnotn{$(-)^\bullet$}{\Cref{sspaces:section conventions}\ref{notn possibly-omitted decorations}}

\indnotn{$(-)^{\times(\bullet+1)}$}{\Cref{hammocks:define 0th coskel and pb of simp obj}}

\indnotn{$(-)_\bullet$}{\Cref{sspaces:section conventions}\ref{notn possibly-omitted decorations}}


\indnotn{$\coprod$}{\Cref{sspaces:section conventions}\ref{pushout/pullback and co/product notation}}


\indnotn{$(-)^\opobj$}{\Cref{sspaces:section conventions}\ref{notation opobj}}


\indnotn{$\diamond$}{T.4.2.2.1}


\indnotn{$\es$}{\Cref{sspaces:section conventions}\ref{ground in qcats}\ref{conventions on infty-cats}}

\indnotn{$(-)_\es$}{\Cref{sspaces:section conventions}\ref{ground in qcats}\ref{initial and terminal objs}}


\indnotn{$=$}{\Cref{sspaces:section conventions}\ref{work invariantly with qcats}\ref{reserve equals sign}}

\indnotn{$(=)$}{\Cref{sspaces:section conventions}\ref{notation for adjunct bifunctor}}

\indnotn{$\cong$}{\Cref{sspaces:section conventions}\ref{work invariantly with qcats}\ref{reserve equals sign}}

\indnotn{$(-)^\simeq$}{\Cref{sspaces:section conventions}\ref{ground in qcats}\ref{conventions on infty-cats}}

\indnotn{$\approx$}{\Cref{sspaces:define model infty-category}}


\indnotn{$(-)_!$}{\Cref{sspaces:section conventions}\ref{left kan extn}}


\indnotn{$(-)^\flat$}{\sec T.3.1}


\indnotn{$\int$}{\Cref{sspaces:section conventions}\ref{notation for co/ends}}


\indnotn{$\loc{(-)}{(-)}$}{\cref{sspaces:subsection sspaces and model infty-cats}, \Cref{rnerves:define localization}}

\indnotn{$(-)[(-)^{-1}]$}{\cref{sspaces:subsection sspaces and model infty-cats}}


\indnotn{$(-)^\natural$}{T.3.1.1.8}


\indnotn{$\prod$}{\Cref{sspaces:section conventions}\ref{pushout/pullback and co/product notation}}


\indnotn{$(-)^\sharp$}{\cref{sspaces:section proof of crazy lemma}, \sec T.3.1}


\indnotn{$\adj$}{\Cref{sspaces:section conventions}\ref{adjunction}}


\indnotn{$(-)_{x/}$}{T.1.2.9.5}

\indnotn{$(-)_{/x}$}{T.1.2.9.4}


\indnotn{$\evendash$}{\Cref{gr:half-open interval notn}}

\indnotn{$\odddash$}{\Cref{gr:half-open interval notn}}

\indnotn{$[-,-]$}{\Cref{sspaces:section conventions}\ref{ground in qcats}\ref{conventions on infty-cats}}

\indnotn{$[-;\ldots;-]$}{\Cref{hammocks:define relative word}, \Cref{fundthm:define model word}}

\indnotn{$(- \da_n - )$}{\Cref{gr:define phpb}}

\indnotn{$(-)([1];-,-)$}{\Cref{qadjns:restricted coCart fibns over 1}}


\indnotn{$\word{3}$}{\Cref{hammocks:notn 3}}

\indnotn{$\tilde{\word{3}}$}{\Cref{fundthm:notn special 3 and 7}}

\indnotn{$\word{7}$}{\Cref{fundthm:notn special 3 and 7}}

\vspace{10pt}


\indnotn{$\Adjn$}{\Cref{sspaces:section conventions}\ref{adjunction}, \Cref{sspaces:section conventions}\ref{multivar adjns}}

\indnotn{$\Alg$}{A.4.1.1.9}

\indnotn{$\Alg^\nonu$}{\sec A.5.4.3}

\vspace{10pt}


\indnotn{$\Bar$}{A.4.4.2.7}

\indnotn{$(-)_\Bergner$}{\Cref{sspaces:section conventions}\ref{different model cats}\ref{sset-enr cat as infty-cat}}

\indnotn{$\biCartFib$}{\Cref{qadjns:notation for bicart fibns}}

\indnotn{$(-)_\BarKan$}{\Cref{rnerves:define BarKan rel str on RelCati}, \Cref{sspaces:section conventions}\ref{different model cats}\ref{relcat as infty-cat}}

\vspace{10pt}


\indnotn{$\bC$}{\Cref{sspaces:define model infty-category}}

\indnotn{$\mf{C}$}{\Cref{sspaces:section conventions}\ref{different model cats}\ref{sset-enr cat as infty-cat}}

\indnotn{$\rvec{\C}$}{\Cref{qadjns:define reedy cat}}

\indnotn{$\lvec{\C}$}{\Cref{qadjns:define reedy cat}}

\indnotn{$c(-)$}{\Cref{sspaces:section conventions}\ref{cosimp and simp objs}, \Cref{fundthm:construct cmp}}

\indnotn{$(-)^c$}{\Cref{sspaces:define co/fibrant objects}}

\indnotn{$\CAlg$}{A.2.1.3.1}

\indnotn{$\CAlg^\nonu$}{\sec A.5.4.4}

\indnotn{$(-)_\can$}{\Cref{sspaces:compare with anderson and bousfield--friedlander}}

\indnotn{$\Cat$}{\Cref{sspaces:section conventions}\ref{ground in qcats}\ref{conventions on infty-cats}}

\indnotn{$\Cati$}{\Cref{sspaces:section conventions}\ref{ground in qcats}\ref{conventions on infty-cats}}

\indnotn{$\strcat$}{\Cref{sspaces:section conventions}\ref{work invariantly with qcats}\ref{strict cats}}

\indnotn{$\strcat_{s\Set}$}{\Cref{sspaces:section conventions}\ref{different model cats}\ref{sset-enr cat as infty-cat}}

\indnotn{$(-)\dashcell$}{\Cref{sspaces:rel cell cxes}}

\indnotn{$\CartFib$}{\Cref{gr:notation for grothendieck construction}}

\indnotn{$\CartFib_\Rel$}{\Cref{qadjns:define relative co/cart fibns}}

\indnotn{$\coCartFib$}{\Cref{gr:notation for grothendieck construction}}

\indnotn{$\coCartFib_\Rel$}{\Cref{qadjns:define relative co/cart fibns}}

\indnotn{$(-)\dashcof$}{\Cref{sspaces:cofibrations have llp(rlp(I))}}

\indnotn{$\colim$}{\Cref{sspaces:section conventions}\ref{colimit notn}}

\indnotn{$\const$}{\Cref{sspaces:section conventions}\ref{constant diagram functor}, \Cref{hammocks:def constant sS-enr infty-cat}}

\indnotn{$(-)^{cf}$}{\Cref{sspaces:define co/fibrant objects}}

\indnotn{$\CSS$}{\Cref{rnerves:define CSSs}}

\indnotn{$\CSS_{\X \subset \Y}$}{\Cref{hammocks:compare with barwick's theory of enriched infty-cats}}

\indnotn{$\cyl$}{\Cref{fundthm:define cyl and path}, \Cref{fundthm:cats of cyls and paths}}

\indnotn{$\scyl$}{\Cref{fundthm:define cyl and path}}

\vspace{10pt}


\indnotn{$\bD$}{\sec T.A.2.7, \Cref{sspaces:section conventions}\ref{work invariantly with qcats}\ref{strict cats}}

\indnotn{$\Delta^{\{i_0,\ldots,i_j\}}$}{\Cref{sspaces:section conventions}\ref{simplex category}}

\indnotn{$\Delta^n$}{T.A.2.7.2}

\indnotn{$\delta^i_n$}{\Cref{sspaces:section conventions}\ref{simplex category}}

\indnotn{$\delta^n_i$}{\Cref{sspaces:section conventions}\ref{simplex category}}

\indnotn{$\partial$}{\Cref{qadjns:latching and matching cats}}

\indnotn{$\partial \Delta^n$}{\sec T.A.2.7}

\indnotn{$\diag$}{\Cref{sspaces:section conventions}\ref{diagonal map}}

\indnotn{$\disc$}{\Cref{sspaces:section conventions}\ref{ground in qcats}\ref{conventions on spaces}, \Cref{sspaces:adjunction between sspaces and ssets}}

\indnotn{$(-)_\DK$}{\Cref{hammocks:DK equivces in SsS}, \Cref{hammocks:define sspatial infty-cats}, \cref{hammocks:subsection intro to hammocks}}

\vspace{10pt}


\indnotn{$\bbE_\infty$}{\Cref{sspaces:rem motivic ghost for motivic morava E-theories}}

\indnotn{$\bbE_n$}{\Cref{sspaces:speculn derived koszul duality}}

\indnotn{$\Etwo$}{\cref{sspaces:subsection GHOsT motivation}}

\indnotn{$\ev$}{\Cref{sspaces:section conventions}\ref{evaluation notation}}

\indnotn{$\Ex$}{\Cref{sspaces:define Ex on sspaces}}

\indnotn{$\Ex^\infty$}{\Cref{sspaces:notn Ex-infty}}

\indnotn{$\Ex^n$}{\Cref{sspaces:iterate sd and Ex}}

\vspace{10pt}


\indnotn{$\bF$}{\Cref{sspaces:define model infty-category}}

\indnotn{$(-)^f$}{\Cref{sspaces:define co/fibrant objects}}

\indnotn{$\Fun$}{\Cref{sspaces:section conventions}\ref{ground in qcats}\ref{conventions on infty-cats}}

\indnotn{$\Fun^\dec$}{\Cref{fundthm:decorated variant}}

\indnotn{$\Fun(-,-)^\Model$}{\Cref{fundthm:enr and tensor model over relcat}}

\indnotn{$\Fun(-,-)^\Rel$}{\Cref{rnerves:define internal hom in relcats}}

\indnotn{$\Fun(-,-)^\bW$}{\Cref{rnerves:define internal hom in relcats}, \Cref{fundthm:enr and tensor model over relcat}}

\indnotn{$\Fun^{\colim}$}{\Cref{sspaces:section conventions}\ref{yoneda}}

\indnotn{$\Fun_\Sigma$}{\cref{sspaces:subsection GHOsT motivation}}

\indnotn{$\Fun^{\textup{surj}}$}{\Cref{hammocks:any SS pulled back from its CSS-localizn}}

\indnotn{$\Fun^{\textup{surj mono}}$}{\Cref{rnerves:define rel infty-cat}}

\indnotn{$\strfun$}{\Cref{sspaces:section conventions}\ref{work invariantly with qcats}\ref{strict cats}}

\vspace{10pt}


\indnotn{$\strgpd$}{\Cref{sspaces:section conventions}\ref{work invariantly with qcats}\ref{strict cats}}

\indnotn{$(-)^\gpd$}{\Cref{sspaces:section conventions}\ref{ground in qcats}\ref{conventions on infty-cats}}

\indnotn{$\Gr$}{\Cref{gr:define grothendieck construction}, \Cref{gr:define two-sided Gr}}

\indnotn{$\Gr_\Rel$}{\Cref{qadjns:define relative co/cart fibns}}

\indnotn{$\Grop$}{\Cref{gr:define grothendieck construction}}

\indnotn{$\Grop_\Rel$}{\Cref{qadjns:define relative co/cart fibns}}

\indnotn{$\Grp$}{\Cref{sspaces:compare with anderson and bousfield--friedlander}}

\vspace{10pt}


\indnotn{$h^i_n$}{\Cref{fundthm:left homotopy corepresentation induces left homotopy}}

\indnotn{$\ho$}{\Cref{sspaces:section conventions}\ref{ground in qcats}\ref{conventions on infty-cats}}

\indnotn{$\hom$}{\Cref{sspaces:section conventions}\ref{twisted arrow gives hom bifunctor}}

\indnotn{$\hom^\dec$}{\Cref{fundthm:decorated variant}}

\indnotn{$\enrhom$}{\Cref{sspaces:section conventions}\ref{twisted arrow gives hom bifunctor}, \Cref{sspaces:section conventions}\ref{enriched and bitensored via two-variable adjunction}, \Cref{hammocks:define space of objects and hom-sspaces}}

\indnotn{$\enrhom_l$}{\Cref{sspaces:section conventions}\ref{multivar adjns}}

\indnotn{$\enrhom_r$}{\Cref{sspaces:section conventions}\ref{multivar adjns}}

\indnotn{$\enrhom^\square_l$}{\Cref{qadjns:defn various products of maps}}

\indnotn{$\enrhom^\square_r$}{\Cref{qadjns:defn various products of maps}}

\indnotn{$\hom^\lsim$}{\Cref{fundthm:define spaces of left and right htpy classes of maps}}

\indnotn{$\hom^\rsim$}{\Cref{fundthm:define spaces of left and right htpy classes of maps}}

\vspace{10pt}


\indnotn{$\II$}{\Cref{hammocks:define relative word}}

\indnotn{$(-)\dashinj$}{\Cref{sspaces:injectives have rlp(I)}}

\indnotn{$(-)_\injective$}{\Cref{qadjns:defn proj and inj model strs}, \Cref{sspaces:section conventions}\ref{ho-co-lims}}

\vspace{10pt}


\indnotn{$(-)_\Joyal$}{\Cref{sspaces:section conventions}\ref{different model cats}\ref{qcat as infty-cat}, \Cref{sspaces:joyal model str on sspaces}}

\vspace{10pt}


\indnotn{$(-)_\KQ$}{\Cref{sspaces:define kan--quillen model structure on ssets}, \Cref{sspaces:define kan--quillen model structure on sspaces}}

\indnotn{$(-)_{\KQ_{\textup{medium}}}$}{\Cref{sspaces:other KQ model structures}}

\indnotn{$(-)_{\KQ_{\textup{strong}}}$}{\Cref{sspaces:other KQ model structures}}

\indnotn{$(-)_{\KQ_{\textup{weak}}}$}{\Cref{sspaces:other KQ model structures}}

\vspace{10pt}


\indnotn{$\bbL$}{\Cref{qadjns:defn der functors of q adjn}, \Cref{qadjns:two-var adjunction thm}}

\indnotn{$\locL$}{\Cref{rnerves:define localization}}

\indnotn{$\leftloc$}{\Cref{sspaces:section conventions}\ref{free vs left vs right localizations}}

\indnotn{$\Latch_{(-)}(-)$}{\Cref{sspaces:section conventions}\ref{generalized matching and latching}, \Cref{qadjns:notation latching and matching objects}}

\indnotn{$(-)_\leftloc$}{\Cref{sspaces:left localizations give model structures}}

\indnotn{$\leftloc_\strcatsup$}{\Cref{sspaces:section conventions}\ref{work invariantly with qcats}\ref{strict cats}}

\indnotn{$\leftloc_{\CartFib(\C)}$}{\Cref{gr:notation for grothendieck construction}}

\indnotn{$\leftloc_{\coCartFib(\C)}$}{\Cref{gr:notation for grothendieck construction}}

\indnotn{$\leftloc_\CSS$}{\Cref{rnerves:define CSSs}}

\indnotn{$\leftloc_{\L(\C)}$}{\Cref{gr:notation for grothendieck construction}}

\indnotn{$\leftloc_{\LFib(\C)}$}{\Cref{gr:notation for grothendieck construction}}

\indnotn{$\leftloc_{\R(\C)}$}{\Cref{gr:notation for grothendieck construction}}

\indnotn{$\leftloc_{\RFib(\C)}$}{\Cref{gr:notation for grothendieck construction}}

\indnotn{$\leftloc_\SS$}{\Cref{hammocks:define segal spaces}}

\indnotn{$\LAdjt$}{\Cref{sspaces:section conventions}\ref{adjunction}}

\indnotn{$\Lambda^n_i$}{T.A.2.7.3}

\indnotn{$\Lax$}{\Cref{gr:define op/lax overcat}}

\indnotn{$\Lax(-)^{\colim}$}{\Cref{gr:notn subcat of lax with colim}}

\indnotn{$\LFib$}{\Cref{gr:notation for grothendieck construction}}

\indnotn{$\ham$}{\Cref{hammocks:define hammock localizn}}

\indnotn{$\hamd$}{\Cref{sspaces:section conventions}\ref{different model cats}\ref{relcat as infty-cat}}

\indnotn{$\preham$}{\Cref{hammocks:define hammock localizn}}

\indnotn{$\lim$}{\Cref{sspaces:section conventions}\ref{colimit notn}}

\indnotn{$\llp$}{\Cref{sspaces:section conventions}\ref{notn lifting properties}}

\indnotn{$\LMod$}{A.4.2.1.13}

\indnotn{$\LQAdjt$}{\Cref{qadjns:define infty-cats QAdjn and LQAdjt and RQAdjt}}

\indnotn{$(-)^\lw$}{\Cref{sspaces:section conventions}\ref{notn possibly-omitted decorations}}

\vspace{10pt}


\indnotn{$\Match_{(-)}(-)$}{\Cref{sspaces:section conventions}\ref{generalized matching and latching}, \Cref{qadjns:notation latching and matching objects}}

\indnotn{$\mxm$}{\Cref{rnerves:pre-nerve as sCSS}}

\indnotn{$\word{m}(x,-)$}{\Cref{hammocks:notation for half-doubly-pointed words}}

\indnotn{$\word{m}(x,y)$}{\Cref{hammocks:define zigzags}, \Cref{fundthm:define infty-cat of model zigzags}}

\indnotn{$\word{m}(-,y)$}{\Cref{hammocks:notation for half-doubly-pointed words}}

\indnotn{$\max$}{\Cref{rnerves:trivial relative infty-category structures}}

\indnotn{$\min$}{\Cref{rnerves:trivial relative infty-category structures}}

\indnotn{$\Model$}{\Cref{fundthm:define model infty-diagrams}}

\indnotn{$\Modeli$}{\Cref{fundthm:define model infty-diagrams}}

\indnotn{$(-)_\Moer$}{\Cref{sspaces:compare with moerdijk}}

\vspace{10pt}


\indnotn{$\Nerve$}{\Cref{sspaces:section conventions}\ref{work invariantly with qcats}\ref{strict cats}}

\indnotn{$\Nervei$}{\Cref{rnerves:define CSSs}}

\indnotn{$\Nerve^+$}{\Cref{rnerves:natural mono}}

\indnotn{$\Nervehc$}{\Cref{sspaces:section conventions}\ref{different model cats}\ref{sset-enr cat as infty-cat}}

\indnotn{$\NerveRezk$}{\Cref{rnerves:infty-catl rezk nerve agrees with 1-catl rezk nerve}}

\indnotn{$\NerveRezki$}{\Cref{rnerves:define infty-categorical rezk pre-nerve and rezk nerve}}

\indnotn{$[n]$}{\sec T.A.2.7}

\indnotn{$[n]_\bW$}{\Cref{rnerves:trivial relative infty-category structures}}

\vspace{10pt}


\indnotn{$\Omega^\infty$}{\Cref{sspaces:KQ model structure on sspectra}}

\indnotn{$(-)^{op}$}{\Cref{sspaces:section conventions}\ref{ground in qcats}\ref{conventions on infty-cats}}

\indnotn{$\opLax$}{\Cref{gr:define op/lax overcat}}

\vspace{10pt}


\indnotn{$\P$}{\Cref{sspaces:section conventions}\ref{yoneda}}

\indnotn{$\pth$}{\Cref{fundthm:define cyl and path}, \Cref{fundthm:cats of cyls and paths}}

\indnotn{$\spth$}{\Cref{fundthm:define cyl and path}}

\indnotn{$\Pi_1$}{\Cref{sspaces:compare with anderson and bousfield--friedlander}}

\indnotn{$\Pi_{\geq 1}$}{\Cref{sspaces:compare with anderson and bousfield--friedlander}}

\indnotn{$\pi_0$}{\Cref{sspaces:section conventions}\ref{ground in qcats}\ref{conventions on spaces}, \Cref{sspaces:adjunction between sspaces and ssets}}

\indnotn{$(-)_{\pi_0}$}{\Cref{sspaces:pi-0 model structure on spaces}}

\indnotn{$\pi_{\geq 1}$}{\Cref{sspaces:section conventions}\ref{ground in qcats}\ref{conventions on spaces}}

\indnotn{$\pi_{\geq n}$}{\Cref{sspaces:try to use n-cotruncation of pointed simply connected spaces to get a model structure}}

\indnotn{$\pr_{\Gr(F)}$}{\Cref{gr:define grothendieck construction}}

\indnotn{$\pr_{\Grop(G)}$}{\Cref{gr:define grothendieck construction}}

\indnotn{$\preNerveRezki$}{\Cref{rnerves:define infty-categorical rezk pre-nerve and rezk nerve}}

\indnotn{$(-)_\projective$}{\Cref{qadjns:defn proj and inj model strs}, \Cref{sspaces:section conventions}\ref{ho-co-lims}}

\indnotn{$\PS$}{{\cref{sspaces:subsection GHOsT motivation}}}

\indnotn{$\pt$}{\Cref{sspaces:section conventions}\ref{ground in qcats}\ref{conventions on infty-cats}}

\vspace{10pt}


\indnotn{$\bbQ \unit_\V$}{\Cref{qadjns:defn of monoidal model infty-cat}}

\indnotn{$\QAdjn$}{\Cref{qadjns:defn various products of maps}}

\vspace{10pt}


\indnotn{$\bbR$}{\Cref{qadjns:defn der functors of q adjn}, \Cref{qadjns:two-var adjunction thm}}

\indnotn{$\rightloc$}{\Cref{sspaces:section conventions}\ref{free vs left vs right localizations}}

\indnotn{$(-)_\rightloc$}{\Cref{sspaces:right localizations give model structures}}

\indnotn{$\RAdjt$}{\Cref{sspaces:section conventions}\ref{adjunction}}

\indnotn{$(-)_\Reedy$}{\Cref{qadjns:define reedy model str}, \Cref{sspaces:section conventions}\ref{ho-co-lims}}

\indnotn{$\RelCat$}{\Cref{rnerves:define rel infty-cat}}

\indnotn{$\RelCati$}{\Cref{rnerves:define rel infty-cat}}

\indnotn{$\strrelcat$}{\Cref{rnerves:weaker defn of rel infty-cat}, \Cref{sspaces:section conventions}\ref{work invariantly with qcats}\ref{strict cats}, \Cref{sspaces:section conventions}\ref{different model cats}\ref{relcat as infty-cat}}

\indnotn{$(-)_\res$}{{\cref{sspaces:subsection GHOsT motivation}}}

\indnotn{$(-)_\Rezk$}{\Cref{sspaces:section conventions}\ref{different model cats}\ref{CSS as infty-cat}, \Cref{rnerves:define rezk rel str on ssspaces}}

\indnotn{$\RFib$}{\Cref{gr:notation for grothendieck construction}}

\indnotn{$\rlp$}{\Cref{sspaces:section conventions}\ref{notn lifting properties}}

\indnotn{$\rlp'$}{\Cref{sspaces:pi-0 model structure on spaces is almost cofgen}}

\indnotn{$\RMod$}{A.4.2.1.36}

\indnotn{$\RQAdjt$}{\Cref{qadjns:define infty-cats QAdjn and LQAdjt and RQAdjt}}

\vspace{10pt}


\indnotn{$\S$}{\Cref{sspaces:section conventions}\ref{ground in qcats}\ref{conventions on spaces}}

\indnotn{$\S^{\geq n}$}{\Cref{sspaces:section conventions}\ref{ground in qcats}\ref{conventions on spaces}}

\indnotn{$\S^{\leq n}$}{\Cref{sspaces:section conventions}\ref{ground in qcats}\ref{conventions on spaces}}

\indnotn{$\S^\fin$}{\Cref{sspaces:section conventions}\ref{ground in qcats}\ref{conventions on spaces}}

\indnotn{$\SS$}{\Cref{hammocks:define segal spaces}}

\indnotn{$\SS_{\X \subset \Y}$}{\Cref{hammocks:compare with barwick's theory of enriched infty-cats}}

\indnotn{$\SsS$}{\Cref{hammocks:define Segal sspaces}}

\indnotn{$s$}{\Cref{sspaces:section conventions}\ref{source and target}, \Cref{gr:define phpb}, \Cref{hammocks:define doubly-pointed relcats}, \Cref{fundthm:define doubly-pointed model infty-diagram}}

\indnotn{$s_{\Lax(\C)}$}{\Cref{gr:define op/lax overcat}}

\indnotn{$s_{\opLax(\C)}$}{\Cref{gr:define op/lax overcat}}

\indnotn{$s(-)$}{\Cref{sspaces:section conventions}\ref{cosimp and simp objs}}

\indnotn{$\sd$}{\Cref{sspaces:define sd on sspaces}}

\indnotn{$\sd^n$}{\Cref{sspaces:iterate sd and Ex}}

\indnotn{$\Set$}{\Cref{sspaces:section conventions}\ref{ground in qcats}\ref{conventions on spaces}}

\indnotn{$\sigma^i_n$}{\Cref{sspaces:section conventions}\ref{simplex category}}

\indnotn{$\sigma^n_i$}{\Cref{sspaces:section conventions}\ref{simplex category}}

\indnotn{$\Sp$}{\Cref{sspaces:try to put model structure on bounded spectra}}

\indnotn{$\Sp^{\geq n}$}{\Cref{sspaces:try to put model structure on bounded spectra}}

\indnotn{$\Sp^{\leq n}$}{\Cref{sspaces:try to put model structure on bounded spectra}}

\indnotn{$\Sp^{[m,n]}$}{\Cref{sspaces:try to put model structure on bounded spectra}}

\indnotn{$\spat$}{\Cref{hammocks:spatialization}}

\indnotn{$\srep$}{\Cref{gr:defn simp repl}}

\indnotn{$s\Set^+$}{T.3.1.0.1}

\vspace{10pt}


\indnotn{$t$}{\Cref{sspaces:section conventions}\ref{source and target}, \Cref{gr:define phpb}, \Cref{hammocks:define doubly-pointed relcats}, \Cref{fundthm:define doubly-pointed model infty-diagram}}

\indnotn{$\tau_{\geq n}$}{\Cref{sspaces:section conventions}\ref{ground in qcats}\ref{conventions on spaces}, \Cref{sspaces:try to put model structure on bounded spectra}}

\indnotn{$\tau_{\leq n}$}{\Cref{sspaces:section conventions}\ref{ground in qcats}\ref{conventions on spaces}, \Cref{sspaces:try to put model structure on bounded spectra}}

\indnotn{$(-)_{\tau_{\geq n}}$}{\Cref{sspaces:try to use n-cotruncation of pointed simply connected spaces to get a model structure}}

\indnotn{$(-)_{\tau_{\leq n}}$}{\Cref{sspaces:n-truncation model structure on spaces}}

\indnotn{$(-)_\Thomason$}{\Cref{gr:define Thomason model str}, \Cref{gr:different Thomasons}}

\indnotn{$\Top$}{\Cref{sspaces:section conventions}\ref{ground in qcats}\ref{conventions on spaces}}

\indnotn{$(-)_\triv$}{\Cref{sspaces:trivial model structure}}

\indnotn{$\TwAr$}{A.5.2.1.1}

\vspace{10pt}


\indnotn{$\forget$}{\Cref{sspaces:section conventions}\ref{free vs left vs right localizations}}

\indnotn{$\forget_{\geq n}$}{\Cref{sspaces:section conventions}\ref{ground in qcats}\ref{conventions on spaces}}

\indnotn{$\forget_{\leq n}$}{\Cref{sspaces:section conventions}\ref{ground in qcats}\ref{conventions on spaces}}

\indnotn{$\forget_\C$}{\Cref{gr:notation UC and UC-dagger}}

\indnotn{$\forget_\C^\dagger$}{\Cref{gr:notation UC and UC-dagger}}

\indnotn{$\forget_\Cat$}{\Cref{sspaces:section conventions}\ref{ground in qcats}\ref{conventions on infty-cats}}

\indnotn{$\forget_\CatsS$}{\Cref{hammocks:define sspatial infty-cats}}

\indnotn{$\forget_\strcatsup$}{\Cref{sspaces:section conventions}\ref{work invariantly with qcats}\ref{strict cats}}

\indnotn{$\forget_{\CartFib(\C)}$}{\Cref{gr:notation for grothendieck construction}}

\indnotn{$\forget_{\coCartFib(\C)}$}{\Cref{gr:notation for grothendieck construction}}

\indnotn{$\forget_\CSS$}{\Cref{rnerves:define CSSs}}

\indnotn{$\forget_{\L(\C)}$}{\Cref{gr:notation for grothendieck construction}}

\indnotn{$\forget_{\LFib(\C)}$}{\Cref{gr:notation for grothendieck construction}}

\indnotn{$\forget_{\R(\C)}$}{\Cref{gr:notation for grothendieck construction}}

\indnotn{$\forget_\Rel$}{\Cref{rnerves:trivial relative infty-category structures}}

\indnotn{$\forget_{\RFib(\C)}$}{\Cref{gr:notation for grothendieck construction}}

\indnotn{$\forget_\S$}{\Cref{sspaces:section conventions}\ref{ground in qcats}\ref{conventions on infty-cats}}

\indnotn{$\forget_\SS$}{\Cref{hammocks:define segal spaces}}

\vspace{10pt}


\indnotn{$\bW$}{\Cref{sspaces:define model infty-category}, \Cref{rnerves:define rel infty-cat}}

\indnotn{$\bWgr$}{\cref{sspaces:subsection sspaces and model infty-cats}}

\indnotn{$\bW_{\dfibn y}$}{\Cref{fundthm:notn for weak equivces cofibrationing from x and-or fibrationing to y}}

\indnotn{$\bW_\he$}{\Cref{sspaces:model infty-cat from simp model cat}}

\indnotn{$\bW_\qi$}{\cref{fundthm:subsection intro to fundthm}}

\indnotn{$\bW_\whe$}{\Cref{sspaces:section conventions}\ref{ground in qcats}\ref{conventions on spaces}}

\indnotn{$\bW_{x \dcofibn}$}{\Cref{fundthm:notn for weak equivces cofibrationing from x and-or fibrationing to y}}

\indnotn{$\bW_{x \dcofibn \dfibn y}$}{\Cref{fundthm:notn for weak equivces cofibrationing from x and-or fibrationing to y}}

\indnotn{$\ol{W}$}{\Cref{sspaces:rem ceg-rem}}

\vspace{10pt}


\indnotn{$\chi^{\C_\bullet}_{x_0,\ldots,x_n}$}{\Cref{hammocks:define space of objects and hom-sspaces}}

\vspace{10pt}


\indnotn{$\Yo$}{\Cref{sspaces:section conventions}\ref{yoneda}}

\vspace{10pt}


\indnotn{$\Z$}{\Cref{hammocks:define relative word}}

\indnotn{$\word{z}_n$}{\Cref{gr:define walking zigzag categories}}

}

\end{multicols}

\bibliographystyle{amsalpha}
\bibliography{sspaces}{}

\end{document}